\documentclass[12pt, reqno, a4paper,oneside]{amsart}

\usepackage[paper=a4paper, top=3.5cm,bottom=3.0cm,left=3.0cm,right=3.0cm]{geometry}
\usepackage{amsfonts,amsmath,amssymb,amsbsy,amsthm,amsmath}
\usepackage{graphicx, epsfig, psfrag}
\usepackage{float}
\usepackage{stmaryrd}
\usepackage{bm}
\usepackage{color}
\usepackage{mathrsfs}
\usepackage{longtable}
\usepackage{enumerate}
\usepackage[normalem]{ulem}
\usepackage{placeins}
\usepackage{environments}
\numberwithin{theorem}{section}
\numberwithin{equation}{section}
\usepackage{epstopdf}
\usepackage{subfigure}

\newcommand{\Rcore}{{R_{\rm def}}}
\newcommand{\burg}{{\sf b}}
\newcommand\up{u_0}

\newcommand{\Adm}{\mathscr{A}}

\newcommand{\Del}{\widetilde{D}}

\newcommand{\nuc}{{\rm nuc}}
\newcommand{\TFW}{{\rm TFW}}
\renewcommand\L{\Lambda}

\newcommand\Fs{\mathcal{F}}
\newcommand\Us{\dot{\mathscr{W}}}
\newcommand\Usz{\Us^{\rm c}}
\newcommand\UsH{\Us^{1,2}}
\newcommand\UsHd{\Us^{-1,2}}
\newcommand\E{\mathscr{E}}

\renewcommand\b{\big}
\newcommand\B{\Big}

\newcommand{\<}{\langle}
\renewcommand{\>}{\rangle}
\newcommand\sep{\,|\,}
\newcommand\bsep{\,\b|\,}

\newcommand\dr{\,{\rm d}r}
\newcommand\dt{\,{\rm d}t}
\newcommand\ds{\,{\rm d}s}
\newcommand\dd{{\rm d}}

\newcommand\R{\mathbb{R}}
\newcommand\N{\mathbb{N}}
\newcommand\Z{\mathbb{Z}}

\newcommand{\wt}[1]{\widetilde#1}
\newcommand{\Dom}{\textnormal{Dom}}
\newcommand\vsig{\varsigma}

\newcommand\ua{\bar{u}}

\newcommand\lin{{\rm lin}}
\newcommand\ulin{u^\lin}

\newcommand\bfO{{\bm 0}}

\newcommand\D{\nabla}
\newcommand\del{\delta}
\newcommand\ddel{\delta^2}

\newcommand\Br{\mathcal{B}}
\newcommand\Gr{\mathcal{G}}

\newcommand\Ftd{\mathcal{F}_{\rm d}}
\newcommand\Ft{\mathcal{F}}

\newcommand{\smfrac}[2]{{\textstyle \frac{#1}{#2}}}
\newcommand\mA{{\sf A}}

\definecolor{cocol}{rgb}{0.7, 0, 0}
\definecolor{hccol}{rgb}{0, 0, 0.7}
\definecolor{fncol}{rgb}{0, 0.5, 0}
\definecolor{alertcol}{rgb}{0.6, 0.0, 0.8}

\newcommand{\hc}[1]{{\color{black} #1}}

\newcommand{\fn}[1]{{\color{black} #1}}

\begin{document}

\title[Geometry Equilibration of Crystalline Defects]{Geometry Equilibration of Crystalline Defects
	in quantum and atomistic descriptions}

\author{Huajie Chen}
\address{School of Mathematical Sciences, Beijing Normal University, Beijing 100190, China}
\email{chen.huajie@bnu.edu.cn}

\author{Faizan Q. Nazar}
\address{CNRS and CEREMADE (UMR CNRS 7534), University of Paris-Dauphine,
	Place de Lattre de Tassigny, 75775 Paris Cedex 16, France}
\email{nazar@ceremade.dauphine.fr}

\author{Christoph Ortner}
\address{Mathematics Institute, Zeeman Building, University of Warwick, Coventry CV4 7AL, UK}
\email{christoph.ortner@warwick.ac.uk}


\subjclass[2000]{65L20, 65L70, 70C20, 74G40, 74G65}

\keywords{Crystal lattices; defects; dislocations, regularity.}

\begin{abstract}
	We develop a rigorous framework for modelling the geometry equilibration of
	crystalline defects. We formulate the equilibration of crystal defects as
	a variational problem on a discrete energy space and establish
	qualitatively sharp far-field decay estimates for the equilibrium configuration.

	This work extends \cite{2013-defects-v3} by admitting infinite-range
	interaction which in particular includes some quantum chemistry based
	interatomic interactions.
\end{abstract}

\maketitle

\section{Introduction}
\label{sec:intro}

\def\Hw{\mathscr{L}}
\def\Wdecay{\mathfrak{W}}
\def\wf{\mathfrak{w}}
\def\n{\mathfrak{n}}
\def\hom{{\rm h}}

The study of crystalline defects is one of the primary \hc{tasks} in material modelling.
Even small defect concentrations can have a major influence on
key physical and chemical properties of materials. Two of the most common and
important classes are point defects and dislocations.

In computational simulations of defects, a small finite domain is employed,
while the far-field behaviour is described via an artificial boundary condition.
Different approaches (see, e.g. \cite[Ch. 6]{HandbookMaterialsModelling})
include cluster calculations, where a cluster containing the defect is defined
with clamped boundary conditions; supercell methods, where the system is
confined to a large box with periodic boundary conditions; and embedded cluster
methods, which attempt to remedy some of the deficiencies of the other two
methods by embedding the cluster in a simpler representation of the surrounding
lattice (for example, if the cluster is described quantum mechanically, the
embedding region may use interatomic potentials; if the cluster is simulated by
interatomic potentials, the embedding region may use continuum models).

The literature devoted to assessing the accuracy and
in particular the cell size effects of such simulations is relatively sparse;
see e.g. \cite{Balluffi,CaiBulChaLiYip2003,Hine2009,MakovPayne1995}
and references therein for a representative sample. In particular, we refer to
\cite{2013-defects-v3} for an extensive study of this problem, which develops a
rigorous framework within which the accuracy of different boundary conditions
can be precisely assessed.
A key restriction of \cite{2013-defects-v3} is the assumption
that each atom only interacts with a finite range neighbourhood. This assumption
is in particular not satisfied for electronic structure models. The aim of the
present work is to extend the rigorous analysis to such cases where interaction
has infinite-range, though we will require some control on the rate of decay
of interaction at infinity. Important models satisfying our assumptions
include Lennard--Jones, Thomas--Fermi--von Weizs\"{a}cker and tight binding.

In the mathematics literature, considerable progress has been made on studying
electronic ground states corresponding to local defects in crystals e.g.
\cite{CancesLeBris:preprint,CancesDeleurenceLewin:2008,CattoLeBrisLions}.
Much less is understood
about the related geometry optimisation (or, lattice relaxation) problem,
\cite{HudsonOrtner:disloc,2013-defects-v3,2016-multipt} in particular the
coupling betweeen electronic and geometry relaxation is essentially open. The
present work addresses this coupling and establishes some key results: we
give general conditions under which the \fn{lattice} relaxation problem can be
formulated as a variational problem on a Hilbert space, and we establish
sharp rates of decay of the discrete equilibrium configurations.

The main results of this paper are presented in Sections \ref{sec:pointD} and
\ref{sec:dislocation}, including the formulation of equilibration as a
variational problem and the generic decay estimate:
\begin{eqnarray}\label{result_intro}
|D\bar{u}(\ell)| \leq C(1+|\ell|)^{-d}
\log^t(2+|\ell|),
\end{eqnarray}
where $\bar{u}$ is the equilibrium
corrector displacement,  $Du(\ell)$ is the finite-difference stencil defined in Section
\ref{sec:intro}, $d$ is the dimension, $t=0$ for point defects and $t=1$ for
dislocations.

The results have a range of consequences, for example:
(1) present a mathematical model for a defect in an atomistic description of a crystalline
solid, which is an ideal benchmark problem for computational multi-scale methods,
but can also be used as a building block for the analytical coarse-graining of material models;
(2) provide a foundation for the analysis of boundary conditions for defect core
simulations;
and (3) allow us to perform rigorous error estimates for QM/MM coupling methods \cite{chen15b,csanyi05}.
We remark that the QM/MM coupling constructions and analysis
described in \cite{chen15b} can be applied verbatim to electronic structure satisfying
the locality properties in Section \ref{sec:siteE},
such as tight binding and TFW (see Section \ref{sec:examples}).

\subsection{Outline}

In Section \ref{sec:generalities} we formulate lattice relaxation as a variational problem:
we specify the reference configuration of a defective system
and introduce the displacement and site energy formulations,
construct suitable discrete function spaces that will be used in our analysis; and define the
energy difference functional on which the variational formulation is based.
In Section \ref{sec:results} we state our main results on the decay of discrete
equilibrium displacement fields for point defects and dislocations.
In Section \ref{sec:examples}, we show that our \fn{lattice} relaxation framework can be applied to
several practical models, that are not covered by previous works.
In Section \ref{sec:conclusions}, we make concluding remarks and discuss future perspectives.
All proofs and technical details are gathered in the appendixes.

\subsection{Notation}
\label{sec:notations}

We define the following spaces with dimension $\fn{d \in \mathbb{N}}$ and index $\fn{k > 0}$
\begin{align}
\Hw_{k,d}:=\Big\{ &\fn{\wf: [0,\infty) \to (0,\infty), ~\| \wf \|_{\Hw_{k,d}} := \| \wf \|_{L^{\infty}([0,\infty))} +} \int_0^{\infty} r^{k+d-1} \wf(r)\dr <\infty,
\nonumber
\\
& \qquad \qquad \qquad \fn{\wf} \text{ is a monotonically decreasing function} \Big\} \quad{\rm and}
\label{space:Hwk}
\\
\Hw_{k,d}^{\log}:=\Big\{ &\wf\in \Hw_{k,d},
~\| \wf \|_{\Hw_{k,d}^{\log}} := \fn{\| \wf \|_{L^{\infty}([0,\infty))} +} \int_0^{\infty} r^{k+d-1} \log^{2}(1+r) \wf(x)\dr <\infty  \Big\}.
\label{space:Hwk_dislocation}
\end{align}
For example, we have
$(1+x)^{-k-d-\varepsilon}\in\Hw_{k,d} \cap \Hw_{k,d}^{\log}$ for all $\varepsilon>0$;
and $e^{-\alpha x}\in\Hw_{k,d} \cap \Hw_{k,d}^{\log}$ for all $\alpha>0$ and index $k > 0$.
Whenever it is clear, we will omit the dimension $d$ and write $\Hw_{k}=\Hw_{k,d}$, $\Hw_{k}^{\log}=\Hw_{k,d}^{\log}$.

The symbol $\langle \cdot, \cdot \rangle$ denotes an abstract duality pairing between a Banach space and its dual.
The symbol $|\cdot|$ normally denotes the Euclidean or Frobenius norm, while $\| \cdot \| $ denotes an operator norm. \fn{For a countable set $I$ and $p \in [1,\infty]$, we use $|\cdot|_{\hc{l}^{p}(I)}$ to denote the $p$-norm over $I$. Additionally, for any finite set $A$, we use $|A|$ to denote its cardinality, though this will be clear from context.}

For a differentiable function $f$, $\nabla f$ denotes the Jacobi matrix and
$\nabla_r f = \nabla f \cdot r$ defines the directional derivative. For \fn{an infinite} \hc{dimensional Banach space $X$ and} $E\in C^2(X)$, the first and second variations are denoted by
$\<\delta E(u), v\>$ and $\<\delta^2 E(u) w, v\>$ for $u,v,w\in X$.
For higher variations of $E\in C^j(X)$, we will use the notation
$\<\delta^j E(u),{\pmb v}\>$ with ${\pmb v}=(v_1,\cdots,v_j)$. \fn{Additionally, for $x,x' \in \R^d$, we use $\delta_{xx'}$ to denote the Kronecker delta function.}

Throughout the article we use the following notation, that for $b \in A$, we denote ${A\setminus b = A \setminus\{b\}}$ and $A-b = \{a-b \sep a\in A\setminus b \}$. 

The symbol $C$ denotes a generic positive constant that may change from one line to the next. When estimating rates of decay or convergence, $C$
will always remain independent of the system size, lattice position or the choice of test functions. The dependence of $C$ on model parameters will normally be clear from the context
or stated explicitly.

\subsection{List of assumptions and symbols}

\def\Lhom{\Lambda^{\rm h}}
\def\Ihom{I^{\rm h}}
\def\Iback{I^{\rm d}}
\def\asRC{\textnormal{\bf (RC)}}
\def\asSE{\textnormal{\bf (S)}}
\def\asSER{\textnormal{\bf (S.R)}}
\def\asSEL{\textnormal{\bf (S.L)}}
\def\asSEH{\textnormal{\bf (S.H)}}
\def\asSEP{\textnormal{\bf (S.P)}}
\def\asSEI{\textnormal{\bf (S.I)}}
\def\asSEPS{\textnormal{\bf (S.PS)}}
\def\asLS{\textnormal{\bf (LS)}}
\def\dl{\mathsf{d}_{\rm l}}
\def\ds{\mathsf{d}_{\rm s}}
\def\Winf{\mathscr{H}^{1,\infty}}
\def\DasSEL{{\bf (DSE.L)}}
\def\idx{{\rm idx}}
\def\asH{\textnormal{\bf (H)}}
\def\asP{\textnormal{\bf (P)}}
\def\asD{\textnormal{\bf (D)}}

Our analysis requires a number of assumptions on the model or the
underlying atomistic geometry.  We list these with page references and brief summaries.
We also list the symbols that are most frequently used in the paper.

\begin{center}
	\begin{minipage}{0.9\textwidth}
		\begin{tabular}{rrl}
			\asRC & p. \pageref{as:asRC} & assumptions for reference configurations \\
			\asSE & p. \pageref{as:SE:pr} &  assumptions for site strain potential	\\
			\asH & p. \pageref{as:H} & assumptions for homogeneous lattices  \\
			\asLS & p. \pageref{as:LS} & assumptions for lattice stability \\
			\asP & p. \pageref{as:P} & assumptions for point defects \\
			\asD & p. \pageref{as:D} & assumptions for dislocations \\[1ex]
			$\Hw_{k,d}, \Hw_{k,d}^{\log}$ & p. \pageref{space:Hwk} &  spaces for weight functions
			\\
			$\L, \Lhom, \L^{\hom}_{*}, \L_0$ & p. \pageref{ref-config}, \pageref{ref-config-L0} &  reference configuration \\
			$D, \widetilde{D}$& p. \pageref{finite-difference}, \pageref{eq:defn_Dprime} & finite difference and elastic strain (permuted finite-diff.)  \\
			$\|D\cdot\|_{\hc{l^2_{\mathcal{N}}}}$ & p. \pageref{eq: nn norm} & norm over nearest neighbour bonds \\
			$\UsH, \Usz$ & p. \pageref{space:UsH} & energy space, compact displacements  \\
			$\Adm^0, \Adm$ & p. \pageref{space:adm} 
			&  spaces of admissible configurations \\
			$\mathscr{H}, \mathscr{H}^{\rm c}, \hc{\mathscr{H}_{\frak{m},\lambda}}$ & p. \pageref{space:HL} & spaces for displacements  \\
			$V_{\ell}, \hc{V^{\hom}}, \Phi_{\ell}, \fn{\Phi^{\hom}_{\ell}}$ & p. \pageref{site-strain-potential}, \pageref{siteE-sietV} & site strain potential and site energy \\
			$V_{\hc{\ell},{\bm \rho}}$ & p. \pageref{eq:dV} & derivative of site potential \\
			$\| D \cdot \|_{\hc{l^p_{\wf,k}}}$   & p. \pageref{eq:norm-def-1} & norm with weighted stencils \\
			$\E, \E^{\hom}$ & p. \pageref{eq:Ediff}, \pageref{energy-diff-hom} & energy difference functionals \\
			$I^{\hom}_{1}, I^{\hom}_{2}, I^{\rm d}$& p. \pageref{sec:proof_interpolation} & interpolation operators \\
			$S, S_0, S^*$& p. \pageref{S:op} & slip operators \\
			$e$& p. \pageref{eq:defn_elastic_strain} & elastic strain, predictor configuration
		\end{tabular}
	\end{minipage}
\end{center}

\section{Preliminaries}
\label{sec:generalities}

\subsection{Reference and deformed configuration}
\label{sec:refConfig}

We consider a single defect embedded in a homogeneous crystalline bulk.
Let $d \in\{2,3\}$ denote the dimension of the reference configuration, $A\in\R^{d \times d}$ be a nonsingular matrix, 
then we define a homogeneous crystal reference configuration by the Bravais lattice $\Lhom:=A\Z^{d}$.
\label{ref-config}
It is  convenient to define $\Lhom_{*} = \Lhom \setminus 0 = \Lhom - 0$, 
which satisfies $\Lhom_{*} = \Lhom - \ell$ for all $\ell \in \Lhom$. 
\fn{Also, for $R > 0$, $B_R \subset \R^d$ denotes the ball of radius $R$ centred at the origin.}

The reference configuration for the defect is a set $\Lambda \subset \R^{d}$ satisfying
\begin{flushleft}\label{as:asRC}
	\asRC \quad
	$\exists ~\Rcore>0$, such that
	$\Lambda\backslash B_{\Rcore} = \Lhom\backslash B_{\Rcore}$ and
	$\L \cap B_{\Rcore}$ is finite.
\end{flushleft}
Clearly, \fn{given $\Lhom$,} choosing $\L = \Lhom$ satisfies \asRC.

For any $\ell\in\L$, we define the set of its \fn{neighbouring sites:}
\begin{align}\label{def1:Nl}
\mathcal{N}(\ell) :=& \left\{ \, m \in \L \setminus \ell \, \Big| \, \exists \,
\fn{a} \in \mathbb{R}^{d} \text{ s.t. }
|a - \ell| = |a - m| \leq |\fn{a} - k| \quad \forall \, k \in \L \, \right\}. \,\,
\end{align}
\fn{The set $\mathcal{N}(\ell)$ contains the nearest neighbours of the site $\ell$. 
	We also remark that for $\L = \Lhom =A\Z^{d}$, there exists a nonsingular matrix $A'\in\R^{d \times d}$ 
	satisfying $A\Z^{d} = A'\Z^{d}$,} such that for all $\ell \in \L$,
\begin{eqnarray}\label{ass:A:nn} 
\fn{
\mathcal{N}(\ell) \supseteq \{\ell\pm A'e_i | 1 \leq i \leq d  \}.
}
\end{eqnarray}
%
\fn{For the sake of convenience, we identify the matrices $A$ and $A'$, 
so we assume that \eqref{ass:A:nn} is satisfied with $A$ replacing $A'$.}

Let $\ds \in \{2,3\}$ denote the physical dimension of the system, then 
\fn{
we introduce ${u_0: \L \rightarrow\R^{\ds}}$ as a fixed predictor prescribing the far-field boundary conditions and let ${y_0(\ell):=x(\ell) +u_0(\ell)}$ 
	denote the corresponding deformation predictor of the atomic or nuclear configuration, where $x:\L \to \R^{\ds}$ is a linear map that will be specified in the next subsection.
	In general, we decompose a deformation $y$ into the predictor $y_0$ 
	and a displacement corrector $u$ satisfying $y = y_{0} + u$.
	For systems with point defects
	or straight line dislocations, the appropriate predictors $u_0$ will be, respectively, specified
	in Sections~\ref{sec:pointD} and \ref{sec:dislocation}.}

\subsection{Displacement spaces}
\label{sec:space}

\fn{We now introduce the various spaces for displacement functions that will be used throughout the article. First, for $u:\L\rightarrow\mathbb{R}^{\ds}$ and $\ell\in\L$, $\rho\in\L-\ell$,
	we define the finite-difference}
\begin{eqnarray}\label{finite-difference}
D_{\rho}u(\ell) := u(\ell+\rho)-u(\ell)
\end{eqnarray}
and $Du(\ell):=\{D_{\rho}u(\ell)\}_{\rho\in\L-\ell}$.
We consider $Du(\ell) \in (\mathbb{R}^{\ds})^{\L-\ell}$ to be a finite-difference stencil with infinite range.
For a stencil $Du(\ell)$, we define the norms
\begin{align}\label{eq: nn norm}
\big|Du(\ell)\big|_{\mathcal{N}} := \bigg( \sum_{\rho\in \mathcal{N}(\ell) - \ell}
\big|D_\rho u(\ell)\big|^2 \bigg)^{1/2} , \,\,
\|Du\|_{\hc{l}^{2}_{\mathcal{N}}(\L)} := \bigg( \sum_{\ell \in \L} |Du(\ell)|_{\mathcal{N}}^2 \bigg)^{1/2},
\end{align}
and the corresponding functional space of finite-energy displacements
\begin{align}\label{space:UsH}
\UsH(\L;\R^{\ds}) := \big\{u:\L\rightarrow\mathbb{R}^{\ds} ~\big\lvert~
\|Du\|_{\hc{l}^{2}_{\mathcal{N}}(\L)}<\infty \big\}.
\end{align}
We also \hc{define} the following subspace of compact displacements
\begin{align}\label{space:Usz}
\Usz(\L;\R^{\ds}) := \big\{u:\L\rightarrow\mathbb{R}^{\ds} ~\big\lvert~
\exists~R>0 \text{ s.t. } u={\rm const} \text{ in } \L\setminus \hc{B_R} \big\}.
\end{align}
As we consider $\ds \in \{2,3\}$ to be fixed, we simply denote
$\UsH(\L;\R^{\ds})$ and $\Usz(\L;\R^{\ds})$ by $\UsH(\L)$ and $\Usz(\L)$, respectively, from this point onwards.
The next lemma follows directly from \cite[Proposition 9]{ortner_shapeev12}.

\begin{lemma}\label{lemma-WcW12dense}
	If {\asRC}~is satisfied, then
	$\UsH(\L)$ is the closure of $\Usz(\L)$ with respect to the norm  $\|D\cdot\|_{\hc{l}^{2}_{\mathcal{N}}(\L)}$.
\end{lemma}

In order to state our assumptions on the model,
we require spaces of admissible atomic or nuclear arrangements. 
\fn{As a technicality arising from the dislocations case, it is necessary to consider an auxiliary lattice $\L_{0} \subset \R^{\ds}$ 
satisfying {\asRC} with respect to a Bravais lattice $\L_{0}^{\hom} \subset \R^{\ds}$. 
In particular, we have that $\L \neq \L_{0}$ for dislocations and $\L = \L_{0}$ for point defects.}\label{ref-config-L0}
For $\frak{m}, \lambda > 0$, we \fn{introduce}
\begin{align}\label{space:adm}
\Adm^{0}_{\frak{m},\lambda}(\L_{0}) &:= \big\{ \, y:\L_{0} \rightarrow \R^{\ds} \, \big| \,\,\, 
B_{\lambda}(x) \cap y(\L_{0}) \neq \emptyset \quad \forall \, x \in \R^{\ds}, \nonumber
\\ \nonumber
&\qquad \qquad \quad \,\, |y(\ell) - y(m)| \geq \frak{m} |\ell-m|
\quad\forall~ \ell, m \in \L_{0} \, \big\}, \quad \text{and } \\[1ex]
\Adm^{0}(\L_{0}) &:= \bigcup_{\frak{m}>0} \bigcup_{\lambda>0} \Adm^{0}_{\frak{m},\lambda}(\L_{0}).
\end{align}

The parameter $\lambda > 0$ ensures that there are no large regions that are devoid of nuclei. 
We require the constraint $\lambda < \infty$ for the definition of the TFW site energy, see Section \ref{sec:TFW}.
The condition $\mathfrak{m}>0$ prevents the accumulation of atoms which is necessary
for many interatomic interactions, including Lennard--Jones, TFW and tight binding,
to define the site energies; see Sections \ref{sec:LJ}, \ref{sec:tb} and \ref{sec:TFW}.

Using $\Adm^{0}(\L_{0})$, we define the space of admissible configurations $\Adm(\L)$ for point defects and dislocations \fn{as follows}.

For point defects, we consider $d = \ds \in \{2,3\}$, $\L = \L_{0}$, $\L^{\hom} = \L^{\hom}_{0}$ and define the admissible space of arrangements to be
$\Adm_{m,\lambda}(\L) := \Adm^{0}_{m,\lambda}(\L)$ and $\Adm(\L) := \Adm^{0}(\L)$. 
\fn{We also define ${x, u_{0}: \L \rightarrow \R^{\ds}}$, $x^{\hom}: \Lhom \rightarrow \R^{\ds}$ by $x(\ell) = \ell$, $u_{0}(\ell) = 0$ for $\ell \in \L$ and $x^{\hom}(\ell') = \ell'$ for  $\ell' \in \Lhom$. We remark that for a finite number of substitution defects in an otherwise perfect lattice, we have that $\L = \Lhom$, otherwise $\L \neq \Lhom$ for all other point defect types.}

\fn{We model dislocations by following the setting used in \cite{2013-defects-v3}.
We consider a model for straight line dislocations obtained by projecting a 3D
crystal into 2D, hence we consider $d=2$, $\ds = 3$. Given a Bravais lattice $\L_{0} = B\Z^3$, oriented and rescaled such that the dislocation direction is $e_3 = (0,0,1) \in B\Z^3$
and with a Burgers vector of the form $\burg=(\burg_1,0,\burg_3) \in B\Z^3$. Additionally, we suppose that if $(\ell_{1}, \ell_{2}, \ell_{3}), (\ell_{1}, \ell_{2}, \ell'_{3})\in \L_{0}$, then $\ell_{3} - \ell'_{3} \in \Z$ and we consider displacements from $B\Z^3$ to $\R^3$ that are periodic in the dislocation
direction $e_{3}$. Thus, we choose a projected reference lattice
	\begin{align}
	\Lambda:=A\Z^2 = \Lhom =\{(\ell_1,\ell_2) ~|~\ell=(\ell_1,\ell_2,\ell_3)\in B\Z^3\}, \label{eq:L-Lhom-dis-def}
	\end{align}
which is again a Bravais lattice as $A$ is non-singular. 
We can construct a linear map $P_{0}: \L \to \L_{0}$, such that 
\begin{align} \label{eq:P0-def}
P_{0}(\ell_{1},\ell_{2}) = (\ell_{1},\ell_{2},P_{0}(\ell_{1},\ell_{2})_{3}), 
\end{align} 
where the third component is linear in $(\ell_{1},\ell_{2})$. 
We remark that the choice of the map $P_{0}$ is arbitrary.
Using this, given a configuration $\wt y: \L_{0} \to \R^{3}$, define the}
projected configuration 
${y: \L \to \mathbb{R}^{3}}$ by:
\begin{align}
\fn{
{ y}(\ell_{1},\ell_{2}) := \wt y(P_{0}(\ell_{1},\ell_{2})) - (0,0,P_{0}(\ell_{1},\ell_{2})_{3}). }
\label{eq:y def}
\end{align}
\fn{Similarly, we can express $\wt y$ in terms of $y$ by the expression 
\begin{align} 
	\wt y(\ell_{1},\ell_{2},\ell_{3}) = y(\ell_{1},\ell_{2}) + (0,0,\ell_{3}). \label{eq:wt y from y def}
\end{align} }
\fn{
We define the admissible space
\begin{align}
\Adm(\L) := \left\{ \, { y} : \L \to \mathbb{R}^{3} \,\, | \,\, {\wt y} \text{ given by } \eqref{eq:wt y from y def}
\text{ satisfies } {\wt y} \in \Adm^{0}(\L_{0})  \, \right\}, 
\label{eq:AL def}
\end{align}
and let $\Adm_{m,\lambda}(\L)$ denote the subspace corresponding to $\Adm^{0}_{m,\lambda}(\L_{0})$.
As $\L = \Lhom$, we define $x: \L \to \R^3$ by $x(\ell_{1},\ell_{2}) = (\ell_{1},\ell_{2},0)$ and let $x^{\hom} \equiv x$.}
%
The modelling of dislocations is discussed further in Section \ref{sec:dislocation}.

\fn{Our} analysis requires that the predictor $y_{0} \in \mathscr{A}(\L)$, 
and we ensure this for the predictors constructed in Sections \ref{sec:homogeneous}--\ref{sec:dislocation}; 
see in particular Lemma~\ref{Lemma - y0 Adm}.

It is also convenient to introduce the following spaces, 
which represent displacements corresponding to a subset of $\mathscr{A}(\L)$
\hc{or $\mathscr{A}_{\frak{m},\lambda}(\L)$}:
\begin{align}\label{space:HL}
\nonumber
\mathscr{H}(\L) &:= \big\{ \, u \in \UsH(\L) \, \big| \,\,\, y_{0} + u \in \Adm(\L) \big\}, \\
\mathscr{H}^{\rm c}(\L) &:= \big\{ \, u \in \Usz(\L) \, \big| \,\,\, y_{0} + u \in \Adm(\L) \big\} \hc{, \qquad \hc{\text{and}}} \\
\nonumber
\hc{\mathscr{H}_{\frak{m},\lambda}(\L)} &:= \hc{\big\{ \, u \in \UsH(\L) \, \big| \,\,\, 
	y_{0} + u \in \Adm_{\frak{m},\lambda}(\L) \big\}.}
\end{align}

The following result is a consequence of Lemma \ref{lemma-WcW12dense} 
\fn{and its proof is given in Appendix~\ref{Subsection-Proof of Lemma HcHdense}.

Recall that for $y : \L \to \R^{\ds}$, we define $\wt{y}: \L_{0} \to \R^{\ds}$ using \eqref{eq:wt y from y def} in the dislocations setting. For point defects, it is also convenient to define $\wt y := y$, since $\L = \L_{0}$. In either case, we have that $y \in \Adm(\L)$ if and only if $\wt y \in \Adm^{0}(\L_{0})$.

\begin{lemma}\label{lemma-HcHdense}
	If {\asRC} holds and 
there exist $\frak{m}, M > 0$ and $\Omega \subset \R^{\ds}$ and an interpolation $I\wt y_{0}: \Omega \to \R^{\ds}$ such that for all $x_{1}, x_{2} \in \Omega$
\begin{align} \label{eq:Iy0-ass}
\frak{m}|x_{1} - x_{2}| \leq |I\wt{y}_{0}(x_{1}) - I\wt{y}_{0}(x_{2})| \leq M|x_{1} - x_{2}|,
\end{align}
where $\Omega = \R^{\ds}$ for point defects and $\Omega = ( \R^{2} \setminus B_{R_{0}} ) \times \R$, with $R_{0} > 0$, for dislocations. Then \fn{$\mathscr{H}(\L)$ is open and}
	$\mathscr{H}^{\rm c}(\L)$ is dense in~$\mathscr{H}(\L)$ with respect to the
	$\|D\cdot\|_{\hc{l}^2_{\mathcal{N}}(\L)}$ norm. Moreover, 
	\begin{align} \label{eq:H-Lambda-equiv-def}
	\mathscr{H}(\L) =  \big\{ \, u = y - y_{0} \in \UsH(\L) \, \big| \,\,\,
	|\wt{y}(\ell) - \wt{y}(m)| \neq 0 \quad\forall~ \ell, m \in \L_{0}\fn{, \text{ s.t. } \ell \neq m} \big\}.
	\end{align}
	Therefore, if $\wt y_{0}(\ell) \neq \wt y_{0}(m)$ for all $\ell, m \in \L_{0}$ such that $\ell \neq m$, then $y_{0} \in \Adm(\L)$.
\end{lemma}

The equivalent formulation of $\mathscr{H}(\L)$ given in \eqref{eq:H-Lambda-equiv-def} shows that
it is a suitably rich space to pose our variational problem.

We remark that the assumption \eqref{eq:Iy0-ass} is shown for dislocations in 
Lemmas~\ref{Lemma - y0 Adm} and \ref{Lemma - y0 no holes}.
For point defects, as $u_{0} \equiv 0$, we can simply define $I\wt y_{0}(x) = x$, 
which trivially satisfies \eqref{eq:Iy0-ass} with $\Omega = \R^{\ds}$.}

\subsection{Site strain potential} 
\label{sec:siteE}
We first recall the linear maps $x: \L \rightarrow \R^{\ds}$, $x^{\hom}: \Lhom \rightarrow \R^{\ds}$ 
introduced in the previous subsection. 
We have $x(\ell) = \ell$, $x^{\hom}(\ell') = \ell'$, for point defects and 
$x(\ell_{1},\ell_{2}) = x^{\hom}(\ell_{1},\ell_{2}) = (\ell_{1},\ell_{2},0)$, for dislocations.

Let $\ell \in \L$ with $\L$ satisfying \asRC, then we consider a collection of mappings 
$V_{\ell}: \Dom(V_{\ell}) \rightarrow\R$ and $V^{\rm h}: \Dom(V^{\hom}) \rightarrow\R$  
with \label{site-strain-potential}
\begin{align*}
\Dom(V_{\ell}) &:= \{ \, Dw(\ell) \, | \, w:\L \rightarrow \R^{\ds} \text{ s.t. } x + w \in \mathscr{A}(\L) \, \}, \\
\Dom(V^{\hom}) &:= \{ \, Dw^{\hom}(\ell) \, | \, w^{\hom}:\Lhom \rightarrow \R^{\ds} \text{ s.t. } x^{\hom} 
+ w^{\hom} \in \Adm(\Lhom), ~\ell \in \Lhom \, \}.
\end{align*}

\begin{remark}
When $\L = \Lhom$, the periodicity of the lattice implies that $\Dom(V_{\ell})$ 
is independent of $\ell \in \L$ and in fact $\Dom(V_{\ell}) = \Dom(V^{\hom})$.
Otherwise, if $\L \neq \Lhom$, the stencils $Dw(\ell)$ and thus $\Dom(V_{\ell})$ 
are dependent on $\ell \in \L$.

Suppose $a \in \R^{\ds}$, $k\in\L$ and $w:\L \rightarrow \R^{\ds}$ satisfies $x + w \in \mathscr{A}(\L)$, 
then for $h \in (0,1]$, define $w_{h}(\ell) := w(\ell) + ha \delta_{\ell k}$. 
It is straightforward to see that $x + w_{h} \in \mathscr{A}(\L)$ for $h$ sufficiently small. 
This condition is sufficient for us to consider partial derivatives of $V_{\ell}$ on $\Dom(V_{\ell})$.
\end{remark}

We \fn{refer to $\{V_{\ell}\}_{\ell \in \L}$ as {\it site strain potentials} and $V^{\hom}$ 
as the {\it homogeneous} site strain potential} if the following assumptions \asSE~are satisfied. 
\fn{Below}, we consider $w$ such that $\fn{y = x + w} \in \mathscr{A}_{\frak{m},\lambda}(\L)$ 
with $\frak{m}, \lambda > 0$. \fn{We introduce the following notation: given $\ell,m \in \L$, 
we denote $r_{\ell m}(y) = |y(\ell) - y(m)|$.
When it is clear from the context, we simply denote $r_{\ell m}(y)$ by $r_{\ell m}$.}

\vskip 0.3cm

\begin{itemize}
	\label{as:SE:pr}
	\item[\asSER]
	{\it Regularity:}
	For all $\ell \in \L$, $V_{\ell}(\fn{Dw(\ell)})$ \fn{possesses} partial derivatives up to $\mathfrak{n}$-th order,
	where $\n\geq 3$ for point defects and $\n\geq 4$ for dislocations.
	We denote the partial derivatives by $V_{\ell,{\pmb \rho}}(\fn{Dw(\ell)})$ with
	$\pmb{\rho}=(\rho_1,\fn{\ldots},\rho_j)\in\b(\L-\ell\b)^{j}$,  $j=1,\fn{\ldots},\mathfrak{n}$
	and
	\begin{eqnarray}\label{eq:dV}
	V_{\hc{\ell},{\bm \rho}}\big({\bm g}\big) :=
	\frac{\partial^j V_{\hc{\ell}}\big({\bm g}\big)}
	{\partial {\bm g}_{\rho_1}\cdots\partial{\bm g}_{\rho_j}}
	\qquad{\rm for}\quad{\bm g}\in(\R^d)^{\L-\ell} .
	\end{eqnarray}

	\item[\asSEL]
	{\it Locality:}
	For $j=1,\fn{\ldots},\n$,
	\fn{define 
	\begin{align}
	\mathcal{P}(j) := \big\{ &\mathcal{A}=\{A_{1},\ldots,A_{k}\} \nonumber \\
	& \quad \fn{|}~
	\mathcal{A} \textnormal{ is a partition of } \{1, \ldots,j\} \fn{ \textnormal{ into non-empty sets}} \big\} \label{eq:partition-def}
	\end{align}
 	and denote the cardinality of the subset $A_i$ by $|A_i|$ and the smallest element in $A_i$ by $i'$. Then, there exist constants
 	${C_j = C_{j}(\frak{m},\lambda) > 0}$
 	independent of~$w$,
 	such that:
 	
	for \emph{point defects} and the \emph{homogeneous lattice} - with $\wf_i \in \mathscr{L}_{i}$ for $i = 1, \ldots, j$
	\begin{align}
	&|V_{\ell,{\bm \rho}}\big(\fn{Dw(\ell)}\big)|  \leq
	C_j \sum_{\{A_{1},\ldots,A_{k}\} \in \mathcal{P}(j)}
	\bigg( \prod_{1\leq i\leq k} \wf_{|A_i|}\big(|\rho_{i'}|\big)
	\prod_{m \in A_{i}} \delta_{\rho_{i'}\rho_{m}} \bigg),
	 \label{eq:Vloc}
	\end{align} }
		\fn{for \emph{dislocations} - with $\wf^{\log}_i \in \mathscr{L}^{\log}_{i}$ for $i = 3, \ldots, j + 2$
	\begin{align}
	&|V_{\ell,{\bm \rho}}\big(\fn{Dw(\ell)}\big)|  \leq
	C_j \sum_{\{A_{1},\ldots,A_{k}\} \in \mathcal{P}(j)}
	\bigg( \prod_{1\leq i\leq k} \wf^{\log}_{|A_i|+2}\big(|\rho_{i'}|\big)
	\prod_{m \in A_{i}} \delta_{\rho_{i'}\rho_{m}} \bigg),
	\label{eq:Vloc-dis}
	\end{align}
	for all $\ell \in \L$ and ${\bm \rho} \in (\L - \ell)^{j}$. The functions $\wf_{i}, \wf^{\log}_{i}$ are independent of~$w$. }
	
	\fn{
	\item[\asSEH]
	{\it Homogeneity:}
	The homogeneous site strain potential $V^{\hom}$ satisfies the assumptions {\asSER} and {\asSEL} with the following substitutions: $\Lhom$ for $\L$, $\Lhom_{*}$ for $\L - \ell$, $V^{\hom}$ for $V_{\ell}$ and where $\ell \in \L$ is omitted.
	
	\fn{\emph{For dislocations: } $V_{\ell} \equiv V^{\hom}$ for all $\ell \in \L = \Lhom$.}
	
	\fn{\emph{For point defects: }}
	There exists $s > d/2$ such that
	if ${y_{1} = x + w_{1} \in \mathscr{A}_{\frak{m},\lambda}(\L)}$, ${y_{2} = x^{\hom} + w_{2} \in \mathscr{A}_{\frak{m},\lambda}(\Lhom)}$ satisfy
	\begin{align}\label{eq: y12 gen hom conditions V}
	&\left\{y_{1}(n_{1})-y_{1}(\ell_{1}) ~\lvert~n_{1} \in \L,~r_{\ell_{1} n_{1}}(y_{1})\leq r \right\} \nonumber \\
	= &\left\{y_{2}(n_{2})-y_{2}(\ell_{2}) ~\lvert~n_{2} \in \Lhom,~r_{\ell_{2} n_{2}}(y_{2}) \leq r \right\}
	\end{align}
	for some $\ell_{1}\in\L$, $\ell_{2}\in\Lhom$ and $r>0$ (see Figure \ref{fig:hompic}),
	then there exist $C_{\rm h} = C_{\rm h}(\frak{m},\lambda)$ and $\wf_1 \in\Hw_{1}$
	such that
	\begin{align}\label{eq:Vhomogeneity}
	\left| V_{\ell_{1},n_{1}-\ell_{1}}(Dw_{1}(\ell_{1})) - V^{\hom}_{,n_{2}-\ell_{2}}(Dw_{2}(\ell_{2})) \right| \leq
	C_{\rm h} (1+r)^{-s} \wf_{1}\big( |\ell_{1} - n_{1}| \big),
	\end{align}
	where $n_{1} \in \L \setminus \ell_{1}$, $n_{2} \in \Lhom \setminus \ell_2$
	satisfy ${y_{1}(n_{1}) - y_{1}(\ell_{1}) = y_{2}(n_{2}) - y_{2}(\ell_{2})}$ and $r_{\ell_{1}n_{1}}(y_{1}) = r_{\ell_{2}n_{2}}(y_{2}) \leq r$.
	}

	\item[\asSEPS]	{\it Point symmetry:}
	The homogeneous site strain potential \fn{$V^{\hom}$} given in~{\asSEH} satisfies
	\begin{eqnarray}\label{point-symmetry}
	\fn{V^{\hom}\b((-g_{-\rho})_{\rho\in\Lhom_{*}}\b) = V^{\hom}(\pmb{g})}
	\qquad\forall~\pmb{g}\in\b(\R^{\ds}\b)^{\Lhom_{*}} .
	\end{eqnarray}
\end{itemize}

\begin{figure}[pb]
	\centerline{\psfig{file=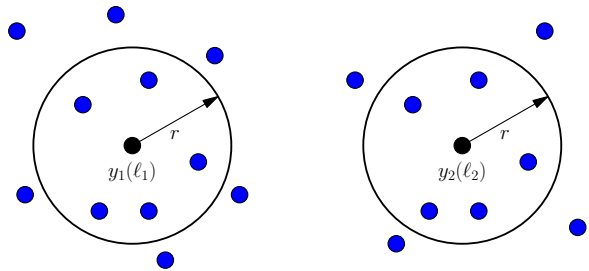,width=3.0in}}
	\caption{A diagram representing the assumption \eqref{eq: y12 gen hom conditions V} appearing in \asSEH.}
	\label{fig:hompic}
\end{figure}

\begin{remark}
	We require stronger locality estimates for
	the analysis of dislocations\fn{, this is justified in} detail in Sections~\ref{sec:homogeneous}--\ref{sec:dislocation}.

	We now give examples of \eqref{eq:Vloc} \fn{for point defects,} when $j \in \{1,2\}$ \fn{and $\wf_j\in\Hw_j$}:
\fn{	
	\begin{align*}
	\left|V_{\ell,\rho}\b(Dw(\ell)\b)\right| &\leq C_1\wf_1\b(|\rho|\b),  \\[1ex]
	\left|V_{\ell,\rho\vsig}\b(Dw(\ell)\b)\right| &\leq C_2\wf_1\b(|\rho|\b)\wf_1\b(|\vsig|\b) \quad {\rm with}~\rho\neq\vsig, \quad {\rm and} \\[1ex]
	\left|V_{\ell,\rho\rho}\b(Dw(\ell)\b)\right| &\leq C_2 \big( \wf_1\b(|\rho|\b)^2 +  \wf_2\b(|\rho|\b) \big).
	\end{align*} 
}
\end{remark}

\begin{remark}\label{remark:SEH}
		\fn{	
	For general point defect systems, we have that $\{ V_{\ell} \}_{\ell \in \L} \neq \{ V^{\hom} \}$. However, in addition to showing \eqref{eq:Vhomogeneity} for the interaction models discussed in Section~\ref{sec:examples}, it is also possible to derive an additional homogeneity estimate:
	suppose that ${y_{1} = x + w_{1} \in \Adm(\L), ~y_{2} = x^{\hom} + w_{2} \in \Adm(\Lhom)}$ satisfy $y_{1}(\ell)=y_{2}(\ell)$ for all $\ell\in \L\backslash B_{\Rcore}$, we have
	\begin{align}\label{eq:extra-hom-est}
	\left| V_{\ell}\b(Dw_1(\ell)\b)-V^{\hom}\b(Dw_2(\ell)\b) \right| \leq C \, \wt{\wf}_{\rm h}(|\ell| - \Rcore)
	\qquad \forall~ \ell \in \L\backslash B_{\Rcore} ,
	\end{align}
	where $\wt{\wf}_{\rm h} : \mathbb{R}_{\geq 0} \to \mathbb{R}_{> 0}$ decays at infinity.
	This estimate is not required \fn{in our analysis}, however, it allows us to interpret $V^{\hom}$ as
	the far-field limit of~$V_{\ell}$, in the following sense: as {\asRC} ensures that
	$\L\backslash B_{\Rcore} = \Lhom\backslash B_{\Rcore}$,  from \eqref{eq:extra-hom-est} we deduce  that
	\begin{align*}
	V_{\ell}\b(Du_1(\ell)\b) - V^{\hom}\b(Du_2(\ell)\b) \to 0 \quad \text{as} \quad |\ell|\rightarrow\infty. 
	\end{align*}

As we do not require the estimate \eqref{eq:extra-hom-est} in our analysis, we do not prove it in this paper. 
However, it has been shown for specific models, see for example,
}
\hc{\cite[Proof of Theorem 10]{chen15a} for the tight binding model and \cite[Lemma 6.12]{Nazar-thesis} for the TFW model.}
\end{remark}

\begin{remark}
	By writing the site strain potential as a function of $\fn{Dw(\ell)}$, we have already assumed that the model is invariant under translations.
	The assumption~{\asSEPS} can also be derived from the isometric and permutational symmetries of the model (see \asSEI~and \asSEP~in Section \ref{sec:examples}),  when all atoms in the system are of the same species.
	We can extend the assumptions to admit
	multiple species of atoms, e.g. \cite{2016-multipt}.
	However, the generalisation will introduce notation that is significantly more complex,
	so we will not pursue this here.
\end{remark}

\subsection{The energy difference functional}
\label{sec:energy}

Let $\L$ satisfy \asRC.
With a given predictor $y_0 = x + u_{0}$, we first formally define the energy-difference functional for
a displacement $u$ $\in \mathscr{H}(\L)$:
\begin{eqnarray}\label{eq:Ediff}
\E(u)  &=& \sum_{\ell\in\Lambda}\Big(V_{\ell}\big(Du_0(\ell)+Du(\ell)\big)-V_{\ell}\big(Du_0(\ell)\big)\Big).
\end{eqnarray}
The following \fn{key} \hc{theorem} states that the energy difference functional is defined on
$\mathscr{H}(\L)$, which we prove in Appendix \ref{sec:proof_theo_EW}.

\begin{theorem}\label{theorem:E-Wc}
	Suppose that \asSER~and \asSEL~hold, $y_{0} \in \Adm(\L)$ and also that \newline
	$F: \Usz(\L) \to \mathbb{R}$,
	given by
	\begin{align} \label{eq:T-ass}
	\b\<F,u\b\> := \sum_{\ell \in \L} \big\<\delta V_{\ell}(Du_0(\ell)),Du(\ell)\big\>
	= \sum_{\ell \in \L} \sum_{\rho \in \L - \ell} V_{\ell, \rho}(Du_0(\ell))^{\rm T} D_{\rho}u(\ell),
	\end{align}
	is a well-defined bounded linear functional on $(\Usz(\L) ,\|D\cdot\|_{\hc{l}^2_{\mathcal{N}}})$.
	Then
	\begin{itemize}
		\item[(i)]
		$\E:\mathscr{H}^{\rm c}(\L) \rightarrow\R$ is well-defined and $\delta \E(0) = F$.
		\item[(ii)]
		$\E:\mathscr{H}^{\rm c}(\L) \rightarrow\R$ is continuous with respect to $\|D\cdot\|_{\hc{l}^2_{\mathcal{N}}(\L)}$;
		hence, there exists a unique continuous extension to $\mathscr{H}(\L)$,
		which we still denote by $\E$.
		\item[(iii)]
		$\E:\mathscr{H}(\L) \rightarrow\R$ is $(\mathfrak{n}-1)$-times
		continuously differentiable with respect to the $\| D \cdot \|_{\hc{l}^2_{\mathcal{N}}(\L)}$ norm.
	\end{itemize}
\end{theorem}

\begin{remark}
	\fn{As we state Theorem \ref{theorem:E-Wc} generally, we then justify the condition that} $F$ is a bounded linear functional for the defect-free homogeneous system $\Lhom$ in Section~\ref{sec:homogeneous}, for point defects in Section~\ref{sec:pointD} and for dislocations in Section~\ref{sec:dislocation}.
\end{remark}

\begin{remark}
	For $u\in\mathscr{H}(\L)$ and ${\pmb v}=(v_1,\fn{\ldots},v_j)\in\big(\Usz(\L)\big)^j$,
	the $j$-th variation of $\E(u)$ is given by
	\begin{eqnarray*}
		\<\delta^j\E(u),{\pmb v}\> = \sum_{\ell\in\L} \sum_{{\pmb \rho}\in(\L-\ell)^j}
		\left\<V_{\ell,{\pmb \rho}}\b(\fn{Du_{0}(\ell)+}Du(\ell)\b), D_{\pmb \rho}\otimes {\pmb v}(\ell) \right\>,
	\end{eqnarray*}
	where ${\pmb \rho}=(\rho_1,\fn{\ldots},\rho_j)\in(\L-\ell)^j$ and
	$D_{\pmb \rho}\otimes {\pmb v}(\ell) := \b( D_{\rho_1}v_1(\ell),\fn{\ldots},D_{\rho_j}v_j(\ell) \b)$.
\end{remark}

\fn{
\begin{remark}
	The proof of Theorem \ref{theorem:E-Wc} also justifies that for all $\ell \in \L$ and {$u \in \mathscr{H}(\L)$}, that $V_{\ell}(\fn{Du_{0}(\ell)+}Du(\ell))$ is $(\mathfrak{n}-1)$-times
	continuously differentiable with respect to $\| D \cdot \|_{\hc{l}^2_{\mathcal{N}}}$.
\end{remark}
}

Under the conditions of Theorem \ref{theorem:E-Wc} we can formulate the
variational problem
\begin{equation}
\label{eq:results:problem}
\text{Find } \ua \in \arg\min \b\{\E(u) \bsep u\in \mathscr{H}(\L)\b\},
\end{equation}
where ``$\arg\min$" is understood in the sense of local minimality with respect
to the $\| D \cdot \|_{\hc{l}^2_{\mathcal{N}}(\L)}$ norm. \fn{ We refer to \eqref{eq:results:problem} as the \emph{lattice relaxation problem}.}
If $\ua$ is a minimiser, then $\ua$ is a first-order critical point satisfying
\begin{eqnarray}\label{eq:1st_order_pd}
\<\delta\E(\ua), v\> = 0	\qquad\forall~v\in\Usz(\L) .
\end{eqnarray}
We shall only be concerned with the structure of solutions to
\eqref{eq:results:problem} or \eqref{eq:1st_order_pd} in our analysis,
assuming their existence.

\section{Main Results}
\label{sec:results}
We will first discuss the homogeneous lattice as a preparation for defective systems,
and then present our results on point defects and dislocations respectively.

\subsection{Homogeneous lattice}
\label{sec:homogeneous}

For homogeneous systems the setting is as follows:
\begin{flushleft}\label{as:H}
	\asH \quad
	$\L=\Lhom=A\Z^d$, $d=\ds \in \{ 2, 3 \}$,
	$x(\ell)=\ell$ and $u_0(\ell)=0$ for all $\ell\in\Lhom$.
\end{flushleft}

In this setting, let $\E^{\rm h}(u):=\E(u)$ be the energy difference functional
defined by~\eqref{eq:Ediff}, which has the form
\begin{align}\label{energy-diff-hom}
\E^{\rm h}(u) = \sum_{\ell\in\Lhom}\Big(V\fn{^{\hom}}\big(Du(\ell)\big)-V\fn{^{\hom}}\big( 0 \big)\Big),
\end{align}
where $V\fn{^{\hom}}$ is the homogeneous site strain potential introduced in Section \ref{sec:siteE}.

Once we establish the following result, it will immediately follow from Theorem~\ref{theorem:E-Wc} that $\E^{\rm h}$ is defined.  We will readily make use of this fact in the analysis of the \fn{lattice} relaxation
problem for point defects.
The proof of the following result is given in Appendix \ref{sec:proof_lattice_symm}.

\begin{lemma}\label{lemma_lattice_symm}
	If the assumptions \asSER, \asSEL, \asSEPS~and \asH~are satisfied, then
	for all $u \in \Usz(\Lhom)$, $(\ell, \rho) \mapsto V_{,\rho}^{\rm h}(0) \cdot D_\rho u(\ell) \in \hc{l}^1(\Lhom \times \Lhom_*)$.
	Moreover,
	\begin{align} \label{eq:Thom-ass}
	\sum_{\ell \in \Lhom} \big\<\delta V^{\hom}(0),Du(\ell)\big\>
	= \sum_{\ell \in \Lhom} \sum_{\rho \in \Lhom_{*}} V^{\hom}_{, \rho}(0)^{\rm T} D_{\rho}u(\ell)  = 0
	\qquad\forall~u\in \Usz(\Lhom).
	\end{align}
\end{lemma}
\fn{ We} define the homogeneous difference operator $H:\UsH(\Lhom)\rightarrow\UsHd(\Lhom)$
by
\begin{eqnarray}\label{eq:def-Huv}
\nonumber
\b\< Hu,v \b\> &:=&
\sum_{\ell\in\Lhom}\b\< \delta^2 V\fn{^{\hom}}(\pmb{0}) Du(\ell), Dv(\ell)\b\>
\\[1ex]
&=&
\sum_{\ell\in\Lhom}\sum_{\rho,\zeta\in\Lhom_{*}}
D_{\zeta}v(\ell)^{\rm T} \cdot \fn{V^{\hom}_{,\rho\zeta}}(\pmb{0}) \cdot D_{\rho}u(\ell)
\qquad \forall~v\in\UsH(\Lhom).
\end{eqnarray}
In our analysis of \fn{defective} systems, we will require the following strong stability of the host homogeneous lattice.

\begin{flushleft}\label{as:LS}
	\asLS \quad {\it Lattice stability:}
	There exists a constant $c_{\rm L} > 0$ depending only on $\Lhom$, such that
	\begin{eqnarray*}
		\b\< Hv,v \b\> \geq c_{\rm L} \|Dv\|^2_{\hc{l}^2_{\mathcal{N}}\hc{(\Lhom)}}
		\qquad\forall~v\in\Usz(\Lhom).
	\end{eqnarray*}
\end{flushleft}

The analysis of the homogeneous difference operator will be presented in \ref{sec:lattice_green}, which will be heavily used in our decay estimate proofs.

\subsection{Point defects}
\label{sec:pointD}

For systems with point defects,
we consider the following setting:
\begin{flushleft}\label{as:P}
	\asP \quad
	$\L$ satisfying {\asRC} with respect to $\Lhom = A\Z^{d}$, $d=\ds \in \{2, 3\}$, $x(\ell)=\ell$ and $u_0(\ell)=0$ for any $\ell\in\Lhom$.
\end{flushleft}

With this setting, let $\E(u)$ be the energy difference functional
defined by \eqref{eq:Ediff}. The following result together with
Theorem \ref{theorem:E-Wc} implies that $\E$ is defined on $\mathscr{H}(\L)$ and ${\E \in C^{\mathfrak{n}-1}(\mathscr{H}(\L),\|D\cdot\|_{\hc{l}^2_{\mathcal{N}}(\L)})}$.

\begin{theorem} \label{corr:deltaE0-pd-est}
	Suppose that the assumptions \asSE~and \asP~are satisfied with ${s>d/2}$ in \asSEH. Then
	$F: \Usz(\L) \to \mathbb{R}$, given by
	\begin{align}\label{eq:T-def-ass}
	\b\<F,u\b\> := \sum_{\ell \in \L} \big\<\delta V_{\ell}(0),Du(\ell)\big\>
	= \sum_{\ell \in \L} \sum_{\rho \in \L - \ell} V_{\ell, \rho}(0)^{\rm T} D_{\rho}u(\ell)
	\end{align}
	defines a bounded linear functional on $(\Usz(\L),\|D\cdot\|_{\hc{l}^2_{\mathcal{N}}(\L)})$.
	That is, there exists $C > 0$ such that $\b| \b\<F,u\b\> \b| \leq C \| Du \|_{\hc{l}^2\hc{(\L)}} < \infty$, for all $u \in \Usz(\L)$.
\end{theorem}

The proof of Theorem \ref{corr:deltaE0-pd-est} is given in Appendix \ref{sec:proof_deltaE0_pd-est}.

Next, we consider the variational problem \eqref{eq:results:problem}.
We establish the rate of decay for a minimising displacement $\ua$
in the following result, whose proof is given in Appendix \ref{sec:proof_decay_pd}.

\begin{theorem}\label{theorem:point-defect-decay}
	Suppose that the conditions in Theorem \ref{corr:deltaE0-pd-est} and  \asLS~hold,
	then for any $\bar{u}$ solving \eqref{eq:1st_order_pd}, there exists $C > 0$ depending on $\bar{u}$ and $s$,
	such that
	\begin{eqnarray} \label{eq:pd-ubar-decay}
	\b|D_{\rho}\bar{u}(\ell)\b| \leq C |\rho|
	\left\{ \begin{array}{ll}
	\b(1+|\ell|\b)^{-d} & {\rm if}~s> d ,
	\\[1ex]
	\b(1+|\ell|\b)^{-s}  \log\b(2+|\ell|\b)  & {\rm if}~\frac{d}{2}<s\leq d ,
	\end{array}\right.
	\end{eqnarray}
	for all $\ell \in \L$ and $\rho \in \L - \ell$, where $s$ is the parameter in \asSEH.
\end{theorem}

In the case $s>d$, an immediately conclusion of this theorem is \fn{that}
\begin{align}
	|D\bar{u}(\ell)|\fn{_{\mathcal{N}}} \leq C\b(1+|\ell|\b)^{-d}, \label{eq:Dbaru-stencil-decay}
\end{align}
\fn{using the stencil norm defined in \eqref{eq: nn norm}. 
	Additionally, in equation \eqref{eq:wnormdef} of Appendix \ref{sec:proof_equiv_norms}, given $k \in \mathbb{N}$ and $\wf \in \Hw_{k}$, 
	we will define weighted finite-difference stencil norms $|D\bar{u}(\ell)|_{\wf,k}$, which can also be shown to satisfy \eqref{eq:Dbaru-stencil-decay}.}
Similar to \cite[Theorem 2.3]{2013-defects-v3}, we also have decay
estimates for higher order derivatives:
for $j \leq \mathfrak{n}-1$, there exist constants $C_j>0$  such that
\begin{eqnarray}\label{decay_Dku_pd}
\b|D_{\pmb{\rho}}\bar{u}(\ell)\b| \leq C_j \bigg(\prod_{i=1}^j|\rho_i| \bigg)
\b(1+|\ell|\b)^{1-d-j}
\end{eqnarray}
for $\pmb{\rho}=(\rho_{1},\ldots,\rho_{j}) \in (\Lhom_{*} \cap \hc{B_{R_{\rm c}}})^{j}$ with some cutoff $R_{\rm c}>0$
and $\ell \in \L \setminus B_{\Rcore}$ with $|\ell|$ sufficiently large.
The proof is a straightforward combination of the arguments used in \cite{2013-defects-v3}
and the techniques developed in the present paper, hence we omit it. It is not clear whether the higher-order decay generalises to arbitrarily large $\rho_i$.

\subsection{Dislocations}
\label{sec:dislocation}
\fn{Throughout this subsection, we assume the existence of site strain potentials $\{ V_{\ell} \}_{\ell \in \Lhom}$ satisfying {\asSEH}, hence $V_{\ell} \equiv V^{\hom}$ for all $\ell \in \L = \Lhom$, where the lattice $\L$ is defined in \eqref{eq:L-Lhom-dis-def}.}

\fn{Recall the model of straight line dislocations introduced in Page~\pageref{eq:L-Lhom-dis-def} of Section~\ref{sec:space}, where $d = 2, \ds = 3$, $\L_{0} \subset \R^{3}$ is a Bravais lattice and $\L \subset \R^2$ is the Bravais lattice defined in \eqref{eq:L-Lhom-dis-def} by projecting $\L_{0}$ onto the plane $\{x_{3} = 0 \}$.}
This projection gives rise to a projected 2D
site strain potential, for $\fn{y = x + w} \in \Adm(\L)$, \fn{denoted by $V^{\hom}(Dw(\ell))$ for $\ell \in \L$.}

Let $\hat{x}\in\R^2$ be the position of the dislocation core and
\begin{align*}
\Gamma := \{x \in \R^2~|~x_2=\hat{x}_2,~x_1\geq\hat{x}_1\}
\end{align*}
be the ``branch
cut'', with $\hat{x}$ chosen such that $\Gamma\cap\Lambda=\emptyset$.  Following
\cite{2013-defects-v3}, we define the far-field predictor $u_0$ by
\begin{eqnarray}\label{predictor-u_0-dislocation}
u_0(\ell):=\ulin(\xi^{-1}(\ell)) + u^{\rm c}(\ell), \quad \text{for all } \ell \in \L,
\end{eqnarray}
where $\ulin \in C^\infty(\R^2 \setminus \Gamma; \R^3)$ and ${u^{\rm c} \in C^{\infty}_{\rm c}(\R^2; \R^{3})}$.
We introduce the function $u^{\rm c}$ to ensure that $y_{0} \in \Adm(\L)$,
see Lemma \ref{Lemma - y0 Adm} for more details. The displacement $\ulin$ is the continuum linear
elasticity solution solving
$\sum_{\rho,\vsig \in \Lhom_{*}} \fn{V^{\hom}_{,\rho\vsig}}(\bfO)\D_{\rho}\D_{\vsig} \ulin(x) = 0$
(see \cite[Appendix A]{2013-defects-v3} for more details). In particular, under
the lattice stability assumption {\bf (LS)} there exists a solution of the form
\begin{align}\label{eq_ulin}
\ulin_{i}(\hat{x}+x)&= \textnormal{Re} \left(\sum_{n=1}^3 B_{i,n}\log(x_1+p_n x_2) \right) \\
\notag
\,\,\,{\rm with} \qquad &B_{i,n},~p_n\in\mathbb{C},~i,n=1,2,3~~ \\
\label{eq_xi}
\text{and} \qquad &\xi(x)=x-\burg_{12}\frac{1}{2\pi}
\eta\left(\frac{|x-\hat{x}|}{\hat{r}}\right)
\arg(x-\hat{x}),
\end{align}
where $\arg(x)$ denotes the angle in $(0,2\pi)$ between $x$ and
$\burg_{12} = (\burg_1, \burg_2) = (\burg_1, 0)$,
$\eta\in C^{\infty}(\R)$ with $\eta=0$ in $(-\infty,0]$, $\eta=1$ in $[1,\infty)$
and $\hat{r} > 0$ is sufficiently large to ensure that $\xi: \mathbb{R}^{2} \setminus \Gamma \to \mathbb{R}^{2} \setminus \Gamma$ is a bijection.
Then we have \cite[Lemma 3.1]{2013-defects-v3}
\begin{eqnarray}\label{ac_def_Lemma3.1}
\big|\nabla^n u_0(x)\big| \leq C|x|^{-n} \qquad\forall~x\in \R^2\backslash(\Gamma\cup B_{\hat{r}}),
\qquad{\rm for~all~}n\in\N .
\end{eqnarray}

\begin{remark}
	It is necessary that $y_{0} \in \Adm(\L)$ and while it is straightforward to observe that $y_{0}$ satisfies the required properties in the far-field, some care is needed to ensure no collisions occur around the defect core.
\end{remark}

The predictor $y_0 = x + u_0$ is constructed in such a way that $y_0$
jumps across~$\Gamma$, which encodes the presence of the dislocation.
Alternatively, constructing the jump over the left-hand plane $\{ x_1 \leq
\hat{x}_1 \}$ also achieves this; see
Appendix~\ref{sec:proof_deltaE0_dislocation} for details.

\begin{remark}\label{remark:anti-in-plane}
	One can treat anti-plane models of pure screw dislocations by admitting
	displacements of the form $u_0 = (0, 0, u_{0,3})$ and $u = (0, 0, u_3)$.
	Similarly, one can treat the in-plane models of pure edge dislocations by
	admitting displacements of the form $u_0 = (u_{0,1}, u_{0,2}, 0)$ and
	$u = (u_1, u_2, 0)$ (see \cite{2013-defects-v3}).
\end{remark}

In summary, we treat dislocations in the following setting {\asD}:

\begin{flushleft}\label{as:D}
	\asD \,\,
	$\L=\Lhom=A\Z^2$, $d=2$, $\ds = 3$, \fn{$x(\ell_{1},\ell_{2}) = (\ell_{1},\ell_{2},0)$} and $u_0$ is given by \eqref{predictor-u_0-dislocation}.
\end{flushleft}

With this setting, let $\E(u)$ be the energy difference functional
defined by \eqref{eq:Ediff}. The following lemma together with
Theorem \ref{theorem:E-Wc} implies that $\E$ is well-defined and that ${\E \in C^{\mathfrak{n}-1}(\mathscr{H},\| D\cdot \|_{\hc{l}^2_{\mathcal{N}}(\L)})}$.
The proof is given in Appendix \ref{sec:proof_deltaE0_dislocation}.

\begin{theorem} \label{corr:deltaE0-dislocation}
	Suppose the assumptions \fn{{\asSE}} and \asD~are satisfied,
	then ${F: \Usz(\L) \to \mathbb{R}}$, given by
	\begin{align}
	\b\<F,u\b\> := \sum_{\ell \in \Lhom} \big\<\delta \fn{V^{\hom}}(Du_{0}(\ell)),Du(\ell)\big\>
	= \sum_{\ell \in \Lhom} \sum_{\rho \in \Lhom_{*}} \fn{V^{\hom}_{, \rho}}(Du_{0}(\ell))^{\rm T} D_{\rho}u(\ell), \label{eq:T-dis-ass}
	\end{align}
	defines a bounded linear functional on $(\Usz(\L),\|D\cdot\|_{\hc{l}^2_{\mathcal{N}}(\L)})$.
	That is, there exists $C > 0$ such that $\b| \b\<F,u\b\> \b| \leq C \| Du \|_{\hc{l}^2(\Lhom)} < \infty$, for all $u \in \Usz(\L)$.
\end{theorem}

We now consider the variational problem \eqref{eq:results:problem}.
We establish the rate of decay for a minimising displacement $\ua$
in the following result, whose proof is given in Appendix \ref{sec:proof_decay_pd}.

\begin{theorem}\label{theorem:dislocation-decay}
	Suppose that the conditions in Theorem \ref{corr:deltaE0-dislocation} and \asLS~hold,
	then for any $\bar{u}$ solving \eqref{eq:1st_order_pd}, there exists $C > 0$ depending on $\bar{u}$,
	such that
	\begin{eqnarray}
	\label{eq:dislocation-ubar-decay}
	\b|D_{\rho}\bar{u}(\ell)\b| \leq C |\rho|
	\b(1+|\ell|\b)^{-2} \log(2+|\ell|)
	\end{eqnarray}
	for all $\ell\in\L$ and $\rho \in \L - \ell$.
\end{theorem}

An immediately conclusion of this theorem is
\begin{align}
	|D\bar{u}(\ell)|\fn{_{\mathcal{N}}} \leq C\b(1+|\ell|\b)^{-2}\log(2+|\ell|)\fn{.} \label{eq:Dubar-decay-est-dis}
\end{align}
\fn{Additionally, given $k \in \mathbb{N}$ and $\wf \in \Hw_{k}$, the stencil norm $|D\bar{u}(\ell)|_{\wf,k}$, 
	introduced in equation \eqref{eq:wnormdef} of Appendix \ref{sec:proof_equiv_norms}, can also be shown to satisfy \eqref{eq:Dubar-decay-est-dis}.}

Similar to \cite[Theorem 3.5]{2013-defects-v3}, we also have a decay estimate
for higher order derivatives: For $j \leq \mathfrak{n}-2$ 
there exist constants $C_j>0$ such that
\begin{eqnarray}\label{decay_Dku_dislocation}
\b|\wt D_{\pmb{\rho}}\bar{u}(\ell)\b| \leq C_j \bigg(\prod_{i=1}^j|\rho_i| \bigg)
\b(1+|\ell|\b)^{-1-j}\log\b(2+|\ell|\b)
\end{eqnarray}
for ${\pmb \rho} = (\rho_{1},\ldots,\rho_j)  \in (\Lhom_{*} \cap \hc{B_{R_{\rm c}}})^{j}$ with some cutoff $R_{\rm c}>0$
and $\ell\in\Lhom$ with $|\ell|$ sufficiently large.
The finite-difference operator $\wt D$ appearing in \eqref{decay_Dku_dislocation} is a permutation
of $D$, accounting for plastic slip, and is described in more detail in  Appendix~\ref{sec:appendix:dislocation:general}.
Similarly as for \eqref{decay_Dku_pd}, the proof of \eqref{decay_Dku_dislocation}
follows from a straightforward combination of  the arguments used in \cite{2013-defects-v3}
and the techniques developed in the present paper, hence we omit it.

\section{Discussions of practical models}
\label{sec:examples}
In the previous sections, we  developed an abstract formulation for the relaxation of crystalline defects. We now demonstrate that this framework applies to several physical models from molecular mechanics to quantum mechanics, that were previously untreated.

\fn{Instead of the site strain potentials $\{V_{\ell}\}_{\ell \in \L}$, we discuss the
	physical models through {\it site energies} $\{ \Phi_{\ell'} \}_{\ell' \in \L_{0}}$, where $\Phi_{\ell'}:\Adm^{0}(\L_{0})\rightarrow\R$ \fn{will be constructed} in this section. \fn{Then, for $y=\fn{x + w \in \Adm(\L)}$ and $\ell \in \L$, we define the} site strain potential $V_{\ell}\b(D\fn{w}(\ell)\b)$ \fn{as}:
	\begin{equation} \label{siteE-sietV}
	V_{\ell}\big(D\fn{w}(\ell)\big) :=
	\left\{ \begin{array}{ll}
	\Phi_{\ell}(y) & \quad \text{for point defects},
	\\[1ex]
	\Phi_{P_{0}(\ell)}(\wt y) & \quad \text{for dislocations},
	\end{array}\right.
	\end{equation}
using that $\L = \L_{0}$, $\Adm(\L) = \Adm^{0}(\L_0)$ for point defects and $\wt y \in \Adm^{0}(\L_{0})$, $P_{0}:\L \to \L_{0}$ given by \eqref{eq:wt y from y def}, \eqref{eq:P0-def} respectively for dislocations.}

\fn{Using \eqref{siteE-sietV}, we can rewrite the assumptions {\asSER} and {\asSEL} in terms of the site energy $\Phi_\ell$ (we skip the details for simplicity of presentations).
	In addition, we require that the site energy is invariant under permutations and isometries
	as in the following assumptions, from which \asSEPS~can be derived.
	\begin{itemize}
		\item[\asSEP]
		{\it Symmetry under permutations:}
		If $\pi:\L_{0}\rightarrow\L_{0}$ is a bijection,
		then for any $\ell' \in \L_{0}$ and $y' \in \Adm^{0}(\L_{0})$, we have $ y'\circ\pi\in\Adm^{0}(\L_{0})$ and $\Phi_{\pi(\ell')}(y'\circ\pi) = \Phi_{\ell'}(y')$.
		
		\item[\asSEI]
		{\it Symmetry under isometries:}
		If $\varphi:\R^{\ds}\rightarrow\R^{\ds}$ is an isometry,
		then for any $\ell' \in \L_{0}$ and $y' \in \Adm^{0}(\L_{0})$, we have $\varphi\circ y'\in\Adm^{0}(\L_{0})$ and $\Phi_{\ell'}(\varphi\circ y') = \Phi_{\ell'}(y')$.
	\end{itemize}
}

\fn{Recall that {\asSEH} requires the homogeneous site strain potential $V^{\hom}$, so we also construct $\{\Phi^{\hom}_{\ell'} \}_{\ell' \in \L_{0}}$ that are used to define $V^{\hom}$. For $y^{\hom} = x^{\hom} + w^{\hom} \in \Adm(\Lhom)$ and $\ell \in \Lhom$, we first introduce:
\begin{equation*}
V^{\hom}_{\ell}\big(D\fn{w}^{\hom}(\ell)\big) :=
\left\{ \begin{array}{ll}
\Phi^{\hom}_{\ell}(y^{\hom}) & \quad \text{for point defects},
\\[1ex]
\Phi_{P_{0}(\ell)}(\wt y^{\hom}) & \quad \text{for dislocations},
\end{array}\right.
\end{equation*}
using that $\Lhom = \Lhom_{0}$, $\Adm(\Lhom) = \Adm^{0}(\Lhom_0)$ for point defects and that $\L = \Lhom$ for dislocations. We also assume $\Phi^{\hom}_{\ell}$ satisfies the same properties as $\Phi_{\ell}$, hence it follows from {\asSEP} that $V^{\hom}_{\ell} = V^{\hom}_{m}$ for all $\ell, m \in \Lhom$, therefore there exists $V^{\hom}$ satisfying: for all $\ell \in \Lhom$ and $y^{\hom} = x^{\hom} + w^{\hom} \in \Adm(\Lhom)$
\begin{equation} \label{siteE-sietVhom}
V^{\hom}\big(D\fn{w}^{\hom}(\ell)\big) =
\left\{ \begin{array}{ll}
\Phi^{\hom}_{\ell}(y^{\hom}) & \quad \text{for point defects},
\\[1ex]
\Phi_{P_{0}(\ell)}(\wt y^{\hom}) & \quad \text{for dislocations}.
\end{array}\right.
\end{equation}
For the dislocations case, observe that {\asSEH} follows directly from \eqref{siteE-sietV}--\eqref{siteE-sietVhom}.}
\fn{We remark} that \fn{as} $\Phi_{\ell}$ is a function of \fn{the} atomic configuration, \fn{it} is more convenient to describe practical models. We will discuss \fn{the construction of site energies for different models} in the following subsections.

\fn{Subsequently,} we \fn{only} consider models for three-dimensional systems, hence ${\ds = 3}$, $d \in \{2,3\}$ and we suppose $\L \subset \R^{d}$, $\L_{0} \subset \R^{3}$ satisfy {\asRC} with respect to Bravais lattices $\Lhom \subset \R^{d}$, $\Lhom_{0} \subset \R^{3}$, respectively. \fn{Specifically, we choose $d = 3$ for point defects and $d = 2$ for dislocations.} In the following subsections, we define \fn{the site energies $\Phi_{\ell}, \Phi^{\hom}_{\ell'}$ for each model.} 
%
%

\subsection{Lennard--Jones}
\label{sec:LJ}
The Lennard--Jones potential\cite{LennardJones:1924a} is a simple model that
approximates the interaction between
a pair of neutral atoms or molecules, in three dimensions.
The most common expression for the Lennard--Jones potential is
\begin{eqnarray}\label{V_LJ}
\phi_{\rm LJ}(r) := 4\varepsilon\left( \left(\frac{\sigma}{r}\right)^{12}
- \left(\frac{\sigma}{r}\right)^{6} \right) ,
\end{eqnarray}
where $\varepsilon$ is the depth of the potential well, $\sigma$ is the finite
distance at which the inter-particle potential is zero. We consider the more general model
\begin{eqnarray}
\phi_{\rm LJ}^{(p,q)}(r) := C_{1}r^{-p} - C_{2}r^{-q} ,
\end{eqnarray}
with $p > q > 0$ and $C_{1},C_{2} > 0$.
Then the Lennard--Jones site energy is given by
\fn{\begin{eqnarray}\label{E_LJ}
\Phi_{\ell}(y) :=\frac{1}{2}\sum_{ k \in \L_{0} \setminus \ell}\phi_{\rm LJ}^{(p,q)}(r_{\ell k}(y)),\quad \fn{\Phi^{\hom}_{\ell'}(\fn{y^{\hom}}) :=\frac{1}{2}\sum_{ k' \in \L^{\hom}_{0} \setminus \ell'}\phi_{\rm LJ}^{(p,q)}(r_{\ell' k'}(\fn{y^{\hom}})),}
\end{eqnarray}
where $r_{\ell k}(y)=|y(\ell)-y(k)|$.}
Note that by fixing the parameters $C_1,C_2$ in the potential,
we have assumed that all atoms of the system belong to the same species.

We obtain from a simple calculation that
if $q>3$, then the assumptions \asSE~are satisfied by the Lennard--Jones site energy \eqref{E_LJ} for point defects,
and if $q>5$, \asSE~are satisfied for dislocations as well.

We briefly justify that \asSEH~holds with arbitrary $s>0$.
%
\fn{Suppose} $y\in \Adm(\L_{0})$ and ${\fn{y^{\hom}}\in \Adm(\Lhom_{0})}$ \fn{and there exists} $\ell\in \L_{0}$ and $\ell' \in \Lhom_{0}$ satisfying
\begin{eqnarray}\label{y-yprim-models}
\!\!\!\!\!\!\!\!\!\!\left\{y(n)-y(\ell) ~\lvert~n \in \L_{0},~r_{\ell n}(y)\leq r \right\}
= \left\{\fn{y^{\hom}}(n')-\fn{y^{\hom}}(\ell') ~\lvert~n' \in \Lhom_{0},~r_{\ell' n'}(\fn{y^{\hom}}) \leq r \right\},
\end{eqnarray}
with $r \geq 0$, we have that there exists $n\in \L_{0}$ and $n'\in\Lhom_{0}$
such that \newline $y(\ell) - y(n)  = \fn{y^{\hom}}(\ell') - \fn{y^{\hom}}(n')$ and
\begin{eqnarray*}
	\Phi_{\ell, n} (y)
	=  \phi_{\rm{LJ}}'( r_{\ell n} (y) ) \frac{y(n) - y(\ell)}{ r_{\ell n} (y) }
	= \phi_{\rm{LJ}}'( r_{\ell' n'} (\fn{y^{\hom}}) ) \frac{ \fn{y^{\hom}}(n') - \fn{y^{\hom}}(\ell')}{ r_{\ell' n'} (\fn{y^{\hom}}) }
	=  \Phi^{\hom}_{\ell', n'} (\fn{y^{\hom}}),
\end{eqnarray*}
hence as $|\Phi_{\ell,n}(y) - \Phi^{\hom}_{\ell',n'}(\fn{y^{\hom}})|=0$, \asSEH~is satisfied with arbitrary $s>0$.

Similarly, we can construct the site energy for general pair potential models
$\Phi_\ell(y) = \frac12 \sum_{\fn{k \in \L_{0} \setminus \ell}} \phi(r_{\ell k}\fn{(y)})$
and show that the assumptions \asSE~are satisfied for most of the practical models,
such as Morse potential \cite{Morse:1929a}, as long as $|\phi^{(j)}(r)| \lesssim r^{-q-j}$ with $q > 3$
for point defects and $q > 5$ for dislocations.

\subsection{Embedded atom model}
\label{sec:EAM}

The site energy of the embedded atom model (EAM)\cite{Daw:1984a}
can be written as
(a pair potential contribution is omitted from the standard EAM formula since we have discussed them in the previous subsection)
\fn{\begin{eqnarray}\label{E_EAM}
\Phi_\ell(y) := \fn{J}_{\alpha} \bigg( \sum_{k\in\L_0 \setminus \ell} \fn{\varrho}_{\beta}(r_{\ell k}(y))\bigg), \,\, \Phi^{\hom}_{\ell'}(\fn{y^{\hom}}) := \fn{J}_{\alpha} \bigg( \sum_{k'\in\L^{\hom}_0 \setminus \ell} \fn{\varrho}_{\beta}(r_{\ell' k'}(\fn{y^{\hom}}))\bigg),
\end{eqnarray}
where $r_{\ell k}(y)=|y(\ell)-y(k)|$},
$\fn{\varrho}_{\beta}$ is the contribution to the electron charge density from atom $k$ of type
$\beta$ at the location of atom $\ell$, and $\fn{J}_\alpha$ is an embedding
function that represents the energy required to place atom $\ell$ of
type $\alpha$ into the electron cloud.

In most of the practical EAM models, there is a cutoff $R_{\rm c}$ established
and interactions between atoms separated by more than $R_{\rm c}$ are ignored,
i.e. the sum in~\eqref{E_EAM} is taken over $k\in\L_0$ satisfying $r_{\ell k}\leq R_{\rm c}$.
We consider here a more general model with infinite range of interaction. This serves
as a simple study of the kind of decay required for many-body interaction.

Let $y\in\Adm^{0}_{\frak{m},\lambda}(\L_{0})$.
We assume a smooth and positive density contribution:
$\fn{\varrho}_{\beta}\in C^{\mathfrak{n}}((0,\infty))$ and $\fn{\varrho}_{\beta}(r)>0,~\forall~r>0$.
We also require a decay property
\begin{eqnarray}\label{condition-rho}
\fn{\varrho}_{\beta}^{(j)}(r)\leq C(1+r)^{-q-j} \quad \forall~r>0,~0\leq j\leq\mathfrak{n}
\qquad{\rm with~some~} q>0.
\end{eqnarray}
Since $y\in\Adm^{0}_{\frak{m},\lambda}(\L_{0})$ and $\fn{\varrho}_{\beta}$ is bounded,
we have that the range of electron density $\sum_{\fn{k \in \L_{0} \setminus \ell}} \fn{\varrho}_{\beta}(r_{\ell k})$
lies in a bounded interval $\big[\underline{\sigma}_{\frak{m},\lambda},\overline{\sigma}_{\frak{m},\lambda}\big]$,
with $0<\underline{\sigma}_{\frak{m},\lambda}<\overline{\sigma}_{\frak{m},\lambda}<\infty$
depending on $\mathfrak{m}$ and $\lambda$.

Further, we assume that the embedding function $\fn{J}_{\alpha}$ is smooth and bounded for admissible electron densities:
$\fn{J}_{\alpha}\in C^{\mathfrak{n}}((0,\infty))$ and
\begin{eqnarray}\label{condition-F}
\big|\fn{J}^{(i)}_{\alpha}(x)\big| \leq C_{\frak{m},\lambda}
\quad\forall~x\in\big[\underline{\sigma}_{\frak{m},\lambda},\overline{\sigma}_{\frak{m},\lambda}\big],
~1\leq i\leq\mathfrak{n}
\end{eqnarray}
where $C_{\frak{m},\lambda}>0$ depends on $\mathfrak{m}$ and $\lambda$.
This assumption is satisfied for most practical models
as $\fn{J}_{\alpha}(\cdot)$ is usually a polynomial or $\sqrt{\cdot}$
(\fn{e.g. }Gupta and Finnis--Sinclair models).

Assume that all atoms of the system belong to the same species (i.e. $\fn{J}_{\alpha}\equiv \fn{J}$ and $\fn{\varrho}_{\beta}\equiv\fn{\varrho}$).
We have from a direct calculation that if $q>3$, then the assumptions \asSEL~are satisfied by the EAM site energy \eqref{E_EAM} for point defects;
and if $q>5$, than \asSEL~are satisfied for dislocations.

To see the conditions for \asSEH, we compare two configurations $y$ and $\fn{y^{\hom}}$ satisfying \eqref{y-yprim-models}, for some $r > 0$.
Define
\begin{align*}
&\fn{\tilde{\fn{\varrho}}_{1}(y)}=\sum_{\substack{\fn{k \in \L_{0} \setminus \ell} \\r_{\ell k(y)}\leq r}} \fn{\varrho}(r_{\ell k}(y)), \quad \fn{\tilde{\fn{\varrho}}_{2}(y)} = \sum_{\substack{\fn{k \in \L_{0} \setminus \ell} \\r_{\ell k(y)} > r}} \fn{\varrho}(r_{\ell k}(y)),
\end{align*}
\fn{similarly} we define $\fn{\tilde{\fn{\varrho}}_{1}(\fn{y^{\hom}})}$ \fn{and} $\fn{\tilde{\fn{\varrho}}_{2}(\fn{y^{\hom}})}$ analogously. The condition \eqref{y-yprim-models} ensures that $\fn{\tilde{\fn{\varrho}}_{1}(y)} = \fn{\tilde{\fn{\varrho}}_{1}(\fn{y^{\hom}})}$.
With the site energy \eqref{E_EAM}, we have for some $\theta \in (0,1)$
\fn{\begin{align*}
\big| \Phi_{\ell, n} (y) - \Phi_{\ell, n} (\fn{y^{\hom}}) \big| &= \left| \fn{J}'\big(\tilde{\fn{\varrho}}_1(y) + \tilde{\fn{\varrho}}_{2}(y) \big) - \fn{J}'\big(\tilde{\fn{\varrho}}_1(y) + \tilde{\fn{\varrho}}_{2}(\fn{y^{\hom}})\big) \right| \, \big| \tilde{\fn{\varrho}}'( r_{\ell n} (y) ) \big| \\
&\leq \left| \fn{J}''\big(\tilde{\fn{\varrho}}(y) + \theta( \tilde{\fn{\varrho}}_{2}(y) - \tilde{\fn{\varrho}}_{2}(\fn{y^{\hom}}) ) \big) \right| \big| \tilde{\fn{\varrho}}_{2}(y) - \tilde{\fn{\varrho}}_{2}(\fn{y^{\hom}}) \big| \big| \tilde{\fn{\varrho}}'( r_{\ell n} (y) ) \big| \\
& \leq  C \left( \big|\tilde{\fn{\varrho}}_{2}(y)\big| + \big|\tilde{\fn{\varrho}}_{2}(\fn{y^{\hom}})\big| \right) \big|\tilde{\fn{\varrho}}'( r_{\ell n} (y) ) \big| ,
\end{align*} }
where we have used the assumption that $\fn{J}''$ is bounded is used in the last estimate.
Now a necessary condition for \asSEH~is
$\fn{\big|\tilde{\fn{\varrho}}_{2}(y)\big| + \big|\tilde{\fn{\varrho}}_{2}(\fn{y^{\hom}})}\big| \leq C(1+r)^{-s}$ with $s>3/2$.
Using the form of $\fn{\varrho}$ in \eqref{condition-rho}, we therefore require $q>9/2$.

In summary, if $q>9/2$, then the assumptions \asSE~are satisfied by the EAM site energy \eqref{E_EAM} for point defects; and if $q>5$, than \asSE~are satisfied for dislocations.

In particular, for all interatomic potentials with similar form as \eqref{E_EAM},
such as second moment tight binding model \cite{finnis03}, Gupta potential \cite{Gupta1981}, Finnis-Sinclair potential \cite{Finnis1984,Sinclair:1971},
the assumptions \asSE~are satisfied.

\subsection{Tight binding}
\label{sec:tb}
\def\Rc{R_{\rm c}}
Whereas the Lennard--Jones and EAM model are classical interatomic potentials, we
now give examples of three quantum mechnical (QM) models that can be treated
within the assumptions of our work. The first QM model that we consider is the
tight binding model. For simplicity of presentation, we consider an orthogonal
two-centre tight binding model\cite{goringe97} with a single orbital per atom.
All results can be extended directly to general non-self-consistent
non-orthogonal tight binding models with multiple orbitals, as described in
\cite{chen15a}.

For a finite system with reference configuration $\Omega\subset\L_{0}$
and $\fn{|\Omega|} = N$,
the two-centre tight binding model is formulated in terms of a discrete
Hamiltonian, with matrix elements
\begin{eqnarray}\label{tb-H-elements}
\Big(\mathcal{H}^{\Omega}(y)\Big)_{\ell k}
=\left\{ \begin{array}{ll} 
h_{\rm ons}\left(\sum_{j \in \Omega \setminus \ell}
\varrho\big(|y({\ell})-y(j)|\big)\right)
& {\rm if}~\ell=k \\[1ex]
h_{\rm hop}\big(|y(\ell)-y(k)|\big) & {\rm if}~\ell \in \Omega \setminus k,
\end{array} \right.
\end{eqnarray}
where $\Rc$ is a cut-off radius,
$h_{\rm ons} \in C^{\mathfrak{n}}([0, \infty))$ is the on-site term, with
${\varrho \in C^{\mathfrak{n}}([0, \infty))}$, $\varrho = 0$ in $[\Rc,\infty)$,
and $h_{\rm hop} \in C^{\mathfrak{n}}([0, \infty))$ is the hopping term with
$h_{\rm hop}(r)=0$ for all $r\in[\Rc,\infty)$.

Note that $h_{\rm ons}$ and $h_{\rm hop}$ are independent of $\ell$ and $k$,
which indicates that all atoms of the system belong to the same species.
We observe that the formulation~\eqref{tb-H-elements} satisfies all the
assumptions on Hamiltonian matrix elements in \cite[Assumptions H.tb, H.loc, H.sym, H.emb]{chen15a}.

With the above tight binding Hamiltonian $\mathcal{H}^{\Omega}$,
we can obtain the band energy of the system
\begin{eqnarray}\label{e-band}
& E^\Omega(y)=\sum_{s=1}^N \mathfrak{f}(\varepsilon_s), 
& \text{where} \quad
\mathcal{H}^{\Omega}(y)\psi_s = \varepsilon_s\psi_s\quad s=1,2,\fn{\ldots},N.
\end{eqnarray}
A canonical choice of $\mathfrak{}$ for modelling solids is the grand-canonical
potential,
\begin{equation*}
\mathfrak{f}(\varepsilon) = 2 k_{\rm B}T \log(1 - f(\varepsilon)),
\quad \text{where} \quad
f(\varepsilon) = \left( 1+e^{(\varepsilon-\mu)/(k_{\rm B}T)} \right)^{-1},
\end{equation*}
$f$ is the Fermi-Dirac distribution function and $\mu$ the chemical potential
of the crystal (see \cite{chen16a} for more details of its derivation),
$k_{\rm B}$ Boltzmann's constant, and $T>0$
the temperature of the system. Other choices of $\mathfrak{f}$ are also
possible \cite{chen15a}.

Tight binding models typically also include a pairwise repulsive potential,
which can be treated purely classically \cite{chen15a}, hence we do
not consider this.

A crucial assumption justified in \cite{chen16a} is that the chemical potential
$\mu$ is independent of the configuration $y$. Under this assumption, following
\cite{chen15a,finnis03}, we can distribute the energy to each atomic site
\begin{eqnarray}\label{E-El_TB}
E^\Omega(y)=\sum_{\ell\in\Omega} \Phi_{\ell}^{\Omega}(y)
\qquad{\rm with}\qquad
\Phi_{\ell}^\Omega(y) := \sum_{s}\mathfrak{f}(\varepsilon_s)
\left|[\psi_s]_{\ell}\right|^2.
\end{eqnarray}
It is then shown in \cite[Theorem 3.1]{chen15a} that
the site energy on an infinite lattice can be defined via the thermodynamic limit
\fn{\begin{eqnarray}\label{E_TB}
\Phi_{\ell}(y) := \lim_{R\rightarrow\infty} \Phi_{\ell}^{\L_{0} \cap B_R(\ell)}(y), \quad \Phi^{\hom}_{\ell'}(\fn{y^{\hom}}) := \lim_{R\rightarrow\infty} \Phi_{\ell'}^{\Lhom_{0} \cap B_R(\ell')}(\fn{y^{\hom}}),
\end{eqnarray} }
and that the resulting site energy is {\em local}: if $y\in\Adm^{0}_{\frak{m},\lambda}(\L_{0})$ then
\begin{equation} \label{eq:Vlocality_TB}
\big|\partial_{y(n_1)}\cdots \partial_{y(n_j)} \Phi_\ell(y)\big|
\leq C \exp\big(-\gamma \sum_{i = 1}^j |y_{n_i} - y_\ell|\big),
\end{equation}
where $C, \gamma$ depends only on $\mathfrak{m}, \lambda$ but is independent of $y$.
This already justifies \asSEL~with weight functions
$\mathfrak{w}_j(r) = e^{-\gamma/2 r}$ (the factor $1/2$ is due to a subtle
technicality).

Along similar lines, under assumptions of \asSEH, \cite[Lemma B.1]{chen15b} establishes
\begin{equation} \label{eq:Vhomogeneity_TB}
\left| \Phi_{\ell_{1},n_{1}}(Dy_{1}(\ell_{1}))
- \Phi^{\hom}_{\ell_{2},n_{2}}(Dy_{2}(\ell_{2})) \right| \leq
C \exp\big( - \gamma (r + |\ell_1-n_1|) \big),
\end{equation}
where $C, \gamma$ may be chosen the same as in \eqref{eq:Vlocality_TB}
without loss of generality. This establishes \asSEH~with arbitrary $s>0$.

Thus, the tight binding model \eqref{E-El_TB}
fulfills all assumptions \asSE~on the site energy.

\subsection{The Thomas--Fermi--von Weizs\"{a}cker (TFW) Model}
\label{sec:TFW}
The second QM model for which we verify that it fits within
our general framework is the {\em orbital-free} TFW model.

Let $\nu \in C^{\infty}_{\rm c}(\mathbb{R}^{3})$ satisfying $\nu \geq 0$ and $\int_{\mathbb{R}^{3}} \nu = 1$ describe the \emph{smeared} distribution of a single nucleus with unit charge and let $Z(\ell) \in \mathbb{N}$ denote the charge of the nucleus at $\ell \in \L$. We assume that $Z(\ell)$ is constant for all $\ell \in \L_{0} \setminus B_{\Rcore}$, where $\L_{0} \setminus B_{\Rcore} = \Lhom_{0} \setminus B_{\Rcore}$. For $y \in \Adm(\L_{0})$, we define the total nuclear distribution corresponding to $y$ as
\begin{align*}
\varrho^{\nuc}(y;x) = \sum_{\ell \in \L_{0}} Z(\ell) \, \nu(x - y(\ell)).
\end{align*}
As $y \in \Adm(\L)$, it follows immediately that $\varrho^{\nuc} = \varrho^{\nuc}(y; \cdot)$ satisfies
\begin{align*}
\sup_{x \in \mathbb{R}^{3}} \int_{B_{1}(x)} \varrho^{\nuc}(z) \text{ d} z < \infty \quad \text{and}
\quad \lim_{R \to \infty} \inf_{x \in \mathbb{R}^{3}} \int_{B_{R}(x)} \varrho^{\nuc}(z) \text{ d} z = + \infty,
\end{align*}
hence \cite[Theorem 6.10]{CattoLeBrisLions} guarantees the existence
and uniqueness of the corresponding TFW ground state $(\varrho,\phi) \in (C^{\infty}(\mathbb{R}^{3}))^{2}$, $\varrho \geq 0$,
solving the following coupled elliptic system pointwise
\begin{subequations}
	\label{eq:u-phi-eq-pair}
	\begin{align}
	&\left( - \Delta  + \frac{5}{3} \varrho^{2/3} - \phi \right) \sqrt{\varrho} = 0, \\
	& \,\, - \Delta \phi = 4\pi (\varrho^{\nuc} - \varrho).
	\end{align}
\end{subequations}
Here $\varrho$ represents the ground state electron density and $\phi$ the electrostatic potential.
Using $(\varrho,\phi)$, we define the ground state energy density
corresponding to $y$,
\begin{align*}
E^{\TFW}(y;\cdot) = |\nabla \sqrt{\varrho}|^{2} + \varrho^{5/3} + \frac{1}{2} \phi \left( \varrho^{\nuc} - \varrho \right).
\end{align*}
One can view \eqref{eq:u-phi-eq-pair} as the Euler--Lagrange equations corresponding to the
\emph{formal} energy given by $\int_{\mathbb{R}^{3}} E^{\TFW}(y;x) \text{ d}x$.
We then introduce a family of partition functions $\{\chi_\ell(y;\cdot)\}_{\ell \in \L_{0}}$
satisfying $\chi_{\ell}(y;\cdot) \in C^\infty(\mathbb{R}^{3}), \chi_{\ell}(y;\cdot) \geq 0$ and
\begin{subequations}
	\label{eq: site-E-partition-conditions}
	\begin{align}
	\sum_{\ell \in \L_{0}} \chi_{\ell}(y; x) &= 1, \quad \text{and} \label{eq: sum 1 eq} \\
	\left| \frac{\partial^{k} \chi_{\ell}(y;x)}{\partial y(j_{1})\cdots \partial y(j_{k})} \right|
	&\leq C e^{- \gamma|x - y(\ell)|} \prod_{1 \leq i \leq k} e^{- \gamma|x - y(j_{i})|}, \label{eq: partition condition 2}
	\end{align}
\end{subequations}
for some $C, \gamma > 0$ and for all $k \in \N$, $\ell, i_{1}, \ldots, i_{k} \in \L_{0}$ and $x \in \mathbb{R}^{3}$.
The construction of such partition functions is given in \cite[Remark 12, Theorem 5.11]{Nazar-thesis} and \cite[Remark 9]{nazar14}.
One such example is the following: for $\alpha > 0$, let $\wt{\chi}(x) = e^{-\alpha |x|^2}$, then for $\ell \in \L_{0}$, define
\begin{align*}
\chi_\ell(y; x) = \frac{ \wt{\chi}(x-y(\ell))}{ \sum_{\ell' \in \L_{0}} \hat{\chi}(x - y(\ell'))}.
\end{align*}
This construction satisfies all of the required conditions.

We then define the TFW site energy as
\fn{\begin{align}
\Phi_{\ell}(y) := \int_{\mathbb{R}^{3}} E^{\TFW}(y;x) \, \chi_{\ell}(y;x) \text{ d}x, \quad \Phi^{\hom}_{\ell'}(\fn{y^{\hom}}) := \int_{\mathbb{R}^{3}} E^{\TFW}(\fn{y^{\hom}};x) \, \chi_{\ell'}(\fn{y^{\hom}};x) \text{ d}x. \label{E_TFW}
\end{align}}
It is shown in \cite{nazar14} and \cite[Proposition 6.3, Theorem 5.11]{Nazar-thesis}
that this site energy is well-defined and that it satisfies the same locality and homogeneity
estimates \eqref{eq:Vlocality_TB} and \eqref{eq:Vhomogeneity_TB} as the tight
binding model. Thus, all assumptions \asSE~are again satisfied (with {\asSEH}
holding for arbitrary $s>0$) by the TFW site energy \eqref{E_TFW}.

\subsection{Reduced Hartree--Fock with Yukawa potential}
\label{sec:rHF}

\def\pot{c}
\def\yp{m}  
\def\ErHF{E^{\rm rHF}}
\def\IrHF{I^{\rm rHF}}
\def\Vext{V_{\rm ext}}
\def\VH{V_{\rm H}}
\def\Ne{N_{\rm e}}

The third QM model we consider is the {\it reduced Hartree--Fock} model \cite{solovej91}, also called the Hartree model in the physics literature.
Instead of the long-range Coulomb interaction, we assume that all the particles interact through a short-range Yukawa potential \cite{yukawa35}.
Similar to the TFW model in Section \ref{sec:TFW}, we can define the ground state energy density,
and then construct the site energy from a family of partition functions.
Assuming that all atoms of the system belong to the same species,
we have from \cite{chen16b} that the assumptions \asSE~are satisfied (with {\asSEH} holding for arbitrary $s>0$) by the rHF site energy.

\section{Conclusions}
\label{sec:conclusions}
We have developed a general framework to study the geometry relaxation
of crystalline defects embedded in a homogeneous host crystal.
Specifically, we formulated energy difference functionals for
solid systems with point defects and dislocations, and justified the  far-field
decay of equilibrium states under a mild far-field stability assumption.
The novel aspect of our work is that we can incorporate an
infinite interaction neighbourhood
in our framework, in particular we showed that we include a wide range of (simplified)
quantum chemistry models for interatomic interaction.

Our results are interesting in their own right in classifying properties of
crystalline defects, but in addition they provide a foundation for the analysis
of different boundary conditions and atomistic multiscale simulation methods,
see e.g. \cite{chen15b,2013-defects-v3,2013-PRE-bqcfcomp,LiOrtnerShapeevEtAl-blended,OrtnerZhang2014}.

The main restriction of present work is that the locality assumptions in \asSEL~are not satisfied for some {\it long-range} potentials. Note that for \asSEL~in
the dislocations case, our assumption requires that the interaction between two
particles should decay faster than $r^{-2-d}$ with respect to their distance
$r$. This is in particular not satisfied by Coulomb or even di-pole
interactions. One would therefore need to justify (or assume) that such
long-range interactions are screened when we apply the framework and results in
this paper. Even in this setting, however, our results help in that we only
require mild algebraic decay of the screened interaction, which may be easier to
justify than the more common and much more stringent assumption of exponential decay.

\appendix

\section{Equivalent norms}
\label{sec:proof_equiv_norms}
\setcounter{equation}{0}

In addition to the norm $\| D \cdot \|_{\hc{l}^2_{\mathcal{N}}}$ introduced in \eqref{eq: nn norm},
we construct a family of norms that use weighted finite-difference
stencils with infinite interaction range. For $k \in \mathbb{N}$,
$\wf\in\Hw_{k}$, $\L$ satisfying \asRC, $\ell \in \L$ and $u \in \UsH(\L)$,
define
\begin{align}\label{eq:wnormdef}
\big|Du(\ell)\big|_{\wf,k} &:= \bigg( \sum_{\rho\in\L-\ell}
\wf(|\rho|) \big|D_\rho u(\ell)\big|^k \bigg)^{1/k}, \\
\label{eq:norm-def-1}
\|Du\|_{\hc{l}^2_{\wf,k}} &:= \bigg( \sum_{\ell \in \L} |Du(\ell)|_{\wf,k}^2 \bigg)^{1/2} \quad \text{for } k = 1, \\
\label{eq:norm-def-2}
\text{and} \quad \|Du\|_{\hc{l}^k_{\wf,k}} &:= \bigg( \sum_{\ell \in \L} |Du(\ell)|_{\wf,k}^k \bigg)^{1/k} \quad \text{ for } k \geq 2.
\end{align}

To the best of our knowledge, the norms given by \eqref{eq:norm-def-1}--\eqref{eq:norm-def-2} are new constructions, and are vital for our analysis.
The following lemma states the relationships between the norms
$\|D\cdot\|_{\hc{l}^2_{\wf},k}$, $\|D\cdot\|_{\hc{l}^k_{\wf},k}$ and $\|D\cdot\|_{\hc{l}^{2}_{\mathcal{N}}}$.

\begin{lemma}\label{lemma:normEquiv}
	Let $k \in \mathbb{N}$, $\wf \in \Hw_{k}$ and $\L$ satisfy \asRC.
	\begin{enumerate}
		\item Suppose $k = 1,2$ and
		$\inf_{\ell \in \L} \inf_{\rho \in \mathcal{N}(\ell) - \ell} \wf(|\rho|) = c_{0} > 0,$
		then there exist
		constants $c_{k} , C_{k} > 0$ such that
		\begin{eqnarray}\label{eq:normest12}
		c_{k} \|Du\|_{\hc{l}^2_{\mathcal{N}}} \leq \|Du\|_{\hc{l}^2_{\wf,k}} \leq C_{k} \|Du\|_{\hc{l}^2_{\mathcal{N}}}
		\qquad\forall~u\in\UsH(\L).
		\end{eqnarray}

		\item Suppose $k>2$, then there exists $C_{k} > 0$ such that
		\begin{eqnarray}\label{eq:normestk}
		\|Du\|_{\hc{l}^k_{\wf,k}} \leq C_{k} \|Du\|_{\hc{l}^2_{\mathcal{N}}}
		\qquad\forall~u\in\UsH(\L).
		\end{eqnarray}
	\end{enumerate}
\end{lemma}

To prove this lemma, we require the following results.

\fn{ \begin{lemma} \label{lemma:omega-sum-int-est}
Let $k \in \mathbb{N}$, $\wf_{1} \in \Hw_{k}$, $\wf_{2} \in \Hw_{k}^{\log}$ and $\L$ satisfy \asRC. Then there exists $C_{k} > 0$, such that for all $\ell \in \L$
\begin{align*}
\sum_{\rho \in \L - \ell} \wf_{1}(|\rho|) |\rho|^{k} &\leq C_{k} \| \wf_{1} \|_{\Hw_{k}}, \qquad 
\sum_{\rho \in \L - \ell} \wf_{2}(|\rho|) |\rho|^{k} \log^{2}(1+|\rho|) \leq C_{k} \| \wf_{2} \|_{\Hw_{k}^{\log}}.
\end{align*} 
\end{lemma} }

\fn{ \begin{proof}
	Fix $\ell \in \L$ and observe that 
	\begin{align}
	r_{0} = r_{0}(\L) := \min_{\ell_{1} \in \L, \ell_{2} \in \L \setminus \ell_{1}}|\ell_{1} - \ell_{2}| > 0. \label{eq:r0-def}
	\end{align}
	It follows that $\{ B_{r_{0}/2}(\ell') \}_{\ell' \in \L}$ forms a collection of disjoint balls and for $\rho \in \L - \ell$, 
	that $|\rho| \geq r_{0}$, so for $x \in B_{r_{0}/2}(\rho)$, we have $\smfrac{|\rho|}{2} \leq |x| \leq \smfrac{3|\rho|}{2}$ 
	by the triangle inequality. Therefore, using that $\wf_{2} \in \Hw_{k}^{\log}$ is decreasing, we deduce
\begin{align*}
&\sum_{\rho \in \L - \ell} \wf_{2}(|\rho|) |\rho|^{k} \log^{2}(1+|\rho|) \\ 
\leq & \frac{2^{k}}{|B_{r_{0}/2}|} \sum_{\rho \in \L - \ell}  \int_{B_{r_{0}/2}(\rho)} \wf_{2}\left(\frac{2|x|}{3}\right) |x|^{k} \log^{2}(1+2|x|) \textrm{ d}x \\
\leq & \frac{2^{k}}{|B_{r_{0}/2}|} \int_{\R^d} \wf_{2}\left(\frac{2|x|}{3}\right) |x|^{k} \log^{2}(1+2|x|) \textrm{ d}x \\ 
= & C_{k} \int_{0}^{\infty} \wf_{2}(r)r^{k+d-1} \log^{2}(1+r) \textrm{ d}r \leq C_{k} \| \wf_{2} \|_{\Hw_{k}^{\log}}.
\end{align*}
An identical argument shows the estimate for $\wf_{1} \in \Hw_{k}$, completing the proof. 
\end{proof} }

\begin{lemma} \label{lemma:pathcount}
	Let $\L$ satisfy {\asRC}, then for all $\ell \in \L$ and $\rho \in \L - \ell$,
	there exists a finite path of lattice points
	$\mathscr{P}(\ell,\ell+\rho) := \{ \ell_{i} \}_{1\leq i \leq N_{\rho}+1} \subset \L$
	from $\ell$ to $\ell + \rho$, such that for each $1 \leq i \leq N_{\rho}$,
	$\ell_{i+1} \in \mathcal{N}(\ell_{i})$. Moreover, there exists $C > 0$,
	independent of $\ell$ and $\rho$, such that $N_{\rho} \leq C|\rho| $.
\end{lemma}

\begin{lemma} \label{lemma:ballcount}
	Let $\L$ satisfy \asRC.
	For $\ell \in \L$ and $n \in \mathbb{N}$, define
	\begin{align*}
	\mathscr{B}_{n}(\ell) = \left\{ \, (\ell_{1},\ell_{2}) \in \L  \times \L \, \bigg| \, n-1 < |\ell_{1} - \ell_{2}| \leq n, \ell \in \mathscr{P}(\ell_{1},\ell_{2}) \, \right\}.
	\end{align*}
	The paths $\mathscr{P}(\ell_{1},\ell_{2})$
	in Lemma \ref{lemma:pathcount} may be chosen such that the following
	statement is true: there exists $C > 0$ such that
	\begin{align} \label{eq:ballcount}
	|\mathscr{B}_{n}(\ell)| \leq C n^{d} \quad \text{for all } \, \ell \in \L.
	\end{align}
\end{lemma}

\begin{figure}[!htb]
	\centering 
	\subfigure[Construction of $\fn{\bar{\L}}$.]{
		\label{fig:gull}
		\includegraphics[height=4cm]{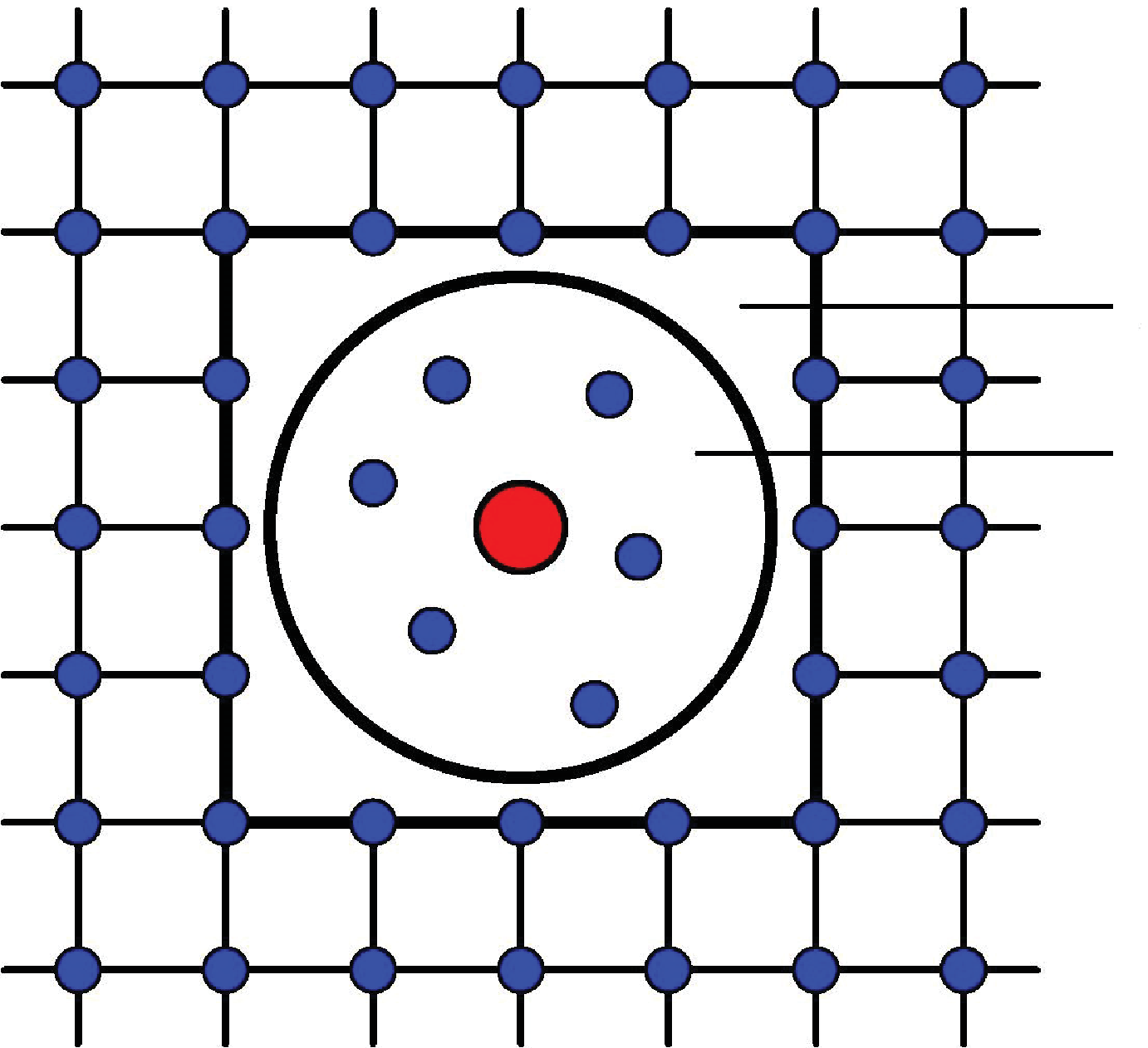}
		\put(-22,80){\makebox(2,2){\footnotesize$\bar{\L}$}}
		\put(-14,62){\makebox(2,2){\footnotesize$B_{\Rcore}$}}
	}
	\subfigure[Examples of paths constructed outside of $\fn{\bar{\L}}$.]{
		\label{fig:tiger} 
		\includegraphics[height=4cm]{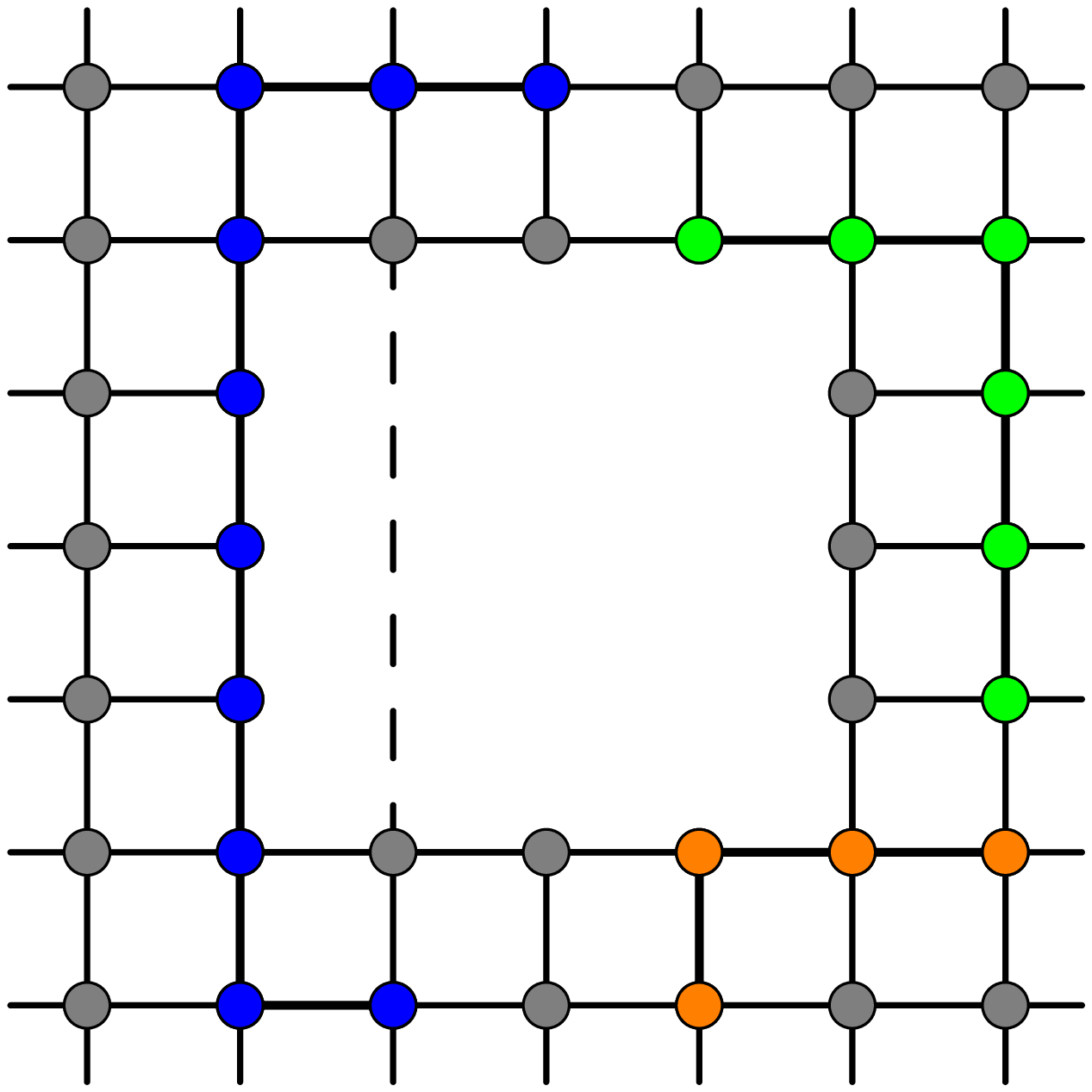}
	}
	\subfigure[Construction of a path inside $\fn{\bar{\L}}$, using a Voronoi cell decomposition.]{
	\label{fig:mouse} 
	\includegraphics[height=4cm]{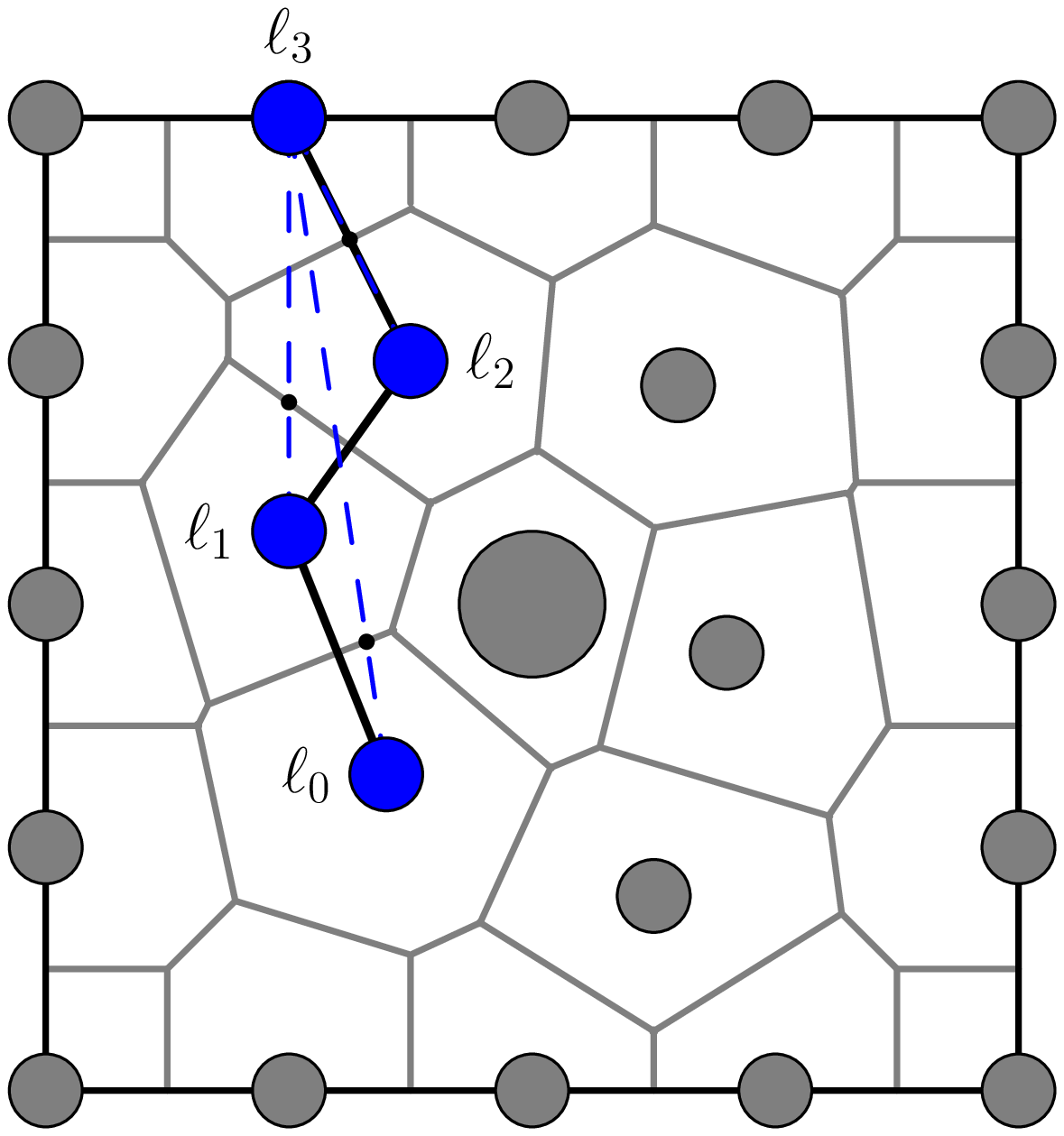}
}
	\caption{Constructions used to prove Lemma \ref{lemma:pathcount}.}
	\label{fig:animals}
\end{figure}


\begin{proof}
	[Proof of Lemma \ref{lemma:pathcount}]
	Let $\Gamma$ denote the unit cell of the lattice \fn{$\Lhom$}, centred at~$0$. 
	Then consider fixed $k \in \mathbb{N}$ large enough such that $k \bar{\Gamma} \supset \hc{B_{\Rcore}}$ (see Figure~2) 
	and define $\fn{\bar{\L}} = \L \cap k \bar{\Gamma}$, which satisfies: $\fn{\bar{\L}}$ is finite, 
	$\fn{\bar{\L}}^{c} := \L \setminus \fn{\bar{\L}} = \L^{\hom} \setminus \fn{\bar{\L}}$ 
	\fn{and $\partial \fn{\bar{\L}} \subset \fn{\bar{\L}} \cap \L^{\hom}$}.

\fn{
	\emph{Case 1} First consider $\ell \in \fn{\bar{\L}}^{c}, \rho \in \fn{\bar{\L}}^{c} - \ell$, then $\rho \in \L^{\hom}$ and can be expressed as $\rho = \sum_{j=1}^{d} n_{j} A e_{j},$ where $(A e_{j})_{1\leq j \leq d}$ are the lattice vectors and $(n_{j}) \in \mathbb{Z}^{d}$. As the lattice vectors are independent, one can define a norm $|\rho|_{1} = \sum_{j=1}^{d} |n_{j}|$, which is equivalent to the standard Euclidean norm $|\rho|$, hence $|\rho|_{1} \leq C |\rho|$, where $C > 0$ is independent of $\rho \in \Lhom$. It is straightforward to construct a lattice path $(\ell_{i})_{1\leq i \leq N_{\rho}} \subset ( \fn{\bar{\L}}^{c} \cup \partial \fn{\bar{\L}} ),$ from $\ell$ to $\ell + \rho$ that avoids the defect core $\L \cap B_{\Rcore}$, such that $\ell_{i+1} \in \{ \ell_{i} \pm A e_{j} \}_{1 \leq j \leq d} \subseteq \mathcal{N}(\ell_{i})$, such that $N_{\rho} \leq 2|\rho|_{1} \leq C |\rho|$.
	
	Moreover, for $1 \leq i \leq d$, define the sets
	\begin{align}
	Q_{1,i} &= \bigg\{ \, \sum_{j=1}^{d} n_{j} Ae_{j} \, \bigg| \, (n_{j}) \in \Z^{d}, n_{i} < - k \text{ and } n_{j} \in [-k,k] \text{ for all } j \neq i \, \bigg\}, \label{eq:A1i-def} \\
	Q_{2,i} &= \bigg\{ \, \sum_{j=1}^{d} n_{j} Ae_{j} \, \bigg| \, (n_{j}) \in \Z^{d}, n_{i} > k \text{ and } n_{j} \in [-k,k] \text{ for all } j \neq i \, \bigg\}. \label{eq:A2i-def}
	\end{align}
	It follows that $\fn{\bar{\L}}^{\rm c} = \bigcup_{i = 1}^{d} (Q_{1,i} \cup Q_{2,i})$. The paths can be chosen such that for $\ell_{1}, \ell_{2} \in \L^{\rm c}_{0}$, then the path $(\ell_{i})_{1\leq i \leq N_{\rho}} \subset ( \fn{\bar{\L}}^{c} \cup \partial \fn{\bar{\L}} )$ from $\ell_{1}$ to $\ell_{2}$ satisfies $(\ell_{i})_{1\leq i \leq N_{\rho}} \cap \fn{\bar{\L}} \neq \emptyset$ if and only if
	\begin{align}
	(\ell_{1},\ell_{2}) \in \bigcup_{i = 1}^{d} \big( (Q_{1,i} \times Q_{2,i}) \cup (Q_{2,i} \times Q_{1,i}) \big). \label{eq:Q12est}
	\end{align} }

	\emph{Case 2} Now consider $\ell \in \fn{\bar{\L}}, \rho \in \fn{\bar{\L}} - \ell$.
	For $\ell' \in \L, \rho' \in \L - \ell'$, then define the Voronoi cell $\mathcal{V}(\ell')$
	\begin{align}
	\mathcal{V}(\ell') &= \left\{ \, \hc{a} \in \mathbb{R}^{d} \, \Big | \, |\hc{a} - \ell'| \leq |\hc{a} - k| \quad \forall \, k \in \L \,  \right\}, \label{eq:Voronoi-def} 
	\end{align}
	and let $\partial \mathcal{V}(\ell')$ denote its boundary. 	Also, define the continuous function ${d_{\ell'}: \mathbb{R}^{d} \to \mathbb{R}_{\geq0}}$ by
   $d_{\ell'}(x) := \max_{k \in \L \setminus \{\ell' \} } \left( |x - \ell'| - |x - k| \right).$
	\fn{ Recall \eqref{eq:r0-def}, that
	\begin{align*}
	r_{0} = \min_{\ell_{1} \in \L, \ell_{2} \in \L \setminus \{\ell_{1} \}}|\ell_{1} - \ell_{2}| > 0,
	\end{align*}
	hence we deduce $d_{\ell'}(\ell') = -r_{0} < 0$ and $d_{\ell'}(k) \geq r_{0} > 0$} for $k \in \L \setminus \{\ell'\}$. 
Also, \fn{we have that} $\mathcal{V}(\ell') = \{ x | d_{\ell'}(x) \leq 0 \}$ and $\partial \mathcal{V}(\ell') = \{ x | d_{\ell'}(x) = 0 \}$. 
For $t \in [0,1]$ define $f(t) = d_{\ell'}(\ell' + t \rho')$, which is continuous and satisfies $f(0) < 0, f(1) > 0$, 
so there exists $t_{0} \in (0,1)$ such that $f(t_{0}) = d_{\ell'}(\ell' + t_{0} \rho') = 0$, i.e. $\ell' + t_{0} \rho' \in \partial \mathcal{V}(\ell')$, 
hence there exists $m = m(\ell',\ell'+\rho') \in \mathcal{N}(\ell')$ such that $t_{0}|\rho'| = |\ell' + t_{0} \rho' - \ell'| = |\ell' + t_{0} \rho' - m|$. 
Applying the triangle inequality gives
	\begin{align}
	|\ell' + \rho' - m| \leq |\ell' + t_{0}\rho' - m| + |(1-t_{0})\rho'| = t_{0}|\rho'| + (1 - t_{0})|\rho'| = |\rho'|. \label{eq:mlesseql}
	\end{align}
	We now show this inequality is actually strict. As the line $L = \{ \ell' + t \rho' | t \in \mathbb{R} \}$ 
	and surface $S = \{ \hc{a} \in \mathbb{R}^{d} ~\big|~ |\hc{a} - \ell'| = |\hc{a} - m| \}$ can only intersect once at 
	$\ell' + t_{0} \rho' \in L \cap S$, then as $\ell' + \rho' \in L \setminus S$, the estimate \eqref{eq:mlesseql} implies
	\begin{align}
	|\ell' + \rho' - m| < |\rho'|. \label{eq:mstrictlessl}
	\end{align}
	We now use \eqref{eq:mstrictlessl} to construct a finite path of neighbouring lattice points from $\ell'$ to $\ell' + \rho'$. Let $\ell_{1}(\ell',\rho') = \ell'$, then for $i \in \mathbb{N}$, given $\ell_{i}(\ell',\rho')$, define $\ell_{i+1}(\ell',\rho') = m(\ell_{i},\ell'+\rho') \in \mathcal{N}(\ell_{i})$, satisfying
	\begin{align}
	|\ell' + \rho' - \ell_{i+1}| < |\ell' + \rho' - \ell_{i}|. \label{eq:liest}
	\end{align}
	We now show that this path reaches $\ell' + \rho'$ after finitely many steps. Observe that $\ell_{1} = \ell' \in \L \cap B_{|\rho'|+1}(\ell'+\rho')$, which is a finite set. From \eqref{eq:liest}, we deduce that $\ell_{i+1} \in \L \cap B_{|\rho'|+1}(\ell'+\rho') \setminus \{\ell_{1},\ldots,\ell_{i}\}$. It follows that the path reaches $\ell' + \rho'$ in $N_{\ell',\rho'}$ steps, where $N_{\ell',\rho'} \leq |\L \cap B_{|\rho'|+1}(\ell'+\rho')|$.

	Now consider the set
	$\left \{ \ell_{i}(\ell,\rho) | \ell \in \fn{\bar{\L}}, \rho \in \fn{\bar{\L}} - \ell, 1 \leq i \leq N_{\ell,\rho} - 1  \right \},$
	which is finite, so $c_{0} > 0$ such that: for all $\ell \in \fn{\bar{\L}}, \rho \in \fn{\bar{\L}} - \ell, 1 \leq i \leq N_{\ell,\rho} -1$, $\ell_{i} = \ell_{i}(\ell,\rho)$ and $\ell_{i+1} = \ell_{i+1}(\ell,\rho)$ satisfy
	$|\ell + \rho - \ell_{i+1}| \leq |\ell + \rho - \ell_{i}| - c_{0}.$
	As $\ell_{1} = \ell$ satisfies $|\ell + \rho - \ell_{1} | = |\rho|$, arguing inductively gives
	$|\ell + \rho - \ell_{i+1}| \leq |\rho| - c_{0} i.$
	Observe that for $i \geq N_{\rho}:= \lceil c_{0}^{-1} |\rho| \rceil$, $|\ell + \rho - \ell_{i+1}| \leq |\rho| - c_{0} i \leq 0$, hence the path reaches $\ell+\rho$ within $N_{\rho} \leq c_{0}^{-1}|\rho| + 1 \leq (c_{0}^{-1} + \mu^{-1})|\rho| = C |\rho|$ steps.

	\emph{Case 3} It remains to consider the case $\ell \in  \fn{\bar{\L}}, \rho \in \fn{\bar{\L}}^{c} \setminus \{ \ell \}$, as the case $\ell \in  \fn{\bar{\L}}^{c}, \rho \in \fn{\bar{\L}} \setminus \{ \ell \}$ is identical. We follow the procedure of Case 2, starting from $\ell_{1} = \ell$ and moving along neighbouring lattice points in $\fn{\bar{\L}}$ until the boundary $\partial \fn{\bar{\L}}$ is reached, hence there exist neighbouring lattice points $\ell_{1},\ldots \ell_{i-1} \in \fn{\bar{\L}} \setminus \partial \fn{\bar{\L}}$ and $\ell_{i} \in \partial \fn{\bar{\L}}$ satisfying $|\ell + \rho - \ell_{i}| \leq |\rho| - c_{0}(i -1)$. As $\ell_{i}, \ell + \rho \in \fn{\bar{\L}}^{c} \cup \partial \fn{\bar{\L}}$, by Case 1, there exists a lattice path $(\ell'_{j})_{1 \leq j \leq N_{i,\rho}} + 1$ along neighbouring lattice points, from $\ell_{i}$ to $\ell+\rho$, satisfying
	$N_{i,\rho} \leq C |\ell+\rho - \ell_{i}| \leq C \left( |\rho| - c_{0}(i-1) \right)$, hence joining these paths creates a lattice path of neighbouring points from $\ell$ to $\ell + \rho$ of length $N_{\rho} + 1$, where $N_{\rho} = i + N_{i,\rho} \leq C( c_{0} i + N_{i,\rho} ) \leq C |\rho|$, where $C$ is independent of $\rho$.
\end{proof}

\fn{ \begin{proof}
		[Proof of Lemma \ref{lemma:ballcount}]
		
		We first recall the subset $\fn{\bar{\L}} \subset \L$ defined in the proof of Lemma \ref{lemma:pathcount}, then decompose $\L^{2}$ into the following sets
		\begin{align*}
		&A_{1} := \fn{\bar{\L}} \times \fn{\bar{\L}},
		\qquad
		A_{2} := \left( \fn{\bar{\L}^{\rm c}} \times \fn{\bar{\L}} \right) \cup \left( \fn{\bar{\L}} \times \fn{\bar{\L}^{\rm c}} \right), \\
		&A_{3} := \big\{ (\ell_{1}, \ell_{2}) \in (\fn{\bar{\L}^{\rm c}})^{2} \, \big| \, \mathscr{P}(\ell_{1},\ell_{2}) \cap \fn{\bar{\L}} = \emptyset \, \big\}, \\
		&A_{4} := \big\{ (\ell_{1}, \ell_{2}) \in (\fn{\bar{\L}^{\rm c}})^{2} \, \big| \, \mathscr{P}(\ell_{1},\ell_{2}) \cap \fn{\bar{\L}} \neq \emptyset \, \big\}, \\
		\end{align*}
		It is clear to see that $\L^{2} = \bigcup_{i =1}^{4} A_i$, hence for all $\ell \in \L$ and $n \in \mathbb{N}$, $|\mathscr{B}_{n}(\ell)| = \sum_{i=1}^{4}|\mathscr{B}_{n}(\ell) \cap A_{i}|$.
		
		We first estimate $|\mathscr{B}_{n}(\ell) \cap A_{1}|$. As $\fn{\bar{\L}}$ is finite, it follows that
		\begin{align}
		|\mathscr{B}_{n}(\ell) \cap A_{1}| \leq |A_{1}| = |\fn{\bar{\L}} \times \fn{\bar{\L}}| \leq |\fn{\bar{\L}}|^{2} n^{d}. \label{eq:bc1}
		\end{align}
		As $\fn{\bar{\L}}^{c} \subset \L^{\hom}$, it follows that
		\begin{align*}
		\mathscr{B}_{n}(\ell) \cap A_{3} \subset \left\{ (\ell_{1},\ell_{1} + \rho) \, \bigg| \ell_{1} \in \fn{\bar{\L}}^{c}, \rho \in \L^{\hom}_{*}, n-1 \leq |\rho| < n, \, \ell \in \mathscr{P}(\ell_{1},\ell_{1} + \rho) \right\}.
		\end{align*}
		For each $\rho \in \L^{\hom}_{*} \cap \hc{B_{n}}$, there exist at most $N_{\rho} + 1 \leq C|\rho| \leq C n$ 
		choices for $\ell_{1} \in \fn{\bar{\L}}^{c}$ such that $\ell \in \mathscr{P}(\ell_{1},\ell_{1} + \rho)$. 
		Consequently,
		\begin{align}
		|\mathscr{B}_{n}(\ell) \cap A_{3}| &\leq C n |\Lhom \cap \left( \hc{B_{n}} \setminus \hc{B_{n-1}} \right) | \nonumber \\
		&\leq C n | \hc{B_{n}} \setminus \hc{B_{n-1}} | \leq C n \left(n^{d} - (n-1)^{d}\right) \leq Cn^{d}. \label{eq:bc3}
		\end{align}
		In order to show the remaining estimates, we require the following estimate. Suppose that $(\ell_{1},\ell_{2}) \in \mathscr{B}_{n}(\ell)$, then using Lemma \ref{lemma:pathcount}, we deduce that
		\begin{align}
		|\ell - \ell_{1}| \leq C |\ell_{1} - \ell_{2}| \max_{\ell' \in \L} \max_{m' \in \mathcal{N}(\ell') } |\ell' - m'| \leq C_{0}n. \label{eq:ell1est}
		\end{align}
		An identical argument also shows that $|\ell - \ell_{2}| \leq C_{0}n$.
		
		Applying \eqref{eq:ell1est}, we now estimate
		\begin{align}
		|\mathscr{B}_{n}(\ell) \cap A_{2} | &\leq 2|(\L^{\hom} \cap B_{C_{0}n}(\ell) ) \times \fn{\bar{\L}} | = | \fn{\bar{\L}} || \L^{\hom} \cap B_{C_{0}n}(\ell) | \nonumber \\ &\leq C | \fn{\bar{\L}}| |B_{C_{0}n}(\ell) | \leq C C_{0}^{d} n^{d}. \label{eq:bc2}
		\end{align}
		For the final estimate, recall the sets $Q_{1,i}, Q_{2,i} \subset \Lhom$, for $1 \leq i \leq d$, defined by \eqref{eq:A1i-def}--\eqref{eq:A2i-def} in the proof of Lemma \ref{lemma:pathcount}. It follows from \eqref{eq:Q12est} that $A_{4} = \bigcup_{i =1}^{d} \big( ( Q_{1,i} \times Q_{2,i} ) \cup ( Q_{2,i} \times Q_{1,i} ) \big )$.
		Now suppose $(\ell_{1},\ell_{2}) \in \mathscr{B}_{n}(\ell) \cap A_{4}$, then without loss of generality suppose that $\ell_{1} \in Q_{1,i}, \ell_{2} \in Q_{2,i}$, for some $1 \leq i \leq d$. By definition, for $i' = 1,2 $, $\ell_{i'} = \sum_{j = 1}^{d} n_{j,i'}Ae_{j}$, where $n_{i,1} < -k, n_{i,2} > k$ and $n_{j,1}, n_{j,2} \in \Z \cap [-k,k]$ for $j \neq i$. Let $c_{0} = \min_{1 \leq i \leq d} |Ae_{i}| > 0$ and observe that $\text{d}(\ell_{1},\fn{\bar{\L}}) = n_{i,1} + k \leq |\ell_{1} - \ell_{2}| \leq n$. Arguing by contradiction, it follows that $n_{i,1} \geq -c_{0}^{-1}n - k$. An identical argument shows that $n_{1,2} \leq c_{0}^{-1}n + k$, so define
		\begin{align*}
		Q_{1,n,i} &= \bigg\{ \, \sum_{j=1}^{d} n_{j} Ae_{j} \, \bigg| \, (n_{j}) \in \Z^{d}, -c_{0}^{-1}n - k \leq n_{i} < - k \text{ and } n_{j} \in [-k,k] \text{ for all } j \neq i \, \bigg\}, \\
		Q_{2,n,i} &= \bigg\{ \, \sum_{j=1}^{d} n_{j} Ae_{j} \, \bigg| \, (n_{j}) \in \Z^{d}, k < n_{i} \leq c_{0}^{-1}n + k \text{ and } n_{j} \in [-k,k] \text{ for all } j \neq i \, \bigg\}.
		\end{align*}
		We then deduce that
		\begin{align}
		|\mathscr{B}_{n}(\ell) \cap A_{4}| &= \sum_{i = 1}^{d} |\mathscr{B}_{n}(\ell) \cap \big( ( Q_{1,i} \times Q_{2,i} ) \cup ( Q_{2,i} \times Q_{1,i} ) \big )| \nonumber \\ \nonumber &= 2 \sum_{i = 1}^{d} |\mathscr{B}_{n}(\ell) \cap ( Q_{1,i} \times Q_{2,i} )| \leq 2 \sum_{i = 1}^{d} |Q_{1,n,i} \times Q_{2,n,i} | \\ & \leq 2 \sum_{i = 1}^{d} |Q_{1,n,i} || Q_{2,n,i} | \leq C d k^{2d -2} n^{2} \leq C n^{d}, \label{eq:bc4}
		\end{align}
		where we have used that $d \in \{2,3\}$. Collecting the estimates \eqref{eq:bc1}, \eqref{eq:bc3}, \eqref{eq:bc2} 
		and \eqref{eq:bc4} gives the desired estimate \eqref{eq:ballcount}
		\begin{align*}
		|\mathscr{B}_{n}(\ell)| = \sum_{i =1 }^{4} |\mathscr{B}_{n}(\ell) \cap A_{i}| \leq Cn^{d}. 
		\end{align*}
	\end{proof}

We now prove Lemma \ref{lemma:normEquiv}.}

\begin{proof}
	[Proof of Lemma \ref{lemma:normEquiv}]
	We show that for $k = 1,2$
	\begin{align}
	c_{k}\|Du\|_{\hc{l}^2_{\mathcal{N}}} \leq \|Du\|_{\hc{l}^2_{\wf,k}}. \label{eq:eqnormlb}
	\end{align}
	From the uniform lower bound on $\wf$,
	for any $\ell \in \L$, $\rho \in \mathcal{N}(\ell) - \ell$, $\wf(|\rho|)^{-1} \leq c_{0}^{-1}$. As $k \leq 2$, using the embedding $\hc{l}^k \subseteq \hc{l}^2$
	\begin{align*}
	\big|Du(\ell)\big|_{\mathcal{N}} &= \bigg( \sum_{\rho\in \mathcal{N}(\ell) - \ell}
	\big|D_\rho u(\ell)\big|^2 \bigg)^{1/2} \leq \bigg( \sum_{\rho\in \mathcal{N}(\ell) - \ell}
	\big|D_\rho u(\ell)\big|^k \bigg)^{1/k} \\
	&\leq \bigg( c_{0}^{-1} \sum_{\rho\in \mathcal{N}(\ell) - \ell} \wf(|\rho|) \big|D_\rho u(\ell)\big|^k \bigg)^{1/k} \leq c_{0}^{-1/k} \big|Du(\ell)\big|_{\wf,k}.
	\end{align*}
	This implies \eqref{eq:eqnormlb} as
	\begin{align*}
	\|Du\|_{\hc{l}^2_{\mathcal{N}}} = \bigg( \sum_{\ell \in \L} |Du(\ell)|_{\mathcal{N}}^2 \bigg)^{1/2} \leq c_{0}^{-1/k} \bigg( \sum_{\ell \in \L} |Du(\ell)|_{\wf, k}^2 \bigg)^{1/2} = c_{0}^{-1/k} \|Du\|_{\hc{l}^2_{\wf,k}}.
	\end{align*}
	
	We now show \eqref{eq:normestk} \fn{for $k \in \N, k > 2$ and the upper bound of \eqref{eq:normest12} for $k = 2$}. By Lemma \ref{lemma:pathcount}, for each $\ell \in \L$ and $\rho \in \L - \ell$, there exists a path $\mathscr{P}(\ell,\ell+\rho) = \{\ell_{i} \in \L |1\leq i \leq N_{\rho}+1\}$ of neighbouring lattice points, such that $N_{\rho} \leq C |\rho|$ and $\rho_{i} := \ell_{i+1} - \ell_{i} \in \mathcal{N}(\ell_{i}) - \ell_{i}$ for all $1 \leq i \leq N_{\rho}$, satisfying $|D_{\rho} u(\ell)| \leq \sum_{i = 1}^{N_{\rho}} |D_{\rho_{i}} u(\ell_{i})|.$
	Applying H{\"o}lder's inequality gives
	\begin{align*}
	|D_{\rho} u(\ell)|^{k} \leq \bigg( \sum_{i = 1}^{N_{\rho}} |D_{\rho_{i}} u(\ell_{i})| \bigg)^{k} \leq |N_{\rho}|^{k-1} \sum_{i = 1}^{N_{\rho}} |D_{\rho_{i}} u(\ell_{i})|^{k} \leq C |\rho|^{k-1} \sum_{i = 1}^{N_{\rho}} |D_{\rho_{i}} u(\ell_{i})|^{k},
	\end{align*}
    so we deduce
	\begin{align}
	\|Du\|_{\hc{l}^k_{\wf,k}} &= \bigg( \sum_{\ell \in \L} |Du(\ell)|_{\wf,k}^k \bigg)^{1/k} = \bigg( \sum_{\ell \in \L} \sum_{\rho \in \L - \ell} \wf(|\rho|) |D_{\rho}u(\ell)|^k \bigg)^{1/k} \nonumber \\
	&\leq C \bigg( \sum_{\ell \in \L} \sum_{\rho \in \L - \ell} \wf(|\rho|) |\rho|^{k-1} \sum_{i = 1}^{N_{\rho}} |D_{\rho_{i}} u(\ell_{i})|^{k} \bigg)^{1/k} \nonumber \\
	&\leq C \bigg( \sum_{\ell \in \L} \sum_{\rho \in \L - \ell} \wf(|\rho|) |\rho|^{k-1} \sum_{i = 1}^{N_{\rho}} \sum_{\rho' \in \mathcal{N}(\ell_{i}) - \ell_{i}} |D_{\rho'} u(\ell_{i})|^{k} \bigg)^{1/k} \nonumber \\
	&= C \bigg( \sum_{\ell' \in \L} \sum_{\rho' \in \mathcal{N}(\ell') - \ell'} |D_{\rho'} u(\ell')|^{k} \sum_{\ell \in \L} \sum_{\substack{\rho \in \L - \ell\\ \ell' \in \mathscr{P}(\ell,\ell+\rho)}} \wf(|\rho|) |\rho|^{k-1} \bigg)^{1/k} \label{eq:normest1}
	\end{align}
	By Lemma \ref{lemma:ballcount}, for all $\ell' \in \L$ and $n \in \mathbb{N}$, the set
	\begin{align*}
	\mathscr{B}_{n}(\ell') = \big\{ (\ell_{1},\ell_{2}) \in \L^{2} \,\big|\, n-1 < |\ell_{1} - \ell_{2}| \leq n, \ell' \in \mathscr{P}(\ell_{1},\ell_{2}) \big\}
	\end{align*}
	satisfies $|\mathscr{B}_{n}(\ell')| \leq C n^{d}.$
	As $\wf$ is a decreasing function, it follows that
	\begin{align}
	&\sum_{\ell \in \L} \sum_{\substack{\rho \in \L - \ell\\ \ell' \in \mathscr{P}(\ell,\ell+\rho)}} \wf(|\rho|) |\rho|^{k-1}
	= \sum_{n=1}^{\infty} \sum_{\ell_{1}, \ell_{2} \in \mathscr{B}_{n}(\ell')} \wf(|\ell_{1} - \ell_{2}|) |\ell_{1} - \ell_{2}|^{k-1} \nonumber \\
	&\fn{\leq C \sum_{n=1}^{\infty} \wf\left( n-1 \right) n^{k+d-1} \leq (1 + 2^{k+d-1}) \|\wf\|_{L^{\infty}([0,\infty))} + C \sum_{n=3}^{\infty} \wf\left( n-1 \right) (n-2)^{k+d-1}} \nonumber \\
	&\leq \fn{C \bigg( \| \wf \|_{\Hw_{k}} +  \sum_{n=1}^{\infty} \wf\left( n+1 \right) n^{k+d-1} \bigg) \leq C \bigg( \| \wf \|_{\Hw_{k}} + \int_{1}^{\infty} \wf(x) x^{k+d-1} \text{d} x \bigg) \leq C \| \wf \|_{\Hw_{k}}.} \label{eq:pathwfest}
	\end{align}
	Combining \eqref{eq:normest1}--\eqref{eq:pathwfest} and applying the embedding \fn{$\hc{l}^2 \subseteq \hc{l}^k,$ as $k \geq 2$}, yields
	\begin{align*}
	\|Du\|_{\hc{l}^k_{\wf,k}} &\leq C \bigg( \sum_{\ell' \in \L} \sum_{\rho' \in \mathcal{N}(\ell') - \ell'} |D_{\rho'} u(\ell')|^{k} \sum_{\ell \in \L} \sum_{\substack{\rho \in \L - \ell\\ \ell' \in \mathscr{P}(\ell,\ell+\rho)}} \wf(|\rho|) |\rho|^{k-1} \bigg)^{1/k} \\
	&\leq C \| \wf \|_{\Hw_{k}}^{1/k} \bigg( \sum_{\ell' \in \L} \sum_{\rho' \in \mathcal{N}(\ell') - \ell'} |D_{\rho'} u(\ell')|^{k} \bigg)^{1/k} \\
	&\leq C \bigg( \sum_{\ell' \in \L} \sum_{\rho' \in \mathcal{N}(\ell') - \ell'} |D_{\rho'} u(\ell')|^{2} \bigg)^{1/2} = C \| Du \|_{\hc{l}^2_{\mathcal{N}}}.
	\end{align*}
	
	It remains to show that for $\wf \in \Hw_{1}$, $\|Du\|_{\hc{l}^2_{\wf,1}} \leq C \|Du\|_{\hc{l}^2_{\mathcal{N}}}$. Define for \fn{$r \geq 0$, $\widetilde{\wf}(r) = \wf(r)(1+r)^{-1} \in \Hw_{2}$}, which satisfies \fn{$\| \widetilde{\wf} \|_{\Hw_{2}} \leq \| \wf \|_{\Hw_{1}}$}. Applying Cauchy--Schwarz \fn{and Lemma \ref{lemma:omega-sum-int-est}} gives
\fn{	\begin{align*}
	|Du(\ell)|_{\wf,1} &= \sum_{\rho\in\L-\ell} \wf(|\rho|) \big|D_\rho u(\ell)\big|
\\ &= \sum_{\rho\in\L-\ell} \left( \wf(|\rho|)^{1/2} (1+ |\rho|)^{1/2} \right) \left( \wf(|\rho|)^{1/2} (1+ |\rho|)^{-1/2} \big|D_\rho u(\ell)\big| \right) \\
	& \leq \bigg( \sum_{\rho\in\L-\ell} \wf(|\rho|) (1+ |\rho|) \bigg)^{1/2} \bigg( \sum_{\rho\in\L-\ell} \wf(|\rho|) (1 + |\rho|)^{-1} \big|D_\rho u(\ell)\big|^{2} \bigg)^{1/2} \\
	&\leq C \| \wf \|_{\Hw_{1}}^{1/2} |Du(\ell)|_{\widetilde{\wf},2}.
	\end{align*} }
	The desired estimate follows from applying \fn{\eqref{eq:normest12} with $k = 2$}
	\begin{align*}
	\|Du\|_{\hc{l}^2_{\wf,1}} 
	&= \bigg( \sum_{\ell \in \L} |Du(\ell)|_{\wf,1}^2 \bigg)^{1/2} \leq \fn{C} \| \wf \|_{\Hw_{1}}^{1/2} \bigg( \sum_{\ell \in \L} |Du(\ell)|_{\widetilde{\wf},2}^2 \bigg)^{1/2} \\ 
	&= \fn{C} \| \wf \|_{\Hw_{1}}^{1/2} \|Du\|_{\hc{l}^2_{\widetilde{\wf},2}} \leq C \| \wf \|_{\Hw_{1}}^{1/2} \| \widetilde{\wf} \|_{\Hw_{2}}^{1/2} \|Du\|_{\hc{l}^2_{\mathcal{N}}} = C \| \wf \|_{\Hw_{1}} \|Du\|_{\hc{l}^2_{\mathcal{N}}}.
	\end{align*}
	This completes the proof.
\end{proof}

\section{Proofs: Energy difference functionals}
\label{sec:proof_energy_diff}
\setcounter{equation}{0}

\subsection{Proof of Lemma \ref{lemma-HcHdense}}
\label{Subsection-Proof of Lemma HcHdense}

The main aim of this Appendix is to prove Theorem \ref{theorem:E-Wc}. 
We first prove Lemma \ref{lemma-HcHdense} to justify that 
$\mathscr{H}^{\rm c}(\L)$ is dense in $\mathscr{H}(\L)$ with respect to the $\|D\cdot\|_{\hc{l}^2_{\mathcal{N}}(\L)}$ norm.

\begin{proof}[Proof of Lemma \ref{lemma-HcHdense}]
We first prove the desired result in the point defects setting, in which case \eqref{eq:H-Lambda-equiv-def} can be stated as
\begin{align}\label{eq:H-Lambda-pd-def}
\mathscr{H}(\L) =  \big\{ \, u \in \UsH(\L) \, \big| \,\,\,
|(y_{0}+u)(\ell) - (y_{0}+u)(m)| \neq 0 \quad\forall~ \ell, m \in \L\fn{, \text{ s.t. } \ell \neq m} \big\}.
\end{align}
Also, we have that $y_{0}(\ell) = \ell$ for all $\ell \in \L$ and that \eqref{eq:Iy0-ass} holds with $I\wt y_{0}(x) = x$ for $x \in \R^{\ds}$. Assuming that \eqref{eq:H-Lambda-pd-def} holds, we now justify that $\mathscr{H}(\L)$ is open with respect to the~{$\|D\cdot\|_{\hc{l}^2_{\mathcal{N}}(\L)}$}~norm. 

Without loss of generality, let $u \in \mathscr{H}(\L)$, $v \in \UsH(\L)$ with ${\| Dv \|_{\hc{l}^2_{\mathcal{N}}(\L)} = 1}$ and let $t > 0$, then there exists $\frak{m} > 0$ such that $y = y_{0} + u$ satisfies ${|y(\ell) - y(m)| \geq \frak{m}|\ell - m|}$ for all $\ell, m \in \L$.
Consider $\ell, m \in \L$ and let $\rho = m - \ell$. By Lemma \ref{lemma:pathcount}, there exists a path $\mathscr{P}(\ell,\ell+\rho) = \{\ell_{i} \in \L |1\leq i \leq N_{\rho}+1\}$ of neighbouring lattice points, such that $N_{\rho} \leq C |\rho|$ and $\rho_{i} := \ell_{i+1} - \ell_{i} \in \mathcal{N}(\ell_{i}) - \ell_{i}$ for all $1 \leq i \leq N_{\rho}$, satisfying
	\begin{align}
	|D_{\rho} v(\ell)| 
	&\leq \sum_{i = 1}^{N_{\rho}} |D_{\rho_{i}} v(\ell_{i})| \leq N_{\rho}^{1/2} \bigg( \sum_{i = 1}^{N_{\rho}} |D_{\rho_{i}} v(\ell_{i})| \bigg)^{1/2} \leq C|\rho|^{1/2} =: C_{0}|\ell - m|^{1/2}, \label{eq:Dv-CS-est}
	\end{align}
	where we have applied Cauchy--Schwarz. Recall \eqref{eq:r0-def}, that $r_{0} = r_{0}(\L) := \min_{\ell_{1} \in \L, \ell_{2} \in \L \setminus \ell_{1}}|\ell_{1} - \ell_{2}| > 0,$ then for $t \leq t_{0}(\frak{m}) := \frac{r_{0}^{1/2} \mathfrak{m}}{2C_{0}}$, applying the triangle inequality implies
	\begin{align}
	&|y(\ell) + tv(\ell) - y(m) - tv(m) | \nonumber \\ &\geq |y(\ell) - y(m)| - t|D_{\rho}v(\ell)| \geq \mathfrak{m}|\ell - m| - C_{0} t |\ell - m|^{1/2} \geq \frac{\mathfrak{m}}{2}|\ell - m|, \label{eq:u-v-open-calc}
	\end{align}
	for all $\ell, m \in \L$, hence \eqref{eq:H-Lambda-pd-def} implies $u + tv \in \mathscr{H}(\L)$ for all $t \in (0,t_{0}]$. It follows that $\mathscr{H}(\L)$ is open with respect to the $\|D\cdot\|_{\hc{l}^2_{\mathcal{N}}(\L)}$ norm. Consequently, by applying Lemma~\ref{lemma-WcW12dense} we also deduce that $\mathscr{H}^{\rm c}(\L)$ is dense in $\mathscr{H}(\L)$ with respect to the $\|D\cdot\|_{\hc{l}^2_{\mathcal{N}}(\L)}$ norm.
	
	It remains to show that \eqref{eq:H-Lambda-pd-def} holds. Clearly, if $u \in \mathscr{H}(\L)$, then $u \in \UsH(\L)$ and $(y_{0}+u)(\ell) \neq (y_{0}+u)(m)$ for any $\ell \in \L, m \in \L \setminus \ell$. We now prove the converse.
	
Recall that $y_{0}(\ell) = \ell$ for $\ell \in \L$ and since $u\in \UsH(\L)$, there exists $R_{1} > 0$ such that $\left( \sum_{\ell \in \L \cap B_{R_{1}}^{\rm c}} \sum_{\rho \in \mathscr{N}(\ell) - \ell} |D_{\rho}u(\ell)|^{2} \right)^{1/2} \leq t_{0}(1)$. Therefore, for $\ell, m \in \L \cap B_{R_{1}}^{\rm c}$, we apply Lemma~\ref{lemma:pathcount} and \eqref{eq:Dv-CS-est} to obtain $|u(\ell) - u(m)| \leq C_{0} t_{0} |\ell - m|^{1/2} \leq \smfrac{|\ell - m|}{2},$
hence
\begin{align}
|y_{0}(\ell) + u(\ell) - y_{0}(m) - u(m) | &\geq |\ell - m| - |u(\ell) - u(m)|
\geq \frac{|\ell - m|}{2}. \label{eq:u-v-open-calc-2}
\end{align}
Then for $\ell, m \in \L$ satisfying $|\ell - m| \geq R_{2} := (2C_{0}\|Du\|_{l^{2}(\L)})^{2}$, a similar argument gives
\begin{align}
&|y_{0}(\ell) + u(\ell) - y_{0}(m) - u(m) | 
\geq |\ell - m| - C_{0} \|Du\|_{l^{2}(\L)} |\ell - m|^{1/2} \geq \frac{|\ell - m|}{2}. \label{eq:u-v-open-calc-3}
\end{align}
Now let $R_{3} = 2\max\{R_{1},R_{2}\} > 0$, so the estimates \eqref{eq:u-v-open-calc-2}--\eqref{eq:u-v-open-calc-3} hold when either $\ell, m \in \L \cap B_{R_{3}/2}^{\rm c}$ or when $\ell \in \L \cap B_{R_{3}/2}, m \in \Lambda \cap B_{R_{3}}^{\rm c}$. We finally consider the case $\ell, m \in \L \cap B_{R_{3}}$. Recall the initial assumption that $(y_{0}+u)(\ell) \neq (y_{0}+u)(m)$ for all $\ell, m \in \L$, then as $\L \cap B_{R_{3}}$ is finite, we have ${\min_{\ell, m \in \L \cap B_{R_{3}}, \ell \neq m} |(y_{0}+u)(\ell) - (y_{0}+u)(m)| =: c_{1} > 0}$, so
\begin{align}
|(y_{0}+u)(\ell) - (y_{0}+u)(m)| \geq c_{1} \geq \frac{c_{1}}{2R_{3}} |\ell - m|. \label{eq:u-v-open-calc-4}
\end{align}
Collecting \eqref{eq:u-v-open-calc-2}--\eqref{eq:u-v-open-calc-4} gives the desired inequality for $\frak{m} = \min\{\frac{1}{2},\frac{c_{1}}{2R_{3}}\} > 0$.

We now show that there exists $\lambda > 0$ such that $\cup_{\ell \in \L} B_{\lambda}(y(\ell)) = \R^{d}$. 

Recall that $u_{0} \equiv 0$, so we have $y(\ell) = \ell + u(\ell)$, for $\ell \in \L$. As $u \in \UsH(\L)$ 
it follows that for any $\varepsilon > 0$ there exists $R = R(\varepsilon) \geq \Rcore > 0$ such that for all $\ell \in \L \cap B_{R}^{\rm c}$ , we have $\max_{\ell' \in \mathscr{N}(\ell) \setminus \ell, \rho \in \mathscr{N}(\ell') - \ell'}|D_{\rho}u(\ell')| \leq \varepsilon$, where the value of $\varepsilon$ will be specified shortly. Let $\Gamma$ denote the unit cell of the lattice \fn{$\Lhom$}, centred at~$0$. 
Then consider fixed $k \in \mathbb{N}$, odd and sufficiently large such that $\Omega_{1} := k \bar{\Gamma} \supset \hc{B_{R}}$, which satisfies $\L \cap \Omega_{1}^{\rm c} = \Lhom \cap \Omega_{1}^{\rm c}$. Additionally, observe that $\Omega_{1}^{\rm c} = \cup_{\ell \in \L \cap \Omega_{1}^{\rm c}} (\Gamma + \ell)$. 
Using \cite{2013-defects-v3,ortner_shapeev12}, we construct an interpolant $Iu: \Omega_{1}^{\rm c} \to \R^{d}$ satisfying ${Iu \in W^{1,2}_{\rm loc}(\Omega_{1}^{\rm c})}$, $Iu(\ell) = u(\ell)$ for all ${\ell \in \L \cap \Omega_{1}^{\rm c}}$ and 
\begin{align*}
\| \nabla Iu \|_{L^{\infty}(\Gamma + \ell)} &\leq C_{1} \max_{\ell' \in \mathscr{N}(\ell) \setminus \ell, \rho \in \mathscr{N}(\ell') - \ell'} |D_{\rho}u(\ell')|,
\end{align*}
for all $\ell \in \L \cap \Omega_{1}^{\rm c}$, where $C_{1} > 0$ is independent of $\ell$ and $\varepsilon$. We choose ${\varepsilon = (4C_{1})^{-1} > 0}$ which ensures that $\| \nabla Iu \|_{L^{\infty}(\Gamma + \ell)} \leq \smfrac{1}{4}$ for all $\ell \in \L \cap \Omega_{1}^{\rm c}$. Now suppose $x_{1}, x_{2} \in \Gamma + \ell$ for some $\ell \in \L \cap \Omega_{1}^{\rm c}$, we have
\begin{align}
|Iu(x_{1}) - Iu(x_{2})| \leq \frac{|x_{1} - x_{2}|}{4}, \label{eq:Iu-est-1}
\end{align}
For general $x_{1}, x_{2} \in \Omega_{1}^{\rm c}$, we can construct a path from $x_{1}, x_{2}$ lying in $\Omega_{1}^{\rm c}$ of length at most $2|x_{1} - x_{2}|$, then applying \eqref{eq:Iu-est-1} yields
\begin{align}
|Iu(x_{1}) - Iu(x_{2})| \leq \frac{|x_{1} - x_{2}|}{2}. \label{eq:Iu-est-2}
\end{align}
We also define an interpolant for $y$ by $Iy(x) := x + Iu(x)$, for $x \in \Omega_{1}^{\rm c}$, with \eqref{eq:Iu-est-2} implying that
\begin{align*}
\frac{|x_{1} - x_{2}|}{2} \leq |Iy(x_{1}) - Iy(x_{2})| \leq \frac{3|x_{1} - x_{2}|}{2},
\end{align*}
for all $x_{1}, x_{2} \in \Omega_{1}^{\rm c}$.
Therefore, the restriction of $Iy$ onto it's image is a homeomorphism. It follows that there exists bounded $\Omega_{2} \subset \R^{d}$ such that $Iy(\Omega_{1}^{\rm c}) = \Omega_{2}^{\rm c}$. So for any $x \in \Omega_{2}^{\rm c}$, there exists a unique $x' \in \Omega_{1}^{\rm c}$ such that $Iy(x') = x$. Additionally, there exists $\ell' \in \L \cap \Omega_{1}^{\rm c}$ such that $|\ell' - x'| \leq C_{2}$, where the constant $C_{2}$ is independent of $x' \in \Omega_{1}^{\rm c}$. Consequently, we deduce that
\begin{align*}
|y(\ell') - x | = |Iy(\ell') - Iy(x')| \leq \frac{3|\ell' - x'|}{2} \leq \frac{3C_{2}}{2}.
\end{align*}
Then, let $\ell \in \L$ minimise $\min_{\ell' \in \L} \textrm{dist}(y(\ell'),\Omega_{2}) =: C_{3} \geq 0$ and observe that the diameter $\textrm{diam}(\Omega_{2})$ is finite as $\Omega_{2}$ is bounded. Therefore, for any $x \in \Omega_{2}$, we have
\begin{align*}
|y(\ell) - x | \leq C_{3} + \textrm{diam}(\Omega_{2}),
\end{align*}
therefore choosing $\lambda := \max\{ \smfrac{3C_{2}}{2}, C_{3} + \textrm{diam}(\Omega_{2}) \} > 0$ ensures ${\cup_{\ell \in \L} B_{\lambda}(y(\ell)) = \R^{d}}$, completing the proof of \eqref{eq:H-Lambda-pd-def} for point defects.

We now consider the dislocations setting, in which case \eqref{eq:H-Lambda-equiv-def} can be stated as
\begin{align}\label{eq:H-Lambda-dis-def}
\mathscr{H}(\L) =  \big\{ \, u = y - y_{0} \in \UsH(\L) \, \big| \,\,\,
|\wt y(\ell) - \wt y(m)| \neq 0 \quad\forall~ \ell, m \in \L_{0} \fn{, \text{ s.t. } \ell \neq m} \big\},
\end{align}
where for $y = y_{0} + u$, $\wt y: \L_{0} \to \R^{3}$ is defined in \eqref{eq:wt y from y def} by $\wt y(\ell) = \ell + (u_{0} + u)(\ell_{12})$, for $\ell = (\ell_{1},\ell_{2},\ell_{3}) \in \L_{0}$ and $\ell_{12} := (\ell_{1}, \ell_{2}) \in \L$. We follow the argument presented in the point defects case, with the following modifications. We begin by assuming \eqref{eq:H-Lambda-dis-def} holds and justify that $\mathscr{H}(\L)$ is open with respect to the~{$\|D\cdot\|_{\hc{l}^2_{\mathcal{N}}(\L)}$}~norm.

Without loss of generality, let $u \in \mathscr{H}(\L)$, $v \in \UsH(\L)$ satisfying $\| Dv \|_{\hc{l}^2_{\mathcal{N}}(\L)} = 1$ and let $t > 0$, then $\wt y(\ell) = \wt y_{0}(\ell) + u(\ell_{12})$ for $\ell \in \L_{0}$ satisfies $\wt y \in \Adm^{0}(\L_{0})$.
For $\ell, m \in \L_{0}$, we apply Lemma~\ref{lemma:pathcount} and follow the argument~\eqref{eq:Dv-CS-est} to obtain
\begin{align}
|v(\ell_{12}) - v(m_{12})| 
&\leq C_{0}|(\ell - m)_{12}|^{1/2} \leq C_{0}|\ell - m|^{1/2}. \label{eq:Dv-CS-dis-est}
\end{align}
Following the argument \eqref{eq:u-v-open-calc} verbatim, it follows that $u + tv \in \mathscr{H}(\L)$, for sufficiently small $t > 0$, hence $\mathscr{H}(\L)$ is open. Consequently, applying Lemma~\ref{lemma-WcW12dense} implies that $\mathscr{H}^{\rm c}(\L)$ is dense in $\mathscr{H}(\L)$ with respect to the $\|D\cdot\|_{\hc{l}^2_{\mathcal{N}}(\L)}$ norm.

It remains to show that \eqref{eq:H-Lambda-dis-def} holds, so we prove the non-trivial inclusion: if $u = y - y_{0} \in \UsH(\L)$ and $\wt y(\ell) \neq \wt y(m)$ for any $\ell \in \L_{0}, m \in \L_{0} \setminus \ell$, then $u \in \mathscr{H}(\L)$ or equivalently $\wt y \in \Adm^{0}(\L_{0})$.

First, for $R > 0$, define the two-dimensional disc
\begin{align*}
D_{R} := \{ (x_{1},x_{2}) \in \mathbb{R}^{2}\, | \, (x_{1}^{2} + x_{2}^{2})^{1/2} < R \, \}.
\end{align*}
Recall our initial assumption \eqref{eq:Iy0-ass}, that there exist $\frak{m}, M, R_{0} > 0$ and an interpolation $I\wt y_{0}: D_{R_{0}}^{\rm c} \times \R \to \R^{\ds}$ satisfying $I\wt y_{0}(\ell) = \wt y_{0}(\ell)$ for $\ell \in \L_{0} \cap (D_{R_{0}}^{\rm c} \times \R)$ such that for all $x_{1}, x_{2} \in D_{R_{0}}^{\rm c} \times \R$
\begin{align} \label{eq:Iy0-ass-rep}
\frak{m}|x_{1} - x_{2}| \leq |I\wt{y}_{0}(x_{1}) - I\wt{y}_{0}(x_{2})| \leq M|x_{1} - x_{2}|.
\end{align}
Without loss of generality, we suppose $\frak{m} \in (0,1)$. As $u\in \UsH(\L)$, there exists $R_{1} \geq R_{0} > 0$ such that 
\begin{align*}
\bigg( \sum_{\ell' \in \L \cap D_{R_{1}}^{\rm c}} \sum_{\rho \in \mathscr{N}(\ell') - \ell'} |D_{\rho}u(\ell')|^{2} \bigg)^{1/2} \leq t_{0}(\frak{m}),
\end{align*} 
therefore, for $\ell, m \in \L_{0} \cap ( D_{R_{1}}^{\rm c} \times \R)$, arguing as in \eqref{eq:u-v-open-calc-2} gives
\begin{align}
|\wt y_{0}(\ell) + u(\ell_{12}) - \wt y_{0}(m) - u(m_{12})| \geq \frac{\mathfrak{m}}{2}|\ell - m|. \label{eq:u-v-open-calc-dis-2}
\end{align}
We now consider the case ${\ell \in D_{R_{1} + R_{2}}^{\rm c} \times \R}$, $m \in D_{R_{1}} \times \R$, where the value of $R_{2} > 0$ will be specified shortly. Let $C'_{1} := \min \{ \, |n_{12}| \, | \, n \in \L_{0} \cap ( D_{R_{1}}^{\rm c} \times \R ) \} > R_{1}$ and fix $\ell' \in \L_{0}$ to be one such minimiser. As $\L_{0}$ is periodic in the ${e_{3}\text{-direction}}$, for all $j \in \Z$, we have ${\ell'_{j} := \ell' + je_{3} \in \L_{0}}$, hence we may choose $k \in \Z$ such that $|(\wt y_{0}(\ell'_{k}) - \wt y_{0}(m))_{3}| \leq 1$. Now define $C'_{2} := \max \{ \, |\wt y_{0}(n)_{12}| \, | \, n \in \L_{0} \cap ( \bar{D}_{C'_{1}} \times \R) \}$ and observe that due to the periodicity of $\L_{0}$, there exists a constant $C'_{3} > 0$ such that $|(\ell'_{k} - m)_{3}| \leq C'_{3}$, for all $m \in \L_{0} \cap ( D_{R_{1}} \times \R)$. We now define ${R_{2} := \max\{ C'_{1} - R_{1} + \smfrac{8(2C'_{2} +1)}{\frak{m}}, 2(6R_{1} + 7C'_{1}), 14C'_{3}, \Big(\frac{4C_{0}\|Du\|_{l^{2}(\L)}}{\frak{m}}\Big)^{2} \}}$. We first estimate
\begin{align*}
|\wt y_{0}(\ell'_{k}) - \wt y_{0}(m)| \leq |(\wt y_{0}(\ell'_{k}) - \wt y_{0}(m))_{12}| +|(\wt y_{0}(\ell'_{k}) - \wt y_{0}(m))_{3}| \leq 2C'_{2} + 1,
\end{align*}
then using the choice of $R_{2}$, \eqref{eq:Iy0-ass-rep} and the triangle inequality implies
\begin{align}
|\wt y_{0}(\ell) - \wt y_{0}(m)| &\geq |\wt y_{0}(\ell) - \wt y_{0}(\ell'_{k})| - | \fn{\wt y}_{0}(\ell'_{k}) - \fn{\wt y}_{0}(m) | \nonumber \\ &\geq \frak{m}|\ell - \ell'_{k}| - (2C'_{2} + 1) \geq \frac{7\frak{m}}{8} |\ell - \ell'_{k}| \geq \frac{3\frak{m}}{4} |\ell - m|. \label{eq:y0-3m4-est}
\end{align}
Following the argument \eqref{eq:u-v-open-calc-3} then gives
\begin{align}
&|\wt y_{0}(\ell) + u(\ell_{12}) - \wt y_{0}(m) - u(m_{12}) | \nonumber \\ & \qquad \qquad \geq \frac{3\frak{m}}{4}|\ell - m| - C_{0} \|Du\|_{l^{2}(\L)} |(\ell - m)_{12}|^{1/2} \geq \frac{\mathfrak{m}}{2}|\ell - m|. \label{eq:u-v-open-calc-dis-3}
\end{align}
Next, we consider ${\ell \in D_{R_{1} + R_{2}} \times \R}$, $m \in D_{R_{1}} \times \R$, then define
\begin{align*}
C'_{4} &:= \max \{ \, |y_{0}(n_{12})|, |u(n_{12})| \, | \, n_{12} \in \L \cap  D_{R_{1} + R_{2}} \} 
\end{align*}
and $R_{3} := \max\{16C'_{4},6(2R_{1} + R_{2})\}$.
If $|(\ell - m)_{3}| \geq R_{3}$, we deduce
\begin{align} \label{eq:y0-3m4-est-2}
|\wt y_{0}(\ell) - \wt y_{0}(m)| &= |y_{0}(\ell_{12}) - y_{0}(m_{12}) + (\ell - m)_{3}| \nonumber \\ &\geq |(\ell - m)_{3}| - 2C'_{4}  \nonumber \geq \frac{7\frak{m}}{8}|(\ell - m)_{3}| \\ &\geq \frac{3\frak{m}}{4}(|(\ell - m)_{12}| +  |(\ell - m)_{3}|) \geq \frac{3\frak{m}}{4}|\ell - m|,
\end{align}
therefore
\begin{align} \label{eq:u-v-open-calc-dis-4}
|\wt y(\ell) - \wt y(m)| &= |\wt y_{0}(\ell) - \wt y_{0}(m)| - |u(\ell_{12}) - u(m_{12})| \nonumber \\ &\geq \frac{3\frak{m}}{4}|\ell - m| - 2C'_{4} \geq \frac{\frak{m}}{2}|\ell - m|.
\end{align}
The remaining case to consider is ${\ell \in \L_{0} \cap( D_{R_{1} + R_{2}} \times \R)}$, ${m \in \L_{0} \cap( D_{R_{1}} \times \R)}$ satisfying ${|(\ell - m)_{3}| \leq R_{3}}$. As $\L_{0}$ is periodic in the $e_{3}\text{-direction}$, it suffices to consider $m \in D_{R_{1}} \times [0,1)$ and ${\ell \in D_{R_{1} + R_{2}} \times [-R_{3},R_{3}+1)}$. Due to our initial assumption on $\wt y$, we have 
\begin{align*}
c_{1} := \min \left\{ \, |\wt y(\ell') - \wt y(m')| \, | \, {m' \in D_{R_{1}} \times [0,1), \ell' \in D_{R_{1} + R_{2}} \times [-R_{3},R_{3}+1)} \right\} > 0,
\end{align*}
then as ${|\ell - m| \leq 2R_{1} + R_{2} + R_{3}}$, we deduce
\begin{align} \label{eq:u-v-open-calc-dis-5}
|\wt y(\ell) - \wt y(m)| \geq c_{1} \geq \frac{c_{1}}{2R_{1} + R_{2} + R_{3}}|\ell - m| =: c_{2}|\ell - m|.
\end{align}
Collecting \eqref{eq:u-v-open-calc-dis-2}--\eqref{eq:u-v-open-calc-dis-5} gives the desired inequality for $\frak{m}' = \min\{\frac{\frak{m}}{2},c_{2}\} > 0$.

We now show that there exists $\lambda > 0$ such that $\cup_{\ell \in \L_{0}} B_{\lambda}(\wt y(\ell)) = \R^{3}$ by making use of our assumption \eqref{eq:Iy0-ass-rep}.

As $u \in \UsH(\L)$ 
it follows that for any $\varepsilon > 0$ there exists $R = R(\varepsilon) \geq R_{0} > 0$ such that for all $\ell' \in \L \cap D_{R}^{\rm c}$ , we have $\max_{\ell'' \in \mathscr{N}(\ell') \setminus \ell', \rho \in \mathscr{N}(\ell'') - \ell''}|D_{\rho}u(\ell'')| \leq \varepsilon$, where the value of $\varepsilon$ will be specified shortly. Let $\Gamma$ denote the unit cell of the lattice $\L$, centred at~$0$. 
Then consider fixed $k \in \mathbb{N}$, odd and sufficiently large such that $\Omega_{1} := k \bar{\Gamma} \supset D_{R}$. Additionally, observe that $\Omega_{1}^{\rm c} = \cup_{\ell' \in \L \cap \Omega_{1}^{\rm c}} (\Gamma + \ell') \subset \R^{2}$. 
Using \cite{2013-defects-v3,ortner_shapeev12}, we construct an interpolant $Iu: \Omega_{1}^{\rm c} \to \R^{3}$ satisfying ${Iu \in W^{1,2}_{\rm loc}(\Omega_{1}^{\rm c})}$, $Iu(\ell') = u(\ell')$ for all ${\ell' \in \L \cap \Omega_{1}^{\rm c}}$ and 
\begin{align*}
\| \nabla Iu \|_{L^{\infty}(\Gamma + \ell'
	)} &\leq C_{1} \max_{\ell''
	 \in \mathscr{N}(\ell'
	) \setminus \ell', \rho \in \mathscr{N}(\ell'') - \ell''} |D_{\rho}u(\ell'')|,
\end{align*}
for all $\ell' \in \L \cap \Omega_{1}^{\rm c}$, where $C_{1} > 0$ is independent of $\ell'$ and $\varepsilon$. We choose $\varepsilon = \frac{\frak{m}}{4C_{1}} > 0$ which ensures that $\| \nabla Iu \|_{L^{\infty}(\Gamma + \ell')} \leq \smfrac{\frak{m}}{4}$ for all $\ell' \in \L \cap \Omega_{1}^{\rm c}$. Following the argument \eqref{eq:Iu-est-1}--\eqref{eq:Iu-est-2} yields for $z_{1}, z_{2} \in \Omega_{1}^{\rm c}$
%
%
\begin{align}
|Iu(z_{1}) - Iu(z_{2})| \leq \frac{\frak{m}}{2}|z_{1} - z_{2}|. \label{eq:Iu-est-dis}
\end{align}
Now we define $I\wt y(x) := I\wt y_{0}(x) + Iu(x_{12})$, for $x \in \Omega_{1}^{\rm c} \times \R$, then combining \eqref{eq:Iy0-ass-rep} and \eqref{eq:Iu-est-2} implies
\begin{align*}
\frac{\frak{m}}{2}|x_{1} - x_{2}|\leq |I\wt y(x_{1}) - I\wt y(x_{2})| \leq \frac{\frak{m} + 2M}{2}|x_{1} - x_{2}|,
\end{align*}
for all $x_{1}, x_{2} \in \Omega_{1}^{\rm c} \times \R$.
Therefore, the restriction of $I\wt y$ onto it's image is a homeomorphism and there exists bounded $\Omega_{2} \subset \R^{2}$ such that $I\wt y(\Omega_{1}^{\rm c}\times \R) = \Omega_{2}^{\rm c} \times \R$. So for any $x \in \Omega_{2}^{\rm c} \times \R$, there exists a unique $x' \in \Omega_{1}^{\rm c} \times \R$ such that $I\wt y(x') = x$. Additionally, there exists $\ell \in \L_{0} \cap ( \Omega_{1}^{\rm c} \times \R)$ such that $|\ell - x'| \leq C_{2}$, where the constant $C_{2}$ is independent of $x' \in \Omega_{1}^{\rm c} \times \R$. Consequently, we deduce that
\begin{align*}
|\wt y(\ell) - x | = |I\wt y(\ell) - I\wt y(x')| \leq \frac{\frak{m} + 2M}{2}|x_{1} - x_{2}| \leq \frac{C_{2}(\frak{m} + 2M)}{2} =: \lambda_{1}.
\end{align*}
Then, let $\ell \in \L_{0}$ minimise $\min_{\ell' \in \L} \textrm{dist}(y(\ell'),\Omega_{2}) =: C_{3} \geq 0$ and observe that the diameter $\textrm{diam}(\Omega_{2})$ is finite as $\Omega_{2}$ is bounded. Let $x \in \Omega_{2} \times \R$, then due to the periodicity of $\L_{0}$ in the $e_{3}$-direction, we have that $\ell + je_{3} \in \L_{0}$ for all $j \in \Z$. We can then find $\ell' = \ell + j'e_{3} \in \L_{0}$ satisfying $|(\wt y(\ell') - x)_{3}| \leq 1$, which implies
\begin{align*}
|\wt y(\ell') - x | \leq C_{3} + \textrm{diam}(\Omega_{2}) + 1 =: \lambda_{2}.
\end{align*}
As this holds independently of $x \in \Omega_{2} \times \R$, choosing ${\lambda := \max\{ \lambda_{1}, \lambda_{2}\} > 0}$ ensures ${\cup_{\ell \in \L_{0}} B_{\lambda}(\wt y(\ell)) = \R^{3}}$, showing \eqref{eq:H-Lambda-dis-def} for dislocations and completing the proof.

\end{proof}

\subsection{Proof of Theorem \ref{theorem:E-Wc}}
\label{sec:proof_theo_EW}
\fn{Recall the space $\mathscr{H}_{m,\lambda}$ for some $\frak{m}, \lambda > 0$, given in \eqref{space:HL}, that $u \in \mathscr{H}_{\frak{m},\lambda}(\L)$ if and only if $u \in \UsH(\L)$ and $y_{0} + u \in \Adm_{\frak{m},\lambda}(\L).$ }
In the following proof, we consider $u \in \mathscr{H}_{m,\lambda}$ for some $\frak{m}, \lambda > 0$. For convenience, we assume that for all $t \in [0,1]$, $y_{0} + tu \in \Adm_{\frak{m}/2,\lambda}$.

We remark that this condition does not hold in general, for example when ${y_{0} + u}$ is a permutation of the initial configuration $y_{0}$. In this case, it is possible to construct a piecewise-linear path $y'_{t} \in \Adm_{\frak{m}/2,\lambda}$ for $t \in [0,1]$, satisfying $y'_{0} = y_{0}$ and ${y'_{1} = y_{0} + u}$. Treating this case requires adapting the estimate \eqref{eq: quad est} due to the fact that $\frac{\text{d}}{\text{d} t} Dy'_{t}(\ell) \not \equiv Du(\ell)$ for some $\ell \in \L$. We omit this argument for the sake of brevity.

\begin{proof}
	[Proof of Theorem \ref{theorem:E-Wc}]
	We remark that $\Hw_{k+2}^{\log} \subset \Hw_{k}$ for all $k > 0$, hence it suffices to assume {\asSEL} holds under the perfect lattice and point defects condition, as the proof under the dislocations condition is identical.

	(i) Let $u\in \mathscr{H}^{\rm c}$ and consider the sum
	\begin{align}
	&\E(u) - \< F,u\> \nonumber \\ & \qquad = \sum_{\ell\in\L}\Big( V_{\ell}\big(Du_0(\ell)+Du(\ell)\big)-V_{\ell}\big(Du_0(\ell)\big) -\big\<\delta V_{\ell}(Du_0(\ell)),Du(\ell)\big\> \Big). \label{eq: E(u) - T eq}
	\end{align}
	Note that \asSER~and \asSEL~imply that there exist $\wf_1\in\Hw_1$,
	$\wf_2\in\Hw_2$ and a constant $C$ depending on $\wf_1$, $\wf_2$
	such that
	\begin{align}
	&\Big| V_{\ell}\big(Du_0(\ell)+Du(\ell)\big)-V_{\ell}\big(Du_0(\ell)\big) -\big\<\delta V_{\ell}(Du_0(\ell)),Du(\ell)\big\> \Big| \nonumber \\
	& \qquad \leq \frac{1}{2} \sum_{\rho,\sigma\in\L-\ell} \int_{0}^{1} (1-t) |V_{\ell,\rho\sigma}(Du_0(\ell) + t Du(\ell))||D_{\rho}u(\ell)||D_{\sigma}u(\ell)| \text{ d}t  \nonumber \\
	& \qquad \leq C\bigg( \sum_{\rho,\sigma\in\L-\ell}\wf_1(|\rho|)\wf_1(|\sigma|)|D_{\rho}u(\ell)||D_{\sigma}u(\ell)| + \sum_{\rho\in\L-\ell}\wf_2(|\rho|)|D_{\rho}u(\ell)|^2  \bigg) \nonumber \\
	& \qquad \leq C\left(\big|Du(\ell)\big|_{\wf_1,1}^2 + \big|Du(\ell)\big|_{\wf_2,2}^2 \right) . \label{eq: quad est}
	\end{align}
	Applying the estimate \eqref{eq: quad est} and Lemma \ref{lemma:normEquiv} to \eqref{eq: E(u) - T eq} yields
	\begin{align}
	&\bigg| \sum_{\ell\in\L}\Big( V_{\ell}\big(Du_0(\ell)+Du(\ell)\big)-V_{\ell}\big(Du_0(\ell)\big) -\big\<\delta V_{\ell}(Du_0(\ell)),Du(\ell)\big\> \Big) \bigg| \nonumber \\
    & \!\!\!\!\!\!\!\!\!\!\! \leq C\sum_{\ell\in\L}\left(\big|Du(\ell)\big|_{\wf_1,1}^2 \!\! + \big|Du(\ell)\big|_{\wf_2,2}^2 \right) \leq C\big(\|Du\|^2_{\hc{l}^2_{\wf_1,1}}\!\!+\|Du\|^2_{\hc{l}^2_{\wf_2,2}}\big)
	\leq C\|Du\|^2_{\hc{l}^2_{\mathcal{N}}}. \label{eq: E(u) - T est 2}
	\end{align}
	This implies that $\E(u) - \<F,u\>$ is well-defined for $u \in \mathscr{H}^{\rm c}$, 
	hence as $F$ is a bounded linear functional on $(\Usz,\|D\cdot\|_{\hc{l}^2_{\mathcal{N}}})$, 
	it follows that $\E(u)$ is well-defined for all $u \in \mathscr{H}^{\rm c}$ and that $\delta \E(0) = F$.

	(ii) Moreover, as $\E(u) - \<F,u\> = \E(u) - \E(0) - \< \delta \E(0),u \>,$
	the estimate \eqref{eq: E(u) - T est 2} implies that  $\E(u) - \E(0) - \< \delta \E(0),u \>$ is continuous on $(\mathscr{H}^{\rm c},\|D\cdot\|_{\hc{l}^2_{\mathcal{N}}})$. 
	As $\E(0)$ is a bounded functional, it also follows that $\E(u)$ is also continuous on $(\mathscr{H}^{\rm c},\|D\cdot\|_{\hc{l}^2_{\mathcal{N}}})$.

	(iii)
	For $u\in\mathscr{H}^{\rm c}$ and $j=1$, we have from \asSEL~and Lemma \ref{lemma:normEquiv} that for any $v\in\Usz$,  there exists $\theta\in(0,1)$ such that
	\begin{align*}
	& \hskip -0.2cm \left| \b\<\delta\E(u),v\b\> - \b\<\delta\E(0),v\b\> \right|
	~\leq~\sum_{\ell\in\L}\bigg| \sum_{\rho,\sigma\in\L-\ell} V_{\ell,\rho\sigma}\big(Du_0(\ell)+\theta Du(\ell)\big) D_{\rho}u(\ell)D_{\sigma}v(\ell) \bigg|
	\\
	& \leq~\sum_{\ell\in\L} \bigg| \sum_{\rho,\sigma\in\L-\ell} \wf_1(|\rho|)\wf_1(|\sigma|)D_{\rho}u(\ell)D_{\sigma}v(\ell)
	+\sum_{\rho\in\L-\ell}\wf_2(|\rho|)D_{\rho}u(\ell)D_{\rho}v(\ell) \bigg|
	\\
	&\leq~C\sum_{\ell\in\L} \left(\big|Du(\ell)\big|_{\wf_1,1}\big|Dv(\ell)\big|_{\wf_1,1} + \big|Du(\ell)\big|_{\wf_2,2}\big|Dv(\ell)\big|_{\wf_2,2} \right)
	\\
	&\leq~C\left( \|Du\|_{\hc{l}^2_{\wf_1,1}}\|Dv\|_{\hc{l}^2_{\wf_1,1}} + \|Du\|_{\hc{l}^2_{\wf_2,2}}\|Dv\|_{\hc{l}^2_{\wf_2,2}} \right)
	~\leq~C\|Du\|_{\hc{l}^2_{\mathcal{N}}}\|Dv\|_{\hc{l}^2_{\mathcal{N}}} .
	\end{align*}
	Since $\delta\E(u)$ is a bounded linear functional on $(\mathscr{H}^{\rm c},\|D\cdot\|_{\hc{l}^2_{\mathcal{N}}})$, 
	there exists a constant $C$ such that $\left| \b\<\delta\E(u),v\b\> \right| \leq C\|Dv\|_{\hc{l}^2_{\mathcal{N}}}$
	for all $v\in\Usz$, which implies that $\E$ is differentiable.
	The case $j=2$ in the following implies $\E$ is continuously differentiable.

    For $u\in\mathscr{H}^{\rm c}$, $2 \leq j \leq \n$, ${\pmb v}=(v_1,\fn{\ldots},v_j)\in\big(\Usz\big)^j$
    and ${\pmb \rho}=(\rho_1,\fn{\ldots},\rho_j)\in(\L-\ell)^j$, we have
	from \asSEL~that there exist $\wf_k\in\Hw_{k}$ for $k=1,\fn{\ldots},j$ such that
	\begin{align*}
	\left| \<\delta^j\E(u),{\pmb v}\> \right| &\leq
	\sum_{\ell\in\L} \sum_{{\pmb \rho}\in(\L-\ell)^j}
	\left| \left\<V_{\ell,{\pmb \rho}}\b(Du(\ell)\b), D_{\pmb \rho}\otimes {\pmb v}(\ell) \right\> \right|
	\\
	&\leq \sum_{\ell\in\L}\sum_{\rho_1,\fn{\ldots},\rho_j} |V_{\ell,\rho_1,\fn{\ldots},\rho_j}\b(Du(\ell)\b)|
	\prod_{1 \leq m \leq j} |D_{\rho_m}v_m(\ell)|.
	\end{align*}
	Applying \eqref{eq:Vloc} gives
	\begin{align}
	&\left| \<\delta^j\E(u),{\pmb v}\> \right| \nonumber \\
	&\leq C \sum_{\ell\in\L}\sum_{\rho_1,\fn{\ldots},\rho_j}
	\sum_{\{A_{1},\ldots,A_{k}\} \in \mathcal{P}(j)}
	\bigg( \prod_{1\leq i\leq k} \wf_{|A_i|}(\rho_{i'}) \prod_{m \in A_{i}} \delta_{\rho_{i'}\rho_{m}} \bigg)
	\prod_{1 \leq m \leq j} |D_{\rho_m}v_m(\ell)|. \label{eq:Ejregest1}
	\end{align}
	First fix $\{ A_{1},\ldots, A_{k} \} \in \mathcal{P}(j)$ and consider
	\begin{eqnarray}
	&&\sum_{\ell\in\L}\sum_{\rho_1,\fn{\ldots},\rho_j} \bigg( \prod_{1\leq i\leq k} \wf_{|A_i|}(\rho_{i'}) \prod_{m \in A_{i}} \delta_{\rho_{i'}\rho_{m}} \bigg) \nonumber
	\prod_{1 \leq m \leq j} |D_{\rho_m}v_m(\ell)| \\
	&=& \sum_{\ell\in\L}\sum_{\rho_{1'},\fn{\ldots},\rho_{k'}} \bigg( \prod_{1\leq i\leq k} \wf_{|A_i|}(\rho_{i'}) \prod_{m \in A_{i}} |D_{\rho_{i'}}v_m(\ell)| \bigg) \nonumber \\
	&=& \sum_{\ell\in\L} \prod_{1\leq i\leq k} \bigg( \sum_{\rho_{i'}} \wf_{|A_i|}(\rho_{i'}) \prod_{m \in A_{i}} |D_{\rho_{i'}}v_m(\ell)| \bigg) \nonumber \\
	&=& \sum_{\ell\in\L} \prod_{1\leq i\leq k} \bigg( \sum_{\rho \in \L - \ell} \prod_{m \in A_{i}} \wf_{|A_i|}(\rho)^{1/|A_{i}|} |D_{\rho}v_m(\ell)| \bigg). \label{eq:Ejregest2}
	\end{eqnarray}
	Then applying the generalised H\"{o}lder's inequality and \eqref{eq:wnormdef} yields
	\begin{align}
	&\sum_{\ell\in\L} \prod_{1\leq i\leq k} \bigg( \sum_{\rho \in \L - \ell} \prod_{m \in A_{i}} \wf_{|A_i|}(\rho)^{1/|A_{i}|} |D_{\rho}v_m(\ell)| \bigg) \nonumber \\
	&\leq \sum_{\ell\in\L} \prod_{1\leq i\leq k} \prod_{m \in A_{i}} \bigg( \sum_{\rho \in \L - \ell} \wf_{|A_i|}(\rho) |D_{\rho}v_m(\ell)|^{|A_{i}|} \bigg)^{1/|A_{i}|} \nonumber \\
	&= \sum_{\ell\in\L} \prod_{1\leq i\leq k} \prod_{m \in A_{i}} |Dv_{m}(\ell)|_{\wf_{|A_{i}|},|A_{i}|}. \label{eq:Ejregest3}
	\end{align}
	We will now show that for each $1 \leq i \leq k$ and $m \in A_{i}$
	\begin{align}\label{eq:Ejregest4}
	\bigg( \sum_{\ell\in\L} |Dv_{m}(\ell)|_{\wf_{|A_i|},|A_i|}^{j} \bigg)^{1/j} \leq C \| Du \|_{\hc{l}^2_{\mathcal{N}}}.
	\end{align}
	To prove \eqref{eq:Ejregest4}, first consider the case $|A_{i}| = 1$. Recall that $j \geq 2$, hence \fn{$\hc{l}^2 \subseteq \hc{l}^j$} and applying Lemma \ref{lemma:normEquiv} yields
	\begin{align*}
	\bigg( \sum_{\ell\in\L} |Dv_{m}(\ell)|_{\wf_{|A_i|},|A_i|}^{j} \bigg)^{1/j}
	&\leq \bigg( \sum_{\ell\in\L} |Dv_{m}(\ell)|_{\wf_{|A_i|},|A_i|}^{2} \bigg)^{1/2} \\
	&\leq \|Dv_{m} \|_{\hc{l}^2_{\wf_{|A_i|},|A_i|}} \leq C \|Dv_{m} \|_{\hc{l}^2_{\mathcal{N}}}.
	\end{align*}
	For the remaining case, observe that $2 \leq |A_{i}| \leq j$, the argument above follows using the embedding \fn{$\hc{l}^{|A_{i}|} \subseteq \hc{l}^j$}
	\begin{align*}
	\left( \sum_{\ell\in\L} |Dv_{m}(\ell)|_{\wf_{|A_i|},|A_i|}^{j} \right)^{1/j}
	&\leq \left( \sum_{\ell\in\L} |Dv_{m}(\ell)|_{\wf_{|A_i|},|A_i|}^{|A_{i}|} \right)^{1/|A_{i}|}  
	\leq C \|Dv_{m} \|_{\hc{l}^2_{\mathcal{N}}}.
	\end{align*}
	By \eqref{eq:Ejregest4}, the summand in \eqref{eq:Ejregest3} consists of the product of $j$ functions, 
	each belonging to $\hc{l}^j$, hence applying the generalised H\"{o}lder's inequality, we deduce
	\begin{align}
	&\sum_{\ell\in\L} \prod_{1\leq i\leq k} \prod_{m \in A_{i}} \bigg( \sum_{\rho \in \L - \ell} \wf_{|A_i|}(\rho) |D_{\rho}v_m(\ell)|^{|A_{i}|} \bigg)^{1/|A_{i}|} \nonumber \\
	&\leq \prod_{1\leq i\leq k} \prod_{m \in A_{i}} \Bigg( \sum_{\ell\in\L} \bigg( \sum_{\rho \in \L - \ell} \wf_{|A_i|}(\rho) |D_{\rho}v_m(\ell)|^{|A_{i}|} \bigg)^{j/|A_{i}|} \Bigg)^{1/j} \nonumber \\
	&\leq C \prod_{1\leq i\leq k} \prod_{m \in A_{i}} \|Dv_{m} \|_{\hc{l}^2_{\mathcal{N}}} = C \prod_{1 \leq m \leq j} \|Dv_{m} \|_{\hc{l}^2_{\mathcal{N}}}. \label{eq:Egregest5}
	\end{align}
	Combining the estimates \eqref{eq:Ejregest2}, \eqref{eq:Ejregest3} and \eqref{eq:Egregest5} yields
	\begin{align}
	\sum_{\ell\in\L}\sum_{\rho_1,\fn{\ldots},\rho_j} \bigg( \prod_{1\leq i\leq k} \wf_{|A_i|}(\rho_{i'}) \prod_{m \in A_{i}} 
	\hc{\delta_{\rho_{i'}\rho_m}}\bigg) \prod_{1 \leq m \leq j} |D_{\rho_m}v_m(\ell)| 
	\leq C \prod_{1 \leq m \leq j} \|Dv_{m} \|_{\hc{l}^2_{\mathcal{N}}}. 
	\label{eq:Egregest6}
	\end{align}
	As the set of partitions $\mathcal{P}(j)$ is finite (in fact $|\mathcal{P}(j)| = B_{j}$, the $j$-th Bell number \cite{Bell34}), 
	collecting the estimates \eqref{eq:Ejregest1}--\eqref{eq:Egregest6} implies
	\begin{align}\label{eq:useful-eq}
	&\left| \<\delta^j\E(u),{\pmb v}\> \right| \nonumber \\
	&\leq C \sum_{\{A_{1},\ldots,A_{k}\} \in \mathcal{P}(j)} \sum_{\ell\in\L}\sum_{\rho_1,\fn{\ldots},\rho_j}
	\bigg( \prod_{1\leq i\leq k} \wf_{|A_i|}(\rho_{i'}) \prod_{m \in A_{i}} \hc{\delta_{\rho_{i'}\rho_m}} \bigg)
	\prod_{1 \leq m \leq j} |D_{\rho_m}v_m(\ell)| \nonumber \\
	&= C \sum_{\{A_{1},\ldots,A_{k}\} \in \mathcal{P}(j)} \sum_{\ell\in\L}\sum_{\rho_1,\fn{\ldots},\rho_j}
	\prod_{m \in A_{i}} |Dv_{m}(\ell)|_{\wf_{|A_{i}|},|A_{i}|} \nonumber  \\
	&\leq C \sum_{\{A_{1},\ldots,A_{k}\} \in \mathcal{P}(j)} \prod_{1 \leq m \leq j} \|Dv_{m} \|_{\hc{l}^2_{\mathcal{N}}} 
	= C |\mathcal{P}(j)| \prod_{1 \leq m \leq j} \|Dv_{m} \|_{\hc{l}^2_{\mathcal{N}}} = C \prod_{1 \leq m \leq j} \|Dv_{m} \|_{\hc{l}^2_{\mathcal{N}}}.
	\end{align}
	It follows that $\E$ is $\n$-times differentiable on $\mathscr{H}$, hence $\E \in C^{\n-1}(\mathscr{H},\| D \cdot \|_{\hc{l}^2_{\mathcal{N}}})$.
\end{proof}

\section{The lattice Green's function}
\label{sec:lattice_green}
\setcounter{equation}{0}

\fn{Recall that in Section \ref{sec:siteE}, we introduced the homogeneous site strain potential $V^{\hom}$. 
	For the sake of convenience, we simply denote $V^{\hom}$ by $V$ in all subsequent proofs.}

\fn{In the following lemma, we give an alternate formulation of the homogeneous difference equation \eqref{eq:def-Huv}.}
%

\begin{lemma}\label{lemma:bilinear-H}
	There exists $h:\fn{\Lhom}\rightarrow\R^{\ds\times \ds}$, defined by
	\begin{eqnarray}\label{eq:h-rho}
	h(\rho) := - \frac{1}{2}\sum_{\xi,\tau\in\Lhom_{*},~\xi-\tau=\rho} V_{,\xi\tau}({\pmb 0})
	\end{eqnarray}
	such that
	\begin{eqnarray}\label{eq:Huv-hrho}
	\b\< Hu,v \b\> = \sum_{\ell\in\Lhom}\sum_{\rho\in\Lhom_{*}} D_{\rho}v(\ell)^{\rm T} \cdot h(\rho) \cdot D_{\rho}u(\ell) .
	\end{eqnarray}
	Moreover,  $h$ satisfies $h(-\rho)=h(\rho)$ and $h(0)=-\sum_{\rho\in\Lhom_{*}}h(\rho)$.
\end{lemma}

\fn{ \begin{proof}
We justify that $h(\rho)$ is defined for every $\rho \in \Lhom$ and is also summable. First consider the case $\rho = 0$. By {\asSEL}, we have
\begin{align*}
|h(0)| &\leq C \sum_{\tau\in\Lhom_{*}} \left( \wf_{1}^{2}(|\tau|) + \wf_{2}(|\tau|) \right) \leq C ( 1 + \| \wf_{1} \|_{L^{\infty}[0,\infty)}) \sum_{\tau\in\Lhom_{*}} \left( \wf_{1}(|\tau|) + \wf_{2}(|\tau|) \right) \\
& \leq C ( 1 + \| \wf_{1} \|_{\Hw_{1}}) ( \| \wf_{1} \|_{\Hw_{1}} + \| \wf_{2} \|_{\Hw_{2}}),
\end{align*}
where we have applied Lemma \ref{lemma:omega-sum-int-est}. For $\rho \in \Lhom_{*}$, by arguing similarly we obtain
\begin{align*}
|h(\rho)| &\leq C \sum_{\tau\in\Lhom_{*}} \wf_{1}(|\tau|) \wf_{1}(|\tau + \rho|).
\end{align*}
First suppose that $|\tau| \geq \frac{3|\rho|}{2}$, then $|\tau + \rho| \geq |\tau| - |\rho| \geq \frac{|\rho|}{2}$, by the triangle inequality. Similarly, when $|\tau| \leq \frac{|\rho|}{2}$, we obtain that $|\tau+\rho| \geq \frac{|\rho|}{2}$. Consequently, by applying Lemma \ref{lemma:omega-sum-int-est}, we deduce
\begin{align*}
|h(\rho)| &\leq C \sum_{\tau\in\Lhom_{*}} \wf_{1}(|\tau|) \wf_{1}(|\tau + \rho|) \\
&\leq C \!\!\!\!\!\!\!\! \sum_{\tau\in\Lhom_{*} \cap (B_{3|\rho|/2}^{\rm c} \cup B_{|\rho|/2}) } \!\!\!\!\!\!\!\! \wf_{1}(|\tau|) \wf_{1}(|\tau + \rho|) + C \!\!\!\!\!\!\!\! \sum_{\tau\in\Lhom_{*} \cap B_{3|\rho|/2} \setminus B_{|\rho|/2} } \!\!\!\!\!\!\!\! \wf_{1}(|\tau|) \wf_{1}(|\tau + \rho|) \\
&\leq C \wf_{1}(|\rho|/2) \Big( \!\!\!\! \sum_{\tau\in\Lhom_{*} \cap (B_{3|\rho|/2}^{\rm c} \cup B_{|\rho|/2}) } \!\!\!\!\!\!\!\! \wf_{1}(|\tau|) + C \!\!\!\!\!\!\!\! \sum_{\tau\in\Lhom_{*} \cap B_{3|\rho|/2} \setminus B_{|\rho|/2} } \!\!\!\!\!\!\!\! \wf_{1}(|\tau + \rho|) \Big) \\
&\leq C \wf_{1}(|\rho|/2) \sum_{\tau\in\Lhom} \wf_{1}(|\tau|) \leq C \| \wf_{1} \|_{\Hw_{1}} \wf_{1}(|\rho|/2).
\end{align*}
This proves that $h$ is defined and a second application of Lemma \ref{lemma:omega-sum-int-est} implies that $h$ is also summable over $\Lhom$. The remaining properties of $h$ follow directly from~{\asSEPS}, while \eqref{eq:Huv-hrho} can be shown by expanding the finite-difference terms. 
\end{proof} 
}

Using \eqref{eq:Huv-hrho} and a direct calculation,
we have
\begin{eqnarray}\label{eq:Huv_conv}
\nonumber
\b\<Hu,v\b\>  = -2\sum_{\ell\in\Lhom}\sum_{\rho\in\Lhom} h(\rho) u(\ell-\rho)v(\ell),
\end{eqnarray}
and hence
\begin{eqnarray}\label{Hu_conv}
Hu=-2h\ast_{\rm d} u := -2\sum_{\rho\in\Lhom}h(\rho)u(\ell-\rho),
\end{eqnarray}
where $\ast_{\rm d}$ denotes the discrete convolution.
From \eqref{eq:Huv_conv}, one can alternatively write $h(\rho) := -\frac{1}{2}\frac{\partial^2 \<Hu,u\>}{\partial u(0) \partial u(\rho)}$
(see also \cite[Lemma 3.4]{Hudson:stab}).

The following lemma defines a lattice Green's function of
the homogeneous difference equation and gives a decay estimate.

\begin{lemma}\label{lemma:decay-Green}
	If \asLS~is satisfied, then
	\begin{itemize}
		\item[(i)]
		there exists $\Gr:\Lhom\rightarrow\R^{d\times\ds}$ such that for any $f:\Lhom\rightarrow\R^{\ds}$ which is compactly supported,
		\begin{eqnarray*}
			H(\Gr \ast_{\rm d} f) = f;
		\end{eqnarray*}
		\item[(ii)]
		for all $j \in \mathbb{N}$, there exist constants $C_j$ such that
		\begin{eqnarray*}
			\left|D_{\pmb \rho}\Gr(\ell)\right| \leq C_j \b( 1+|\ell| \b)^{2-d-j} \prod_{i=1}^j |\rho_i|
			\qquad\forall~{\pmb\rho}=(\rho_1,\fn{\ldots},\rho_j)\in(\Lhom_{*})^j ;
		\end{eqnarray*}
		\item[(iii)]
		$\Gr$ can be chosen in such a way that there exists $C_0>0$ such that
		\begin{eqnarray*}
			\left|\Gr(\ell)\right| \leq \left\{ \begin{array}{ll}
				C_0\b(1+|\ell|\b)^{2-d} & {\rm if}~d=3 , \\[1ex]
				C_0\log\b(2+|\ell|\b) & {\rm if}~d=2 .
			\end{array}\right.
		\end{eqnarray*}
	\end{itemize}
\end{lemma}

\begin{proof}
	The proof is similar to that of \cite[Lemma 6.2]{2013-defects-v3}.
	
	{\it 1. Semi-discrete Fourier transform.}
	Let $\Br$ be the first Brillouin zone associated with $\Lhom$.
	For $u\in\ell^1(\Lhom;\R^{\ds})$ and $k\in\Br$, we define the semi-discrete Fourier transform and its inverse
	\begin{eqnarray*}
		\hat{u}:=\Ftd[u](k):=\sum_{\ell\in\Lhom}e^{ik\cdot\ell}u(\ell)
		\qquad{\rm and}\qquad
		\Ftd^{-1}[\hat{u}](\ell):=\int_{\Br}e^{-ik\cdot\ell}\hat{u}(k)\dd k.
	\end{eqnarray*}
	We can transform \eqref{eq:def-Huv} to Fourier space
	\begin{eqnarray*}
		\b\<Hu,v\b\>=\int_{\Br}\hat{u}(k)^*\hat{H}(k)\hat{v}(k)\dd k
	\end{eqnarray*}
	with
	\begin{eqnarray*}
		\hat{H}(k) = \sum_{\rho,\xi\in\Lhom_{*}} V_{,\rho\xi}(\pmb{0}) \b(e^{ik\cdot\rho}-1\b) \b(e^{-ik\cdot\xi}-1\b) .
	\end{eqnarray*}
	Moreover, \eqref{Hu_conv} implies that $\hat{H}(k)=-2\hat{h}(k)$.
	We have from \asSEL~ and \asLS~ that
	there exist constants $\underline{c}$ and $\bar{c}$ such that
	\begin{eqnarray}\label{equiv_Hk_k2}
	\underline{c}|k|^2{\sf Id} \leq\hat{H}(k) \leq \bar{c}|k|^2{\sf Id}
	\qquad\forall~k\in\Br.
	\end{eqnarray}

	{\it 2. Definition of the lattice Green's function.}
	Since $\hat{H}(k)$ is invertible for $k\in\Br\backslash 0$,
	we could define $\Gr$ by
	$\hat{\Gr}(k)=\hat{H}(k)^{-1}$ so that $H(\Gr\ast_{\rm d} f)=\Ftd^{-1}\b[\hat{H}\hat{\Gr}\hat{f}\b]=f$.
	However, a subtle technical point is that $\hat{\Gr}$ is not  integrable if $d=2$ and $\hat{f}(0)\neq 0$,
	hence it is not obvious how $\Gr\ast_{\rm d}f=\Ftd^{-1}\b[\hat{\Gr}\hat{f}\b]$ is to be understood.
	To account for this, we note that shifting $\Gr$ by a constant does not modify statements (i) and (ii).
	Therefore, we define
	\begin{eqnarray}\label{def_Green}
	\Gr(\ell):= c+\int_{\Br}\hat{H}(k)^{-1}\b(e^{ik\cdot\ell}-1\b) \dd k,
	\end{eqnarray}
	where $c$ is a constant that we will determine in the proof of (iii).
	Note that \eqref{equiv_Hk_k2} implies that $\b| \hat{H}(k) (e^{ik\cdot\ell}-1)\b|\leq C|k|^{-1}\in L^1(\Br)$, and hence $\Gr$ is well-defined.
	
	To prove (i), it is now sufficient to show that $H\Gr={\sf Id}\delta$.
	Since $|h(\rho)|\leq C\wf_1(|\rho|/2)$, $H\Gr=-2h\ast_{\rm d}\Gr$ is well-defined and we can write
	\begin{align*}
	\b(H\Gr\b)(\ell) &= -2\sum_{\ell\in\Lhom}h(\rho)\Gr(\ell-\rho)
	= -2\sum_{\ell\in\Lhom}h(\rho) \int_{\Br}\hat{H}(k)^{-1}\b(e^{ik\cdot(\ell-\rho)}-1\b) \dd k
	\\
	&= \int_{\Br}\left( \sum_{\ell\in\Lhom}h(\rho) \b(e^{ik\cdot(\ell-\rho)}-1\b) \right) \hat{h}(k)^{-1} \dd k.
	\qquad
	\end{align*}
	Using $h(-\rho)=h(\rho)$ and $h(0)=-\sum_{\rho\in\Lhom_{*}}h(\rho)$, we obtain
	\begin{eqnarray*}
		\b(H\Gr\b)(\ell) &=& \int_{\Br}\left( \sum_{\ell\in\Lhom}h(\rho) e^{ik\cdot(\ell-\rho)} \right) \hat{h}(k)^{-1} \dd k
		\\
		&=& \int_{\Br}\left( \sum_{\ell\in\Lhom}h(\rho) e^{-ik\cdot\rho} \right) \hat{h}(k)^{-1} e^{ik\cdot \ell} \dd k
		\\
		&=&  \int_{\Br} \hat{h}(k)\hat{h}(k)^{-1}e^{ik\cdot\ell}\dd k
		~=~ {\sf Id} \int_{\Br}e^{ik\cdot\ell} \dd k
		~=~ {\sf Id}\delta(\ell),
	\end{eqnarray*}
	which completes the proof of (i).
	
	{\it 3. Modified continuum Green's function.}
	Define the partial differential operator
	\begin{eqnarray}\label{def_cpd}
	L:=\sum_{\rho\in\Lhom_{*}}\nabla_{\rho}\cdot h(\rho)\nabla_{\rho}.
	\end{eqnarray}
	%
	Let $G$ denote the Green's function of $L$,
	$\Fs$ and $\Fs^{-1}$ denote the standard Fourier transform and its inverse.
	Using the assumption \asLS, we have from \cite[Theorem 6.2.1]{Morrey} that
	\begin{eqnarray}\label{def_G_con_k}
	\Fs[G](k) := \bigg(\sum_{\rho\in\Lhom_{*}}h(\rho)(k\cdot\rho)^2\bigg)^{-1}
	\in C^{\infty}(\R^d\backslash 0)^{\ds\times\ds}
	\end{eqnarray}
	and
	\begin{eqnarray}\label{decay_G_con}
	\b|\nabla^j G(x)\b|\leq C|x|^{2-d-j} \qquad{\rm for}~j=1,2,\cdots .
	\end{eqnarray}
	
	Let $\hat{\eta}(k)\in C_{\rm c}^{\infty}(\Br)$ with $\eta(k)=1$ in a neighbourhood of the origin,
	and extended by zero to $\R^d\backslash\Br$.
	Let $\eta:=\Fs^{-1}[\hat{\eta}]$ and $\tilde{G}:=\eta\ast G$ be the smeared Green's function.
	Note that $\hat{\eta}\in C^{\infty}$ implies that $\eta$ decays super-algebraically.
	Then we have from \eqref{decay_G_con} that
	\begin{eqnarray}\label{decay_G_con_eta}
	\b|\nabla^j \tilde{G}(x)\b|\leq C|x|^{2-d-j} \qquad{\rm for}~j=1,2,\cdots ~~{\rm and}~~|x|\geq 1.
	\end{eqnarray}

	{\it 4. Decay estimates.}
	It is then only necessary for us to compare the Green's functions $\Gr$ and $\tilde{G}$.
	Using the fact $(e^{ik\cdot\rho}-1)\Gr\in L^1(\Br)$, we have
	\begin{eqnarray*}
		D_{\rho}\Gr(\ell)=\int_{\Br}\hat{\Gr}(k)\b(e^{ik\cdot(\ell+\rho)}-e^{ik\cdot\ell}\b)\dd k
		=\Ftd^{-1}\b[(e^{ik\cdot\rho}-1)\hat{\Gr}\b]
	\end{eqnarray*}
	for $\rho\in\Lhom_{*}$.
	Analogously, if $\pmb{\rho}=(\rho_1,\cdots,\rho_j)\in(\Lhom_{*})^j$, then
	\begin{eqnarray*}
		D_{\pmb{\rho}}\Gr(\ell) = \Ftd^{-1}\b[\hat{p}_{\pmb{\rho}}\hat{\Gr}\b]
		\qquad{\rm with}\qquad
		\hat{p}_{\pmb{\rho}} = \prod_{s=1}^j\b(e^{ik\cdot\rho_s}-1\b) .
	\end{eqnarray*}
	Then we have
	\begin{eqnarray*}
		\Ftd\b[D_{\pmb{\rho}}(\tilde{G}-\Gr)\b] = \hat{p}_{\pmb{\rho}}\b(\hat{\eta}\hat{G}-\hat{\Gr}\b)
		= \hat{p}_{\pmb{\rho}}\hat{\eta}\b(\hat{G}-\hat{\Gr}\b) + \hat{p}_{\pmb{\rho}}(\hat{\eta}-1)\hat{\Gr}
		=: \hat{g}_1+\hat{g}_2,
	\end{eqnarray*}
	where we slightly abuse the notation by writing $\hat{G}:=\Ft[G]$.
	Since $\hat{g}_2\in C^{\infty}(\Br)$, we obtain that $g_2$ decays super-algebraically.
	
	The first term $\hat{g}_1\in C^{\infty}(\Br\backslash 0)$, and it is only necessary for us to estimate its regularity at 0.
	To that end, we first note that $\hat{G}^{-1}$ is the second-order Taylor polynomial of $\hat{\Gr}^{-1}$ at $k=0$,
	and the odd terms vanish due to symmetry of the lattice.
	Therefore, $\hat{h}_4(k):=\hat{\Gr}(k)^{-1}-\hat{G}(k)^{-1}$ has a power series which starts with quartic terms
	and converge in a ball $B_{\epsilon}\subset\Br$.
	We can now rewrite
	\begin{eqnarray*}
		\hat{p}_{\pmb{\rho}}\b(\hat{G}-\hat{\Gr}\b) = \hat{p}_{\pmb{\rho}}\b(\hat{G}-(\hat{G}^{-1}+\hat{h}_4)^{-1}\b)
		=\hat{p}_{\pmb{\rho}}\hat{G}\b(I-(I+\hat{G}\hat{h}_4)^{-1}\b),
	\end{eqnarray*}
	and $(I+\hat{G}\hat{h}_4)^{-1}$ has a power series that is convergent in $B_{\epsilon}$ for $\epsilon$ sufficiently small.
	Then we have
	\begin{eqnarray*}
		\b|\nabla^t\b(\hat{p}_{\pmb{\rho}}(\hat{G}-\hat{\Gr})\b)\b| \leq C|k|^{j-t}\prod_{s=1}^j|\rho_s|,
	\end{eqnarray*}
	where we have used the fact that $\hat{p}_{\pmb{\rho}}(k)\sim\prod_{s=1}^j(k\cdot\rho_s)$ as $k\to 0$.
	In particular, it follows that $\nabla^t\hat{g}_1\in L^1(\Br)$ provided $t\leq j+d-1$, and hence
	\begin{eqnarray*}
		\b|g_1(\ell)\b| = \b|\Ftd^{-1}[\hat{g}_1](\ell)\b|
		\leq C|\ell|^{1-d-j}\prod_{s=1}^j|\rho_s|.
	\end{eqnarray*}
	Then we have from $D_{\pmb{\rho}}(\tilde{G}-\Gr)=g_1+g_2$ and the super-algebraic decay of $g_1$ that
	\begin{eqnarray}\label{decay_G-Gr}
	\b|D_{\pmb{\rho}}(\tilde{G}-\Gr)(\ell)\b| \leq C(1+|\ell|)^{1-d-j}\prod_{s=1}^j|\rho_s| ,
	\end{eqnarray}
	which together with \eqref{decay_G_con_eta} completes the proof of (ii).
	
	It remains to establish the case $j=0$ in (iii).
	Note that (ii) implies $|D_{\rho}\Gr(\ell)|\leq C|\ell|^{1-d}$ for all $\rho$, hence there exists a constant $c_0$
	such that $\Gr(\ell)\to c_0$ uniformly as $|\ell|\to\infty$.
	We obtain by summing $D_{\rho}\Gr$ along a suitable path that
	\begin{eqnarray*}
		\Gr(\ell)-c_0 = \sum_{s=0}^{\infty} \Big( \Gr\b(\ell+s\rho\b) - \Gr\b(\ell+(s+1)\rho\b) \Big)
		= -\sum_{s=0}^{\infty} D_{\rho}\Gr(\ell+s\rho) ,
	\end{eqnarray*}
	which together with (ii) completes the proof of (iii) by choosing the constant $c$ in \eqref{def_Green}
	such that $c_0=0$.
\end{proof}

With the definition and decay estimates of the lattice Green's function,
we are now able to prove decay estimates for the linearised lattice elasticity problem
\begin{eqnarray}\label{eq:lin-lattice}
\b\<Hu,v\b\>=\b\<g,Dv\b\> \qquad\forall~v\in\UsH(\Lhom) ,
\end{eqnarray}
where $g \in \UsHd(\Lhom)$. The canonical form of $\b\<g,Dv\b\>$ is
\begin{eqnarray}\label{eq:gDv}
\b\<g,Dv\b\> := \sum_{\ell\in\Lhom}\sum_{\rho\in\Lhom_{*}}g_{\rho}(\ell)\cdot D_{\rho}v(\ell),
\end{eqnarray}
with $g=\b\{g(\ell)\b\}_{\ell\in\Lhom}$ and $g(\ell) \in \b(\R^{\ds}\b)^{\Lhom_{*}}$.

\begin{lemma}
	\label{lemma:decay-hom-eq}
	Let \asLS~be satisfied and $u\in\UsH(\Lhom)$ be the solution of \eqref{eq:lin-lattice}.
	If there exist $s > d/2$ and $\wf_k\in\Hw_{k}$ for $k=1,2,3$, such that
	\begin{eqnarray}\label{eq:decay_g}
	\left| \b\<g,Dv\b\> \right| \leq \fn{C \sum_{k=1}^{3} \sum_{\ell \in \Lhom} \tilde{g}(\ell) |Dv(\ell)|_{\wf_{k},k}}
	\qquad\forall~v\in\UsH(\Lhom),
	\end{eqnarray}
	with
	\fn{$ \tilde{g}(\ell) = (1+|\ell|)^{-s} + \sum_{j=1}^{3} |Du(\ell)|_{\wf_j,j}^{2}$},
	then there exists $C > 0$ such that for any $\rho\in\Lhom_{*}$ and $\ell \in \Lhom$,
	\begin{eqnarray}\label{decay_Du}
	\b|D_{\rho}u(\ell)\b| \leq C 	|\rho| \left\{ \begin{array}{ll}
	\b(1+|\ell|\b)^{-d} & {\rm if}~s> d , \\[1ex]
	\b(1+|\ell|\b)^{-s} \log\b(2+|\ell|\b)  & {\rm if}~\frac{d}{2}<s \leq d.
	\end{array}\right.
	\end{eqnarray}
\end{lemma}

\begin{proof}
	We adapt the argument used to show \cite[Lemma 6.3]{2013-defects-v3}.
	For $\ell\in\Lhom$, testing \eqref{eq:lin-lattice} with $v(m):=D_{\rho}\Gr(\ell-m)$
	and $\rho\in\Lhom_{*}$ yields
	\begin{eqnarray*}
		D_{\rho}u(\ell) = - \b\< Hu,D_{\rho}\Gr(\ell - \cdot) \b\> = - \b\< g,DD_{\rho}\Gr(\ell - \cdot) \b\>,
	\end{eqnarray*}
	and together with Lemma \ref{lemma:decay-Green} (ii), \eqref{eq:decay_g} and \fn{that Lemma \ref{lemma:omega-sum-int-est} implies}
	${\sum_{\sigma\in\Lhom_{*}}\wf_k(|\sigma|)|\sigma|^k<\infty}$ for $k=1,2,3$, \fn{we deduce }
	\begin{align}
	\b|D_{\rho}u(\ell)\b|
	&\leq C \sum_{\ell \in \Lhom} \fn{ \tilde{g}(m)
	\bigg( \sum_{k=1}^{3} \sum_{\sigma\in\Lhom_{*}}\wf_k(|\sigma|) |D_{\sigma}D_{\rho}\Gr(m-\ell)|^k \bigg)^{1/k} }
	\nonumber \\
	&\leq C |\rho| \sum_{m \in \Lhom} \bigg( (1+|m|)^{-s} +  \sum_{k=1}^{3}
	|Du(m)|_{\wf_{k},k}^{2} \bigg) \b(1+|m-\ell|\b)^{-d}. \label{proof:res_a}
	\end{align}
	Then, using the definition \eqref{eq:wnormdef} and \eqref{proof:res_a},
	it follows that
	\begin{align} \label{proof:res_a1}
	&\sum_{j=1}^{3} |Du(\ell)|_{\wf_{j},j} = \sum_{j=1}^{3} \bigg( \sum_{\rho\in\Lhom_{*}}
	\wf_{j}(|\rho|) \big|D_\rho u(\ell)\big|^{j} \bigg)^{1/j} \nonumber
	\\
	& \qquad \quad \leq C \sum_{m \in \Lhom} \bigg( (1+|m|)^{-s} +  \sum_{k=1}^{3}
	|Du(m)|_{\wf_{k},k}^{2} \bigg) \b(1+|m-\ell|\b)^{-d}.
	\qquad
	\end{align}
	We first estimate the linear part in \eqref{proof:res_a1}. For $s> d$, we have
	\begin{eqnarray}\label{proof:res_b}
	\nonumber
	&& \sum_{m \in \Lhom} \b(1+|m|\b)^{-s} \b(1+|m-\ell|\b)^{-d} \\
	\nonumber
	&\leq& C(1+|\ell|)^{-d} \sum_{|m - \ell|\geq|\ell|/2} (1+|m|)^{-s} + C(1+|\ell|)^{-s}\sum_{|m - \ell|\leq|\ell|/2}(1+|m-\ell|)^{-d} \\
	&\leq& C\Big( (1+|\ell|)^{-d} + (1+|\ell|)^{-s}\log(2+|\ell|) \Big)
	~\leq~C(1+|\ell|)^{-d} .
	\end{eqnarray}
	For $0 < s \leq d$, we introduce an exponent $\delta>0$ and estimate
	\begin{align}\label{proof:resb1}
	& \quad
	\sum_{|m-\ell| \geq |\ell|/2} \b(1+|m|\b)^{-s} \b(1+|m-\ell|\b)^{-d} \nonumber
	\\
	&\leq C(1+|\ell|)^{-s+\delta}  \sum_{|m-\ell| \geq |\ell|/2} \b(1+|m|\b)^{-s} \b(1+|m-\ell|\b)^{-d + s-\delta} \nonumber
	\\
	&\leq C(1+|\ell|)^{-s+\delta} \bigg(  \sum_{|m-\ell| \geq |\ell|/2} \b(1+|m|\b)^{-(d+\delta)} \bigg)^{\frac{s}{d+\delta}}
	\bigg(  \sum_{|m-\ell| \geq |\ell|/2} \b(1+|m-\ell|\b)^{-(d+\delta)} \bigg)^{\frac{d+ \delta -s}{d+\delta}} \nonumber
	\\
	&\leq C(1+|\ell|)^{-s+\delta} \sum_{m\in \Lhom}\b(1+|m|\b)^{-(d+\delta)}.
	\end{align}
	Applying the bound $\sum_{m\in\Lhom}(1+|m|)^{-(d+\delta)}\leq C\delta^{-1}$ \fn{gives}
	\begin{align}
	\sum_{|m-\ell| \geq |\ell|/2} \b(1+|m|\b)^{-s} \b(1+|m-\ell|\b)^{-d} &\leq C\b(1+|\ell|\b)^{-s}\frac{(2+|\ell|)^{\delta}}{\delta}  \nonumber
	\\ &\leq C\b(1+|\ell|\b)^{-s}\log(2+|\ell|), \label{proof:resb2}
	\end{align}
	where $\delta=1/\log(2+|\ell|)$ is chosen for the second inequality.
	Therefore, combining \eqref{proof:res_b}--\eqref{proof:resb2} gives
	\begin{eqnarray}\label{proof:res_c}
	\sum_{m\in \Lhom} \b(1+|m|)^{-s} \b(1+|m-\ell|\b)^{-d} \leq C(1+|\ell|)^{-s}\log(2+|\ell|) ,
	\end{eqnarray}
	Collecting the estimates \eqref{proof:res_a1} and \eqref{proof:res_c},
	we obtain
	\begin{align}\label{proof:res_e}
	\sum_{k=1}^{3} |Du(\ell)|_{\wf_{k},k} \leq C \Bigg( z(|\ell|)
	+ \sum_{m\in\Lhom} \b(1+|m-\ell|\b)^{-d} \bigg( \sum_{k=1}^{3} |Du(m)|_{\wf_{k},k}^{2} \bigg) \Bigg),
	\end{align}
	where $z(r)=(1+r)^{-d}$ if $s>d$ and $z(r)=(1+r)^{-s}\log(2+r)$ if $0 < s \leq d$.
	It remains to estimate the nonlinear residual term appearing in \eqref{proof:res_a1}. For $r > 0$, define 
	$w(r):=\sup_{\ell\in\Lhom,|\ell|\geq r}\left(\sum_{k=1}^{3}|Du(\ell)|_{\wf_k,k}\right).$
	Our aim is to show that there exists constant $C > 0$ such that for all $r > 0$
	\begin{align}\label{eq: w decay result}
	w(r) \leq C (1 + r)^{-t},
	\end{align}
	where $t = \min\{ s, \frac{2d}{3} \} > 0$. For any $|m|\geq 2r$, by applying Lemma \ref{lemma:normEquiv}, we deduce
	\begin{align}\label{eq: r w est}
	&\sum_{m\in\Lhom}\b(1+|m-\ell|\b)^{-d} \bigg( \sum_{k=1}^{3} |Du(m)|_{\wf_{k},k}^{2} \bigg)
	\nonumber \\
	&= \sum_{|m|< r}\b(1+|m|\b)^{-d}  \bigg( \sum_{k=1}^{3} |Du(m+\ell)|_{\wf_{k},k}^{2} \bigg)
	+ \sum_{|m|\geq r}\b(1+|m|\b)^{-d} \bigg( \sum_{k=1}^{3} |Du(m+\ell)|_{\wf_{k},k}^{2} \bigg)
	\nonumber \\
	&\leq C w(r)^{3/2}
	\Bigg( \left \|(1+|m|)^{-d}\right\|_{\hc{l}^{4/3}} \bigg( \sum_{k = 1}^{2} \|Du \|_{m^{2}_{\wf_{k},k}}^{1/2} \bigg)
	+ \left\|(1+|m|)^{-d}\right\|_{\hc{l}^{6/5}} \| Du \|_{\hc{l}^3_{\wf_{3},3}}^{1/2} \Bigg)
	\nonumber \\
	& \quad + C(1+r)^{-d} \bigg( \sum_{k=1}^{2}\|Du\|_{m^2_{\wf_k,k}}^{2} \bigg)
	+ C (1+r)^{-2d/3} \|Du\|_{\hc{l}^3_{\wf_3,3}}^{2}
	\nonumber \\
	&\leq C\left( (1+r)^{-2d/3} + w(r)^{3/2} \right).
	\end{align}
	Collecting the estimates \eqref{proof:res_e} and \eqref{eq: r w est} yields
	\begin{eqnarray}\label{proof:res_f}
	w(2r) \leq C(1+r)^{-t} + \eta(r)w(r),
	\end{eqnarray}
	where $\eta(r)=w(r)^{1/2}$ and $\eta(r)\to 0$ as $r\to\infty$, as
	$Du(\ell) \in \hc{l}^2(\Lhom)$.
	From here on, we can follow \cite[Lemma 6.3]{2013-defects-v3} verbatim, but
	show the argument for the sake of completeness.
	For $r > 0$, let $v(r):=\frac{w(r)}{z(r)}$.
	Multiplying \eqref{proof:res_f} with $\frac{2^d}{z(2r)}$, we obtain
	\begin{eqnarray*}
		v(2r)\leq C\b(1+\eta(r)v(r)\b).
	\end{eqnarray*}
	Since $\eta(r)\to 0$ as $r \to \infty$, there exists $r_0>0$ such that for all $r>r_0$
	\begin{eqnarray*}
		v(2r)\leq C+\frac{1}{2}v(r) \qquad\forall~r>r_0.
	\end{eqnarray*}
	Arguing by induction, (c.f. \cite[Lemma 6.3]{2013-defects-v3}),
	we infer that $v$ is bounded on $\R_+$, which implies \eqref{eq: w decay result} holds.
	\begin{eqnarray*}
		w(r)\leq C(1+r)^{-t}  \qquad\forall~r>0,
	\end{eqnarray*}
	and thus implies
	\begin{eqnarray}\label{proof:Du_2d3}
	|Du(\ell)|_{\wf_k,k} \leq C(1+|\ell|)^{-t}, \qquad \text{for } k=1,2,3.
	\end{eqnarray}
	Substituting \eqref{proof:Du_2d3} into \eqref{proof:res_a} gives
	\begin{align}\label{eq: D sigma u est 2}
	\b|D_{\rho}u(\ell)\b|
	&\leq C |\rho|\sum_{m\in \Lhom} \bigg( (1+|m|)^{-s} +  \sum_{k=1}^{3}
	|Du(m)|_{\wf_{k},k}^{2} \bigg) \b(1+|m-\ell|\b)^{-d} \nonumber \\
	&\leq C|\rho| \sum_{m \in \Lhom} \bigg( (1+|m|)^{-s} + (1+|m|)^{-2t} \bigg)
	\b(1+|m-\ell|\b)^{-d},
	\end{align}
	then repeating the argument \eqref{proof:res_b}--\eqref{proof:res_c} with \eqref{eq: D sigma u est 2} and $s' = \min\{ s, \frac{4d}{3} \} > d/2$, we obtain the desired estimate \eqref{decay_Du}.
	This completes the proof.
\end{proof}

\begin{remark}
	From \eqref{eq:gDv}, we can write
	\begin{eqnarray*}
		\b\<g,Dv\b\> = \sum_{\ell\in\Lhom}f(\ell) \cdot v(\ell)
		\qquad{\rm with}\qquad
		f(\ell) = \sum_{\rho\in\Lhom_{*}} D_{-\rho} g_{\rho}(\ell)  .
	\end{eqnarray*}
	If there is an extra decay assumption on $f$ as $|f(\ell)| \leq (1+|\ell|)^{-(s+1)}$,
	then we can drop the log term in \eqref{decay_Du} for the case $d/2<s<d$.
	However, it is not clear whether this assumption can be true (for e.g. point defects case in Lemma \ref{lemma-g-ubar-decay})
	from our existing assumptions in \asSE.
\end{remark}

\section{Proofs: Point Defects}
\label{sec:proof-pd}
\setcounter{equation}{0}

The purpose of this section is to prove Theorems \ref{corr:deltaE0-pd-est} and \ref{theorem:point-defect-decay}. First, we establish results for the homogeneous lattice that will be used in our proofs.

\fn{Recall that in Section \ref{sec:siteE}, we introduced the homogeneous site strain potential $V^{\hom}$. For the sake of convenience, we simply denote $V^{\hom}$ by $V$ in all subsequent proofs.}

\subsection{Proof of Lemma \ref{lemma_lattice_symm}}
\label{sec:proof_lattice_symm}

\begin{proof}
	[Proof of Lemma \ref{lemma_lattice_symm}]
	As we have remarked in Appendix B.1 on Page \pageref{sec:proof_theo_EW}, it is sufficient to consider that {\asSEL} holds under the homogeneous lattice and point defects condition.

	The result follows directly from \cite[Lemma 2.12]{OrtnerTheil2012}, once we establish that for any $u \in \Usz(\Lhom)$, the sum
	\begin{align}
	\sum_{\ell \in \Lhom} \sum_{\rho \in \Lhom_{*}} \hc{V}_{,\rho}(0) \cdot D_{\rho} u(\ell) \label{eq:sumcheck}
	\end{align}
	converges absolutely. As $u \in \Usz(\Lhom)$, $u$ is constant outside $\hc{B_{R}}$, for some $R > 0$. Then, for fixed $\rho \in \Lhom_{*}$, applying Lemma \ref{lemma:pathcount} yields
	\begin{align}
	\sum_{\ell \in \Lhom} |D_{\rho} u(\ell)| &\leq C |\rho| \sum_{\ell \in \Lhom} \sum_{\rho' \in \mathcal{N}(\ell) - \ell} |D_{\rho'} u(\ell)|. \nonumber
	\end{align}
	Moreover, there exists $R' > 0$ such that $\mathcal{N}(\ell) - \ell \subset \hc{B_{R'}}$ for all $\ell \in \L^{\hom}$, hence for $\ell \in \Lhom \cap \hc{B_{R + R'}}^{\rm c}$ and $\rho' \in \mathcal{N}(\ell) - \ell$, $D_{\rho'}u(\ell) = 0$. Consequently, applying Cauchy--Schwarz twice, we deduce
	\begin{align}
	\sum_{\ell \in \Lhom} |D_{\rho} u(\ell)| 
	&\leq C |\rho| \sum_{\ell \in \Lhom \cap \hc{B_{R+R'}}} \sum_{\rho' \in \mathcal{N}(\ell) - \ell} |D_{\rho'} u(\ell)| \nonumber \\ 
	&\leq C |\rho| (R+R')^{{d}/2} \Bigg( \sum_{\ell \in \Lhom \cap \hc{B_{R+R'}}} \bigg( \sum_{\rho' \in \mathcal{N}(\ell) - \ell} |D_{\rho'} u(\ell)| \bigg)^{2} \Bigg)^{1/2} \nonumber \\
	&\leq C |\rho| (R+R')^{{d}/2} \bigg( \sum_{\ell \in \Lhom \cap \hc{B_{R+R'}}} \sum_{\rho' \in \mathcal{N}(\ell) - \ell} |D_{\rho'} u(\ell)|^{2} \bigg)^{1/2} \nonumber \\
	&= C |\rho| (R+R')^{{d}/2} \| Du \|_{\hc{l}^2_{\mathcal{N}}(\Lhom)}, \label{eq:uWc-ell1-est}
	\end{align}
	where the constant $C$ is independent of $\rho$. Applying \eqref{eq:Vloc}, \eqref{eq:uWc-ell1-est} and using that $\wf_{1} \in \Hw_{1}$, we deduce
	\begin{align}
	\sum_{\rho \in \Lhom_{*}} \sum_{\ell \in \Lhom} | \hc{V}_{,\rho}(0) \cdot D_{\rho} u(\ell) |
	&\leq C \sum_{\rho \in \Lhom_{*}} \wf_{1}(|\rho|) \sum_{\ell \in \Lhom} |D_{\rho} u(\ell)| \nonumber \\
	&\leq C (R+R')^{{d}/2} \| Du \|_{\hc{l}^2_{\mathcal{N}}(\Lhom)} \sum_{\rho \in \Lhom_{*}} \wf_{1}(|\rho|) |\rho| < \infty. \label{eq:sumest}
	\end{align}
	As the sum \eqref{eq:sumest} converges, changing the order of summation of \eqref{eq:sumcheck} is allowed and the sum is well-defined. Then as $u \in \Usz(\Lhom)$, it follows that for all $\rho \in \Lhom_{*}$, $\sum_{\ell \in \Lhom} D_{\rho}u(\ell) = 0$, hence the desired result holds
	\begin{align*}
	\sum_{\ell \in \Lhom} \sum_{\rho \in \Lhom_{*}} \hc{V}_{,\rho}(0) \cdot D_{\rho} u(\ell) = \bigg( \sum_{\rho \in \Lhom_{*}} \hc{V}_{,\rho}(0) \bigg) \cdot \bigg( \sum_{\ell \in \Lhom} D_{\rho}u(\ell) \bigg) = 0.
	\end{align*}
	This completes the proof.
\end{proof}

\subsection{Interpolations between defective and homogeneous lattice}
\label{sec:proof_interpolation}

We now introduce three interpolation operators that map displacements from a defective reference
configuration to the corresponding homogeneous lattice and vice versa.
These operators will be used to prove Theorem \ref{corr:deltaE0-pd-est}.

\begin{lemma} \label{Lemma-Int-hom}
	There exists an operator $\Ihom_{1} : \UsH(\L) \to \UsH(\Lhom)$ and $C_{*} > 0$ such that for all $u \in \UsH(\L)$
	\begin{align*}
	\| D\Ihom_{1} u \|_{\hc{l}^2_{\mathcal{N}}(\Lhom)} \leq C_{*} \| Du \|_{\hc{l}^2_{\mathcal{N}}(\L)},
	\end{align*}
	and there exists $R_{0} = R_{0}(u) > 0$ such that $\Ihom u(\ell) = u(\ell)$ for all $|\ell| > R_{0}$.
	Moreover, suppose $u \in \mathscr{H}_{\frak{m},\lambda}(\L)$, where $0 < \frak{m} < 1$, then \fn{there exists $\lambda' = \lambda'(u) > 0$ such that $\Ihom_{1} u \in \mathscr{H}_{\frak{m},\lambda'}(\Lhom)$}.
\end{lemma}

The key property of the interpolation $\Ihom_{1}$ is that it maps $\mathscr{H}(\L)$ to $\mathscr{H}(\Lhom)$, hence for any $u \in \mathscr{H}(\L)$, the locality and homogeneity estimates \eqref{eq:Vloc}, \eqref{eq:Vhomogeneity} continue to hold for $\Ihom_{1} u$.

We also define the following pair of linear interpolations.

\begin{lemma} \label{lemma-Ihom-norm-est}
	There exists a bounded linear operator $\Ihom_{2}:\UsH(\L)\rightarrow\UsH(\Lhom)$ and $C > 0$ such that for all $u \in \UsH(\L)$
	\begin{align*}
	\| D\Ihom_{2} u \|_{\hc{l}^2_{\mathcal{N}}(\Lhom)} \leq C \| Du \|_{\hc{l}^2_{\mathcal{N}}(\L)},
	\end{align*}
	and $\Ihom_{2} u(\ell) = u(\ell)$ for all $|\ell| > \Rcore$, $u \in \UsH(\L)$.

Moreover, there exists $c_{0} \in (0,1]$ such that for all $r > 0$, there exists $R \geq r$ satisfying: for all $\wf \in \Hw_{1}$
	\begin{align*}
	\sum_{\ell \in \Lhom \cap B_{r}} |D \Ihom_{2} u(\ell)|_{\wf,1} &\leq C \sum_{\ell \in \L \cap B_{R}} |D u(\ell)|_{\wt \wf,1},
	\end{align*}
	where $\wt \wf(r) := \wf(c_{0}r) \in \Hw_{1}$.
\end{lemma}

The main differences between $\Ihom_{1}, \Ihom_{2}$ is that $\Ihom_{2}$ is linear and in general $\Ihom_{2}$ does not map $\mathscr{H}(\L)$ to $\mathscr{H}(\Lhom)$ and that $\Ihom_{2}u(\ell) = u(\ell)$ for all $|\ell| > \Rcore$, where the constant $\Rcore$ is independent of $u \in \UsH(\L)$ whereas the constant $R_{0}$ appearing in Lemma \ref{Lemma-Int-hom} is dependent on $u \in \UsH(\L)$.

One can also define interpolations from $\UsH(\Lhom)$ to $\UsH(\L)$, analogously to Lemmas \ref{Lemma-Int-hom} and \ref{lemma-Ihom-norm-est}. However, for the purpose of our analysis, we only require the following interpolation.

\begin{lemma} \label{lemma-Iback-norm-est}
	There exists a bounded linear operator $\Iback:\UsH(\Lhom)\rightarrow\UsH(\L)$ and $C > 0$ such that for all $v^{\hom} \in \UsH(\Lhom)$
	\begin{align*}
	\| D\Iback v^{\hom} \|_{\hc{l}^2_{\mathcal{N}}(\L)} \leq C \| Dv^{\hom} \|_{\hc{l}^2_{\mathcal{N}}(\Lhom)},
	\end{align*}
	and $\Iback v^{\hom}(\ell) = v^{\hom}(\ell)$ for all $|\ell| > \Rcore$, $v^{\hom} \in \UsH(\Lhom)$.

	Moreover, there exists $\fn{c_{0}'} \in (0,1]$ such that for all $r > 0$, there exists $R \geq r$ satisfying: for all $\wf \in \Hw_{1}$
	\begin{align*}
	\sum_{\ell \in \L \cap B_{r}} |D \Iback v^{\hom}(\ell)|_{\wf,1} &\leq C \sum_{\ell \in \Lhom \cap B_{R}} |D v^{\hom}(\ell)|_{\wt \wf,1},
	\end{align*}
	\fn{	where $\wt \wf(r) := \wf(\fn{c_{0}'}r) \in \Hw_{1}$.}
\end{lemma}


\begin{proof}[Proof of Lemma \ref{Lemma-Int-hom}]
	Let $u \in \mathscr{H}(\L)$, hence there exist $\lambda > 0$ and $0 < \frak{m} < 1$ such that $y_{0} + u \in \Adm_{\frak{m},\lambda}(\L)$. Alternatively, when $u \in \UsH(\L) \setminus \mathscr{H}(\L)$, we simply choose $\frak{m} = 1/2$. As $\| Du \|_{\ell^{2}_{\mathcal{N}}(\L)} < \infty$, for all $\varepsilon > 0$, there exists ${R_{1} = R_{1}(\varepsilon) > \Rcore}$ such that
	for all $R \geq R_{1}$
	\begin{align}
	\| Du \|_{\ell^{2}_{\mathcal{N}}(\L \setminus B_{R})} = \bigg( \sum_{|\ell| > R} \sum_{\rho\in \mathcal{N}(\ell) - \ell}
	\big|D_\rho u(\ell)\big|^2 \bigg)^{1/2} < \varepsilon. \label{eq:Du-eps-est}
	\end{align}
	Using that $\L \setminus B_{\Rcore} = \Lhom \setminus B_{\Rcore}$ and $R_{1} > \Rcore$, we define $u_{1} \in \UsH(\Lhom)$ by
	\begin{align*}
	u_{1}(\ell) := \left\{
	\begin{array}{ll}
	u(\ell) & \text{if } |\ell| > R_{1}, \\
	0 & \text{otherwise.}
	\end{array}
	\right.
	\end{align*}
	Due to the boundary terms introduced by the truncation, it is not in general possible to estimate the norm $\| Du_{1} \|_{\ell^{2}_{\mathcal{N}}(\Lhom)}$ in terms of $\| Du\|_{\ell^{2}_{\mathcal{N}}(\L)}$, however, there exists ${R_{2} = R_{2}(\varepsilon) > R_{1}}$ such that choosing $|\ell| > R_{2}$ ensures that $D_{\rho}u_{1}(\ell) = D_{\rho}u(\ell)$ for all ${\rho \in \mathcal{N}(\ell) - \ell}$, hence for any subset $A \subset B_{R_{2}^{\rm c}}$, we have ${\| Du_{1} \|_{\ell^{2}_{\mathcal{N}}(\Lhom  \cap A)} = \| Du \|_{\ell^{2}_{\mathcal{N}}(\L \cap A)}}$.
	
	For any $v \in \UsH(\Lhom)$, it is possible to define a triangulation of $\Lhom$ and construct an interpolant \cite{2013-defects-v3,ortner_shapeev12}
	\begin{align*}
	\wt v \in \dot{W}^{1,2} := \big\{ w: \mathbb{R}^{d} \to \mathbb{R}^{d}\, | \, w \in W^{1,2}_{\rm loc}, \nabla w \in L^{2} \big\},
	\end{align*}
	satisfying $\wt v(\ell) = v(\ell)$ for all $\ell \in \Lhom$ and there exist constants $C, R_{3} > 0$ such that for all $v \in \UsH(\Lhom)$ and $R \geq R_{3}$,
	\begin{align}
	\| \nabla \wt v \|_{L^{2}(B_{R/2}^{\rm c})} \leq C \| Dv \|_{\ell^{2}_{\mathcal{N}}(\Lhom \setminus B_{R/4})}. 
	\label{eq:int-est-1}
	\end{align}
	Additionally, by \cite[Lemma 3]{ortner_shapeev12} there exists $C', R_{4} > 0$ such that for all $v \in \UsH(\Lhom)$ and $\ell \in \Lhom$
	\begin{align*}
	|v(\ell)| = |\wt v(\ell)| \leq C' \| \wt v \|_{L^{2}(B_{R_{4}}(\ell))}.
	\end{align*}
	Suppose that $R \geq 2R_{4}$ and define $A_{R} = B_{5R/2} \setminus B_{R/2}$, then we deduce that for all $v \in \UsH(\Lhom)$
	\begin{align}
	\sum_{\substack{\ell \in \Lhom \\ R < |\ell| < 2R }} |v(\ell)|^{2} \leq C C_{0} \| \wt v \|_{L^{2}(A_{R})}^{2} = C \| \wt v \|_{L^{2}(A_{R})}^{2}, \label{eq:v-L2-est}
	\end{align}
	where the final constant is independent of $R > 0$. We remark that the constants $R_{3}, R_{4}$ are independent of $\varepsilon > 0$.
	
	Choose $\eta \in C^{\infty}(\mathbb{R}^{d})$ satisfying $0 \leq \eta \leq 1$, $\eta = 0$ on $B_{1}, \eta = 1$ over $B_{2}^{\rm c}$ and $\| \eta \|_{W^{1,\infty}(\mathbb{R}^{d})} \leq C_{\eta}$, then for $R \geq R_{5} = R_{5}(\varepsilon) := \max\{4R_{2}, R_{3}, 2R_{4} \}$, define $\eta_{R}(x) := \eta(x/R)$. Observe that $\eta_{R}$ satisfies $|\nabla \eta_{R}(x)| \leq C_{\eta} R^{-1}$ for all $x \in \mathbb{R}^{d}$. Now define
	$\wt u_{R} \in \dot{W}^{1,2}$ by
	\begin{align*}
	\wt u_{R}(x) = \eta_{R}(x) \, ( \, \wt u_{1}(x) - \langle \, \wt u_{1} \rangle_{A_{R}} ) + \langle \, \wt u_{1} \rangle_{A_{R}},
	\end{align*}
	where $\langle \, \wt u_{1} \rangle_{A_{R}} = |A_{R}|^{-1} \int_{A_{R}} \wt u_{1}$ denotes the average integral of $\wt u_{1}$ over $A_{R}$.
	We now define $u_{R} \in \UsH(\Lhom)$ by $u_{R}(\ell) = \wt u_{R}(\ell)$ for all $\ell \in \Lhom.$ We now estimate the norm
	\begin{align} \label{eq:mark}
	\| Du_{R} \|_{\ell^{2}(\Lhom)}^{2} = \sum_{\ell \in \Lhom} \sum_{\rho \in \mathcal{N}(\ell) - \ell} |D_{\rho}u_{R}(\ell)|^{2}.
	\end{align}
	
	We remark that while similar estimates have been performed in \cite{2013-defects-v3}, we additionally require that our construction satisfies $\Ihom_{1} u \in \mathscr{H}(\Lhom)$. In particular, this prevents the accumulation of atoms, which is necessary to treat the models given in Section \ref{sec:examples}. Hence, we need to establish more careful estimates.
	
	We can decompose the above sum into the following terms
	\begin{align}
	\| Du_{R} \|_{\ell^{2}(\Lhom)}^{2} &\leq \sum_{|\ell| < R} \sum_{ \substack{\rho \in \mathcal{N}(\ell) - \ell \\ |\ell + \rho| < R}} |D_{\rho}u_{R}(\ell)|^{2} + \sum_{|\ell| > 2R} \sum_{ \substack{\rho \in \mathcal{N}(\ell) - \ell \\ |\ell + \rho| > 2R}} |D_{\rho}u_{R}(\ell)|^{2} \label{eq:DuR-decomp1} \\ & \qquad + 3 \sum_{R < |\ell| < 2R} \sum_{ \rho \in \mathcal{N}(\ell) - \ell} |D_{\rho}u_{R}(\ell)|^{2}. \label{eq:DuR-decomp2}
	\end{align}
	As $\eta_{R} = 0$ over $B_{R}$, it follows that $D_{\rho}u_{R}(\ell) = 0$ whenever $|\ell|, |\ell + \rho| < R$, therefore the first term in \eqref{eq:DuR-decomp1} vanishes
	\begin{align}
	\sum_{|\ell| < R} \sum_{ \substack{\rho \in \mathcal{N}(\ell) - \ell \\ |\ell + \rho| < R}} |D_{\rho}u_{R}(\ell)|^{2} = 0. \label{est:part1-final}
	\end{align}
	Similarly, using that $\eta_{R} = 1$ on $B_{2R}^{\rm c}$, it follows that ${D_{\rho}u_{R}(\ell) = D_{\rho} \wt u_{1}(\ell) = D_{\rho}u_{1}(\ell)}$ for all $|\ell|, |\ell + \rho| > 2R$, hence the second term of \eqref{eq:DuR-decomp1} can be estimated by
	\begin{align}
	\sum_{|\ell| > 2R} \sum_{ \substack{\rho \in \mathcal{N}(\ell) - \ell \\ |\ell + \rho| > 2R}} |D_{\rho}u_{R}(\ell)|^{2} &= \sum_{|\ell| > 2R} \sum_{ \substack{\rho \in \mathcal{N}(\ell) - \ell \\ |\ell + \rho| > 2R}} |D_{\rho}u_{1}(\ell)|^{2} \leq \sum_{|\ell| > 2R} \sum_{ \rho \in \mathcal{N}(\ell) - \ell} |D_{\rho}u_{1}(\ell)|^{2} \nonumber \\ &= \| Du_{1} \|_{\ell^{2}_{\mathcal{N}}(\Lhom \setminus B_{2R})}^{2}. \label{est:part2-final}
	\end{align}
	It remains to estimate the term \eqref{eq:DuR-decomp2}, we first express
	\begin{align*}
	D_{\rho}u_{R}(\ell) = ( \eta_{R}(\ell+\rho) - \eta_{R}(\ell) ) \, ( \, \wt u_{1}(\ell) - \langle \, \wt u_{1} \rangle_{A_{R}} ) + \eta_{R}(\ell + \rho) D_{\rho} \wt u_{1}(\ell),
	\end{align*}
	which can be estimated by
	\begin{align}
	|D_{\rho}u_{R}(\ell)| &= | \eta_{R}(\ell+\rho) - \eta_{R}(\ell) | | \, \wt u_{1}(\ell) - \langle \, \wt u_{1} \rangle_{A_{R}} | + | D_{\rho} \wt u_{1}(\ell)| \nonumber \\
	&\leq \bigg| \int_{0}^{1} \nabla \eta_{R}(\ell+t\rho) \text{ d}t \bigg| | \, \wt u_{1}(\ell) - \langle \, \wt u_{1} \rangle_{A_{R}} | + | D_{\rho} u_{1}(\ell)| \nonumber \\
	&\leq CR^{-1} |\rho| | \, \wt u_{1}(\ell) - \langle \, \wt u_{1} \rangle_{A_{R}} | + | D_{\rho} u_{1}(\ell)| \nonumber \\
	&\leq CR^{-1} | \, \wt u_{1}(\ell) - \langle \, \wt u_{1} \rangle_{A_{R}} | + | D_{\rho} u_{1}(\ell)|. \label{eq:DuR-int-est-1}
	\end{align}
	Applying \eqref{eq:DuR-int-est-1} to \eqref{eq:DuR-decomp2} yields
	\begin{align}
	&\sum_{R < |\ell| < 2R} \sum_{ \rho \in \mathcal{N}(\ell) - \ell} |D_{\rho}u_{R}(\ell)|^{2} \nonumber \\ & \quad \leq C R^{-2} \sum_{R < |\ell| < 2R} \sum_{ \rho \in \mathcal{N}(\ell) - \ell} | \, \wt u_{1}(\ell) - \langle \, \wt u_{1} \rangle_{A_{R}} |^{2} + C\sum_{R < |\ell| < 2R} \sum_{ \rho \in \mathcal{N}(\ell) - \ell} |D_{\rho}u_{1}(\ell)|^{2} \nonumber \\
	& \quad \leq C R^{-2} \sum_{R < |\ell| < 2R} | \, \wt u_{1}(\ell) - \langle \, \wt u_{1} \rangle_{A_{R}} |^{2} + C \| Du_{1} \|_{\ell^{2}_{\mathcal{N}}(\Lhom \setminus B_{R})}^{2}. \label{eq:part3-est-1}
	\end{align}
	We observe that applying \eqref{eq:v-L2-est} to $v = u_{1} - \langle \, \wt u_{1} \rangle_{A_{R}}$, which satisfies $\wt v = \wt u_{1} - \langle \, \wt u_{1} \rangle_{A_{R}}$, gives
	\begin{align}
	C R^{-2} \sum_{R < |\ell| < 2R} | \, \wt u_{1}(\ell) - \langle \, \wt u_{1} \rangle_{A_{R}} |^{2} \leq C R^{-2} \| \, \wt u_{1} - \langle \, \wt u_{1} \rangle_{A_{R}} \|_{L^{2}(A_{R})}^{2}. \label{eq:part3-est-2}
	\end{align}
	By applying Poincar\'{e}'s inequality on $A_{1}$ and using a standard scaling argument, we deduce that
	\begin{align}
	C R^{-2} \| \, \wt u_{1} - \langle \, \wt u_{1} \rangle_{A_{R}} \|_{L^{2}(A_{R})}^{2} &\leq (C R^{-2})( C_{P} R^{2}) \| \nabla \wt u_{1} \|_{L^{2}(A_{R})}^{2} \nonumber \\ &\leq C \| \nabla \wt u_{1} \|_{L^{2}(B_{R/2}^{\rm c})}^{2}, \label{eq:part3-est-3}
	\end{align}
	hence by inserting \eqref{eq:part3-est-2}--\eqref{eq:part3-est-3} into \eqref{eq:part3-est-1} and applying \eqref{eq:int-est-1}, we infer that
	\begin{align}
	\sum_{R < |\ell| < 2R} \sum_{ \rho \in \mathcal{N}(\ell) - \ell} |D_{\rho}u_{R}(\ell)|^{2} &\leq C \| \nabla \wt u_{1} \|_{L^{2}(B_{R/2}^{\rm c})}^{2} \leq C \| D u_{1} \|_{\ell^{2}_{\mathcal{N}}(\Lhom \setminus B_{R/4})}^{2}, \label{eq:part3-est-final}
	\end{align}
	where the final constant is independent of $R > 0$.
	
	Combining the estimates \eqref{est:part1-final},\eqref{est:part2-final} and \eqref{eq:part3-est-final} with \eqref{eq:DuR-decomp1}--\eqref{eq:DuR-decomp2}, we deduce that
	\begin{align}
	\| Du_{R} \|_{\ell^{2}(\Lhom)} &\leq C \| D u_{1} \|_{\ell^{2}_{\mathcal{N}}(\Lhom \setminus B_{R/4})} = C \| D u \|_{\ell^{2}_{\mathcal{N}}(\Lhom \setminus B_{R/4})} \leq C_{*} \| Du \|_{\ell^{2}_{\mathcal{N}}(\Lhom)}. \label{eq:IR-bound}
	\end{align}
	
	
	We remark that collecting the estimates \eqref{eq:part3-est-2}--\eqref{eq:part3-est-final} and applying \eqref{eq:Du-eps-est} with \newline ${R/4 \geq R_{2} > R_{1}}$, we infer that
	\begin{align}
	\| \, u_{1} - \langle \, \wt u_{1} \rangle_{A_{R}} \|_{\ell^{2}(\Lhom \cap ( B_{2R} \setminus B_{R} ) ) } \leq CR \| Du \|_{\ell^{2}_{\mathcal{N}}(\Lhom \setminus B_{R/4})} \leq CR \, \varepsilon. \label{eq:u1-ell2-est}
	\end{align}
	It remains to show that we can choose $R > 0$ sufficiently large to ensure that for all $\ell, m \in \Lhom$,
	\begin{align}
	|\ell - m + u_{R}(\ell) - u_{R}(m)| \geq \frak{m} |\ell - m|. \label{eq:uR-m-est}
	\end{align}
	Once \eqref{eq:uR-m-est} has been shown, it follows that $y_{0}^{\hom} + u_{R} \in \Adm_{\frak{m},C_{*}\lambda}(\Lhom)$ for sufficiently large $R > 0$.
	
	We prove \eqref{eq:uR-m-est} by consider four distinct cases.
	
	\emph{Case 1} Suppose that $\ell, m \in B_{R}$, then as $\eta_{R} = 0$ on $B_{R}$, it follows that $u_{R}(\ell) = u_{R}(m)$, hence
	\begin{align}
	|\ell - m + u_{R}(\ell) - u_{R}(m)| = |\ell - m| \geq \frak{m} |\ell - m|. \label{eq:case1}
	\end{align}
	\emph{Case 2} Suppose that $|\ell| > 2R$ and $|m| < R$, which implies that $|\ell - m| > R$, then by Lemma \ref{lemma:pathcount}, there exists a path $\mathscr{P}(\ell,m) = \{\ell_{i} \in \L |1\leq i \leq N_{\ell,m}+1\}$ of neighbouring lattice points, such that $N_{\ell,m} \leq C |\ell - m|$ and $\rho_{i} := \ell_{i+1} - \ell_{i} \in \mathcal{N}(\ell_{i}) - \ell_{i}$ for all $1 \leq i \leq N_{\ell,m}$, satisfying
	$|u_{R}(\ell) - u_{R}(m)| \leq \sum_{i = 1}^{N_{\ell,m}} |D_{\rho_{i}} u_{R}(\ell_{i})|.$
	Applying Cauchy-Schwarz
	gives
	\begin{align}
	|u_{R}(\ell) - u_{R}(m)| &\leq N_{\ell,m}^{1/2} \bigg( \sum_{i = 1}^{N_{\ell,m}} |D_{\rho_{i}} u_{R}(\ell_{i})|^{2} \bigg)^{1/2} \leq C|\ell - m|^{1/2} \| Du_{R} \|_{\ell^{2}_{\mathcal{N}}(\L)} \nonumber \\ &\leq C C_{*} \lambda |\ell - m|^{1/2} = C_{1} \lambda |\ell - m|^{1/2}. \label{eq:case1-path}
	\end{align}
	Choosing $R \geq R_{6} = (\frac{C_{1} \lambda}{(1 - \frak{m})})^{2}$ ensures that $(1 - \frak{m})|\ell - m|^{1/2} \geq (1 - \frak{m})R^{1/2} \geq C_{1} \lambda$, hence
	\begin{align}
	|\ell + m + u_{R}(\ell) - u_{R}(m)| &\geq |\ell - m| -|u_{R}(\ell) - u_{R}(m)| \geq |\ell - m| - C_{1} \lambda |\ell - m|^{1/2} \nonumber \\ &= |\ell - m|^{1/2} \,( |\ell - m|^{1/2} - C_{1} \lambda ) \geq \frak{m} |\ell - m|. \label{eq:case2}
	\end{align}
	\emph{Case 3} Now suppose $\ell, m \in B_{R}^{\rm c}$, then applying the estimate \eqref{eq:DuR-int-est-1} gives
	\begin{align}
	|u_{R}(\ell) - u_{R}(m)| &\leq |\eta_{R}(\ell) - \eta_{R}(m)|| \, u_{1}(\ell) - \langle \, \wt u_{1} \rangle_{A_{R}} | + |\eta_{R}(m)||u_{1}(\ell) - u_{1}(m)| \nonumber \\
	&\leq CR^{-1} |\ell - m| | \, u_{1}(\ell) - \langle \, \wt u_{1} \rangle_{A_{R}} | + |u_{1}(\ell) - u_{1}(m)|. \label{eq:case3-est-part1}
	\end{align}
	In order to estimate the first term of \eqref{eq:case3-est-part1}, we apply \eqref{eq:u1-ell2-est} and the embedding $\ell^{2} \subset \ell^{\infty}$ to deduce
	\begin{align}
	CR^{-1} | \, u_{1}(\ell) - \langle \, \wt u_{1} \rangle_{A_{R}} | &\leq CR^{-1} \| \, u_{1} - \langle \, \wt u_{1} \rangle_{A_{R}} \|_{\ell^{\infty}(\Lhom \cap ( B_{2R} \setminus B_{R} ) ) } \nonumber \\
	&\leq CR^{-1} \| \, u_{1} - \langle \, \wt u_{1} \rangle_{A_{R}} \|_{\ell^{2}(\Lhom \cap ( B_{2R} \setminus B_{R} ) ) } \nonumber \\
	&\leq (CR^{-1})(CR\varepsilon) = C_{1} \varepsilon, \label{eq:case3-est-part2}
	\end{align}
	where the constant $C_{1}$ is independent of $R > 0$.
	
	We estimate the second term appearing in \eqref{eq:case3-est-part1} using the argument of \eqref{eq:case1-path}. By Lemma \ref{lemma:pathcount}, there exists a path $\mathscr{P}(\ell,m) = \{\ell_{i} \in \L |1\leq i \leq N_{\ell,m}+1\}$ of neighbouring lattice points, such that $N_{\ell,m} \leq C |\ell - m|$ and $\rho_{i} := \ell_{i+1} - \ell_{i} \in \mathcal{N}(\ell_{i}) - \ell_{i}$ for all $1 \leq i \leq N_{\ell,m}$, satisfying
	$|u_{1}(\ell) - u_{1}(m)| \leq \sum_{i = 1}^{N_{\ell,m}} |D_{\rho_{i}} u_{1}(\ell_{i})|.$ Moreover, the path $\mathscr{P}(\ell,m)$ can be chosen to ensure that $|\ell_{i}| > R$ for all $1 \leq i \leq N_{\ell,m} + 1$.
	Applying Cauchy-Schwarz gives
	\begin{align*}
	|u_{1}(\ell) - u_{1}(m)| &\leq N_{\ell,m}^{1/2} \bigg( \sum_{i = 1}^{N_{\ell,m}} |D_{\rho_{i}} u_{1}(\ell_{i})|^{2} \bigg)^{1/2} \\ &\leq C|\ell - m|^{1/2} \bigg( \sum_{\ell \in \Lhom \setminus B_{R} } \sum_{\rho \in \mathcal{N}(\ell) - \ell} |D_{\rho} u_{1}(\ell)|^{2} \bigg)^{1/2} \nonumber \\ &\leq C |\ell - m|^{1/2} \| Du_{1} \|_{\ell^{2}_{\mathcal{N}}(\Lhom \setminus B_{R})}.
	\end{align*}
	Using \eqref{eq:Du-eps-est} and that $c_{1} := \inf_{\substack{\ell , m \in \L \\ \ell \neq m}} |\ell - m| > 0$, we deduce that
	\begin{align}
	|u_{1}(\ell) - u_{1}(m)| &\leq C|\ell - m|^{1/2} \| Du_{1} \|_{\ell^{2}_{\mathcal{N}}(\Lhom \setminus B_{R})} \nonumber \\
	&\leq C c_{1}^{-1/2} \varepsilon |\ell - m| = C_{2} \varepsilon |\ell - m|. \label{eq:case3-est-part3}
	\end{align}
	Now let $\varepsilon_{0} = (C_{1} + C_{2})^{-1}(1 - \frak{m})$, then for $R \geq \max\{ R_{5}(\varepsilon_{0}), R_{6} \}$, collecting the estimates \eqref{eq:case3-est-part1}--\eqref{eq:case3-est-part3} gives
	\begin{align*}
	|u_{R}(\ell) - u_{R}(m)| \leq (C_{1} + C_{2}) \varepsilon |\ell - m| \leq (1 - \frak{m}) |\ell - m|,
	\end{align*}
	hence we obtain the desired estimate
	\begin{align}
	|\ell - m + u_{R}(\ell) - u_{R}(m)| \geq |\ell - m| - |u_{R}(\ell) - u_{R}(m)| \geq \frak{m}|\ell - m|. \label{eq:case3}
	\end{align}
	\emph{Case 4} In our final case, we consider $\ell \in B_{2R} \setminus B_{R}$ and $m \in B_{R}$. We follow the argument of Case 3, while observing that as $\eta_{R}(m) = 0$, we obtain
	\begin{align*}
	|u_{R}(\ell) - u_{R}(m)| &\leq |\eta_{R}(\ell) - \eta_{R}(m)|| \, u_{1}(\ell) - \langle \, \wt u_{1} \rangle_{A_{R}} | + |\eta_{R}(m)||u_{1}(\ell) - u_{1}(m)| \nonumber \\
	&\leq |\eta_{R}(\ell) - \eta_{R}(m)|| \, u_{1}(\ell) - \langle \, \wt u_{1} \rangle_{A_{R}} | \\
	&\leq CR^{-1} |\ell - m| | \, u_{1}(\ell) - \langle \, \wt u_{1} \rangle_{A_{R}} | \\ &\leq C_{1} \varepsilon_{0} |\ell - m| \leq (1 - \frak{m})|\ell - m|.
	\end{align*}
	hence we obtain the desired estimate
	\begin{align}
	|\ell - m + u_{R}(\ell) - u_{R}(m)| \geq |\ell - m| - |u_{R}(\ell) - u_{R}(m)| \geq \frak{m}|\ell - m|. \label{eq:case4}
	\end{align}
	Collecting the estimates \eqref{eq:case1}--\eqref{eq:case4} implies \eqref{eq:uR-m-est}.
	
	Finally, we define ${I^{\hom}_{1}u := u_{R} \in \UsH(\Lhom)}$, where $R = \max\{ R_{5}(\varepsilon_{0}), R_{6} \}$, which satisfies $y_{0}^{\hom} + I^{\hom}_{1}u \in \Adm_{\frak{m}, C_{*}\lambda}$ and $I^{\hom}_{1}u(\ell) = u(\ell)$ for all $|\ell| > R_{0} := 2R$.
	This completes the proof.
\end{proof}

\fn{
\begin{proof}[Proof of Lemma \ref{lemma-Ihom-norm-est}]
	Fix $\wf \in {\mathscr{L}_{1}}$,
	$u \in \UsH(\L)$ and for $\ell \in \Lhom \setminus B_{\Rcore}$, define ${\Ihom_{2} u(\ell) = u(\ell)}$. Due to the periodicity of $\Lhom$, for each $\ell \in \Lhom$ there exists a bounded Voronnoi cell $\mathcal{V}^{\hom}(\ell) \subset B_{\widetilde{R}}(\ell)$, where $\widetilde{R} = \smfrac{1}{2} \sum_{i = 1}^{3} |A e_{i}| > 0$, satisfying
	\begin{align}
	\mathcal{V}^{\hom}(\ell) &= \left\{ \, x \in \mathbb{R}^{3} \, \Big | \, |x - \ell| \leq |x - k| \quad \forall \, k \in \Lhom \,  \right\}
	\end{align}
	and $\mathcal{V}^{\hom}(\ell) = \mathcal{V}^{\hom}(0) + \ell$, due to the translation-invariance of $\Lhom$. Using the definition \eqref{eq:Voronoi-def}, one may also define the Voronoi cell $\mathcal{V}(\ell)$ for $\ell \in \L$. As $\L \setminus B_{\Rcore} = \Lhom \setminus B_{\Rcore}$, there exists $R' > 0$ such that $\ell \in \L \setminus B_{R'}$ guarantees $\mathcal{V}(\ell) = \mathcal{V}^{\hom}(\ell) \subset B_{\widetilde{R}}(\ell)$, which is bounded. As $\{ \mathcal{V}(\ell) | \ell \in \L \}$ cover $\mathbb{R}^{3}$, it follows from the definition \eqref{eq:Voronoi-def} that for $\ell \in \L \cap B_{R's}$
	\begin{align}
	\mathcal{V}(\ell) \subset \bigcup_{\ell' \in \L \cap B_{R'}} \mathcal{V}(\ell') \subset \left( \bigg( \bigcup_{\ell' \in \L \setminus B_{R'}} \mathcal{V}(\ell') \bigg)^{\mathrm{o}} \right)^{\rm c}, \label{Vell Int bound}
	\end{align}
	where $X^{\mathrm{o}}$ denotes the interior of the set $X$. As the right-hand side of \eqref{Vell Int bound} is a bounded set, there exists $R_{0} > 0$ such that $\mathcal{V}(\ell) \subset B_{R_{0}}(\ell)$ for all $\ell \in \L$. In addition, there exists $R_{1} > 0$ such that
	\begin{align}
	\Lhom \cap B_{\Rcore} \subset \bigcup_{\ell' \in \L \cap B_{R_{1}}} \mathcal{V}(\ell'),
	\end{align}
	hence for each $\ell \in \Lhom \cap B_{\Rcore}$, there exists $\ell' \in \L \cap B_{R_{1}}$ such that ${\ell \in \mathcal{V}(\ell') \subset B_{R_{0}}(\ell')}$. Note that the choice of $\ell'$ is in general not unique. Then define $\Ihom_{2} u(\ell) = u(\ell')$. It follows from the construction that $\Ihom_{2}$ is linear.
	
	Consider distinct $\ell_{1}, \ell_{2} \in \Lhom \cap B_{\Rcore}$, then for $i = 1,2$ there exist $\ell_{i}' \in \L \cap B_{R_{1}}$ such that $\Ihom_{2} u(\ell_{i}) = u(\ell'_{i})$ and $|\ell_{i} - \ell_{i}'| \leq R_{0}$. It holds that \newline ${|\ell_{1} - \ell_{2}| \geq \min_{\ell,\ell' \in \L, \ell \neq \ell'} |\ell - \ell'| =: c_{1} > 0}$ and similarly that \newline ${|\ell_{1}' - \ell_{2}'| \leq \max_{\ell,\ell' \in \Lhom, \ell \neq \ell'} |\ell - \ell'| =: c_{2} > 0}$, hence
	$|\ell_{1} - \ell_{2}| \geq \smfrac{c_{1}}{c_{2}}|\ell_{1}' - \ell_{2}'| =: c_{0}|\ell_{1}' - \ell_{2}'|$, where $c_{0} \in (0,1]$. Using that $\wf \in \mathscr{H}^{1}$, it follows that $\wf$ is decreasing, hence
	\begin{align}
	\wf(|\ell_{1} - \ell_{2}|) \,|\Ihom_{2} u(\ell_{1}) - \Ihom_{2} u(\ell_{2})| &\leq \wf(c_{0}|\ell_{1}' - \ell_{2}'|) \,|u(\ell'_{1}) - u(\ell'_{2})| \nonumber \\ & \leq \sum_{\ell' \in \L \cap B_{R_{1}}} \wf(c_{0}|\ell_{1}' - \ell'|) \,|u(\ell'_{1}) - u(\ell')|. \label{eq:Ihom-12-exp-est1}
	\end{align}
	We then define $\wt \wf(r) := \wf(c_{0}r)$, and observe that $\wt \wf \in \mathscr{H}^{1}$. In the case $\ell_{1} \in \Lhom \cap B_{\Rcore}$ and $\ell_{2} \in \Lhom \setminus B_{\Rcore}$, then $\Ihom_{2} u(\ell_{2}) = u(\ell_{2})$, then a similar argument shows
	\begin{align}
	\wf(|\ell_{1} - \ell_{2}|) \,|\Ihom_{2} u(\ell_{1}) - \Ihom_{2} u(\ell_{2})| &\leq \wf(c_{0}|\ell_{1}' - \ell_{2}'|) \,|u(\ell'_{1}) - u(\ell'_{2})|. \label{eq:Ihom-12-exp-est2}
	\end{align}
	Now, we decompose
	\begin{align*}
	&|D \Ihom_{2} u(\ell_{1})|_{\wf,1} = \sum_{\ell_{2} \in \Lhom} \wf(|\ell_{1} - \ell_{2}|) \, | \Ihom_{2} u(\ell_{1}) - \Ihom_{2} u(\ell_{2})| \nonumber \\
	& = \sum_{\ell_{2} \in \Lhom \cap B_{\Rcore}} \hspace{-15pt} \wf(|\ell_{1} - \ell_{2}|) \, | \Ihom_{2} u(\ell_{1}) - \Ihom_{2} u(\ell_{2})| + \hspace{-9pt} \sum_{\ell_{2} \in \Lhom \setminus B_{\Rcore}} \hspace{-5pt} \wf(|\ell_{1} - \ell_{2}|) \,| \Ihom_{2} u(\ell_{1}) - \Ihom_{2} u(\ell_{2})|,
	\end{align*}
	then apply \eqref{eq:Ihom-12-exp-est1}--\eqref{eq:Ihom-12-exp-est2} to deduce
	\begin{align}
	|D \Ihom_{2} u(\ell_{1})|_{\wf,1} &\leq C |\Lhom \cap B_{\Rcore}| \sum_{\ell' \in \L \cap B_{R_{1}}} \wf(c_{0}|\ell_{1} - \ell_{2}|) \, | u(\ell_{1}') - u(\ell')| \nonumber \\ & \qquad + C \sum_{ \ell' \in \L \setminus B_{\Rcore}} \wf(c_{0}|\ell_{1} - \ell_{2}|) \, | u(\ell_{1}') - u(\ell_{2}')| \nonumber \\
	&\leq C \sum_{\ell_{2}' \in \L \setminus \ell_{1}'} \wt \wf(|\ell_{1} - \ell_{2}|) \, | u(\ell_{1}') - u(\ell_{2}')| = C \sum_{\rho' \in \L - \ell_{1}'} \wt \wf(|\rho|) \, | D_{\rho'} u(\ell_{1}')| \nonumber \\ &= C |D u(\ell_{1}')|_{\wt \wf,1} \leq C \sum_{\ell' \in \L \cap B_{R_{1}}} |D u(\ell')|_{\wt \wf,1}. \label{eq:Ihom-12-exp-est3}
	\end{align}
	An identical argument shows that for $\ell_{1} \in \Lhom \setminus B_{\Rcore}$
	\begin{align}
	&|D \Ihom_{2} u(\ell_{1})|_{\wf,1} \leq C |D u(\ell_{1})|_{\wt \wf,1}. \label{eq:Ihom-12-exp-est4}
	\end{align}
	Let $r > 0$ and choose $R = \max\{ R_{1}, r \}$, then combining \eqref{eq:Ihom-12-exp-est3}--\eqref{eq:Ihom-12-exp-est4} yields the desired estimate
	\begin{align*}
	\sum_{\ell \in \Lhom \cap B_{r}} |D \Ihom_{2} u(\ell)|_{\wf,1} &\leq C \sum_{\ell \in \L \cap B_{R}} |D u(\ell)|_{\wt \wf,1}.
	\end{align*}
	We now estimate $\| D \Ihom_{2} u \|_{\hc{l}^2(\Lhom)}$ using \eqref{eq:Ihom-12-exp-est3}--\eqref{eq:Ihom-12-exp-est4} and Cauchy--Schwarz
	\begin{align*}
	\| D \Ihom_{2} u \|_{\hc{l}^2_{\wf,1}(\Lhom)}^{2} &= \sum_{\ell \in \Lhom} |D \Ihom_{2} u(\ell)|^{2}_{\wf, 1} \\
	& = \sum_{\ell \in \Lhom \cap B_{\Rcore}} |D \Ihom_{2} u(\ell)|^{2}_{\wf, 1} + \sum_{\ell \in \Lhom \setminus B_{\Rcore}} |D \Ihom_{2} u(\ell)|^{2}_{\wf, 1} \\
	& \leq C \sum_{\ell \in \Lhom \cap B_{\Rcore}} \bigg( \sum_{\ell' \in \L \cap B_{R_{1}}} |D u(\ell')|_{\wt \wf, 1} \bigg)^{2} + C \sum_{\ell \in \L \setminus B_{\Rcore}} |D u(\ell)|^{2}_{\wt \wf, 1} \\
	& \leq C |B_{R_{1}}| \sum_{\ell \in \Lhom \cap B_{\Rcore}} \sum_{\ell' \in \L \cap B_{R_{1}}} |D u(\ell')|^{2}_{\wt \wf, 1} + C \sum_{\ell \in \L \setminus B_{\Rcore}} |D u(\ell)|^{2}_{\wt \wf, 1} \\
	& \leq C |\Lhom \cap B_{\Rcore}| \sum_{\ell' \in \L \cap B_{R_{1}}} |D u(\ell')|^{2}_{\wt \wf, 1} + C \sum_{\ell' \in \L \setminus B_{\Rcore}} |D u(\ell')|^{2}_{\wt \wf, 1} \\
	& \leq C \sum_{\ell' \in \L} |D u(\ell')|^{2}_{\wt \wf, 1} = C \| D u \|_{\hc{l}^2_{\wt \wf, 1}(\L)}^{2}.
	\end{align*}
	The desired estimate then follows by an application of Lemma \ref{lemma:normEquiv}
	\begin{displaymath}
	\| D \Ihom_{2} u \|_{\hc{l}^2_{\mathcal{N}}(\Lhom)} \leq C \| D \Ihom_{2} u \|_{\hc{l}^2_{\wf,1}(\Lhom)} \leq C \| D u \|_{\hc{l}^2_{\wt \wf,1}(\L)} \leq C \| D u \|_{\hc{l}^2_{\mathcal{N}}(\L)}.
	\end{displaymath}
\end{proof}

\begin{proof}[Proof of Lemma \ref{lemma-Iback-norm-est}]
	This holds from following the proof of Lemma \ref{lemma-Ihom-norm-est} verbatim.
\end{proof} }

\subsection{Proof of Theorem \ref{corr:deltaE0-pd-est}}
\label{sec:proof_deltaE0_pd-est}

To show the boundedness of the functional $F$, defined in the statement of
Theorem~\ref{corr:deltaE0-pd-est}, we shall use the interpolation functions introduced
in Section \ref{sec:homogeneous} to compare displacements on $\L$ to those on $\L^{\hom}$
in the following result.

\begin{lemma} \label{lemmaforceinterpest}
	Suppose that the assumptions of Theorem \ref{corr:deltaE0-pd-est} hold.
	Let \newline ${\Ihom_{1}, \Ihom_{2}: \UsH(\L) \to \UsH(\L^{\hom})}$ and ${\Iback: \UsH(\L^{\hom}) \to \UsH(\L)}$
	be the interpolation operators defined in Lemmas \ref{Lemma-Int-hom}, \ref{lemma-Ihom-norm-est} and \ref{lemma-Iback-norm-est}.
	Then, for $u \in \mathscr{H}(\L)$, define the functionals
	${T_{u}: \Usz(\L) \to \mathbb{R}}$ and $T^{\hom}_{u}: \Usz(\Lhom) \to \mathbb{R}$ by
	\begin{align}
	\b\<T_{u},v\b\> &= \sum_{\ell_{1} \in \L} \sum_{\rho_{1} \in \L - \ell_{1}} V_{\ell_{1},\rho_{1}}(Du(\ell_{1})) \cdot D_{\rho_{1}} v(\ell_{1}) \nonumber \\ & \qquad - \sum_{\ell_{2} \in \Lhom} \sum_{\rho_{2} \in \Lhom_{*}} V_{,\rho_{2}}(D\Ihom_{1}u(\ell_{2})) \cdot D_{\rho_{2}} \Ihom_{2}v(\ell_{2})
	\label{eq:Tudef},
	\\
	\b\<T^{\hom}_{u},v^{\hom}\b\> &= \sum_{\ell_{1} \in \L} \sum_{\rho_{1} \in \L - \ell_{1}} V_{\ell_{1},\rho_{1}}(Du(\ell_{1})) \cdot D_{\rho_{1}} \Iback v^{\hom}(\ell_{1}) \nonumber \\ & \qquad - \sum_{\ell_{2} \in \Lhom} \sum_{\rho_{2} \in \Lhom_{*}} V_{,\rho_{2}}(D\Ihom_{1}u(\ell_{2})) \cdot D_{\rho_{2}} v^{\hom}(\ell_{2})
	\label{eq:Thomudef}.
	\end{align}
	There exists a constant $C >0$, depending on $s$ and $u \in \mathscr{H}(\L)$, such that for all
	${v \in \Usz(\L)}$ and $v^{\hom}\in\Usz(\Lhom)$
	\begin{align}
	\label{eq:T-v}
	\left| \b\<T_{u},v\b\> \right| &\leq C \| Dv \|_{\hc{l}^2_{\mathcal{N}}(\L)}, \\
	\label{eq:Thom-vhom}
	\left| \b\<T^{\hom}_{u},v^{\hom}\b\> \right|
	&\leq C \| Dv^{\hom} \|_{\hc{l}^2_{\mathcal{N}}(\L^{\hom})}.
	\end{align}
\end{lemma}
Using Lemma \ref{lemmaforceinterpest}, we define unique continuous extensions of $T_{u}$, $T^{\hom}_{u}$ to $\UsH(\L), \UsH(\Lhom)$, respectively. 
For convenience, we choose to denote these extensions by $T_{u}$, $T^{\hom}_{u}$ and remark that they satisfy \eqref{eq:Tudef}--\eqref{eq:Thom-vhom} 
for all ${v \in \UsH(\L)}$ and $v^{\hom}\in\UsH(\Lhom)$, respectively.

\begin{proof}
	We first observe from \asRC~and Lemmas \ref{Lemma-Int-hom} and \ref{lemma-Iback-norm-est} that for each ${u \in \mathscr{H}(\L)}$ and $v^{\hom} \in \Usz(\Lhom)$, there exists $R_{0}> \Rcore > 0$ \fn{satisfying $\Ihom_{1} u(\ell) = u(\ell)$ for all $\ell \in \L \setminus B_{R_{0}}$. It follows that}
	\begin{align} \label{eq:Dv-DIv-comp}
	D_{\rho} \Ihom_{1} u(\ell) = D_{\rho} u(\ell) \quad \text{ and } \quad D_{\rho} \Iback v^{\hom}(\ell) = D_{\rho} v^{\hom}(\ell),
	\end{align}
for all \fn{$\ell \in \L\setminus B_{2{R_{0}}}$, and either
	$\rho \in \L^{\hom} \cap B_{R_{0}}$ or
	$\rho \in \left( \L^{\hom} \setminus B_{R_{0}}\right) - \ell$.}
	
	We then decompose the right-hand side of \eqref{eq:Thomudef} into four terms:
	\begin{align}
	&\quad
	\b\<T^{\hom}_{u},v^{\hom}\b\>
	\nonumber \\
	&= \sum_{\ell_{1} \in \L} \sum_{\rho_{1} \in \L - \ell_{1}} V_{\ell_{1},\rho_{1}}(Du(\ell_{1}))\cdot D_{\rho_{1}}\Iback v^{\hom}(\ell_{1}) - \sum_{\ell_{2} \in \L^{\hom}} \sum_{\rho_{2} \in \L^{\hom}_{*}} \hc{V}_{,\rho_{2}}(D \Ihom_{1} u(\ell_{2}))\cdot D_{\rho_{2}}v^{\hom}(\ell_{2})
	\nonumber \\
	&= \sum_{\substack{\ell \in \L^{\hom} \\ |\ell| \geq 2R_{0}}} \sum_{\substack{\rho \in \L^{\hom}_{*} \\ |\rho| < R_{0}}} \left( V_{\ell,\rho}(Du(\ell)) - \hc{V}_{,\rho}(D \Ihom_{1} u(\ell)) \right) \cdot D_{\rho}v^{\hom}(\ell)
	\label{eq:Vdiffp1}
	\\
	& + \sum_{\substack{\ell \in \Lhom \\ |\ell| \geq 2R_{0}}} \sum_{\substack{\rho \in (\L^{\hom} \setminus B_{R_{0}}) - \ell \\ |\rho| \geq R_{0}}} \left( V_{\ell,\rho}(Du(\ell)) - \hc{V}_{,\rho}(D \Ihom_{1} u(\ell)) \right) \cdot D_{\rho}v^{\hom}(\ell) 	\label{eq:Vdiffp2} \\
	& + \sum_{\substack{\ell \in \L^{\hom} \\ |\ell| \geq 2R_{0}}} \bigg( \sum_{\substack{\rho_{1} \in (\L \cap B_{R_{0}}) - \ell \\
			|\rho_{1}| \geq R_{0}}} V_{\ell,\rho_{1}}(Du(\ell)) \cdot D_{\rho_{1}}\Iback v^{\hom}(\ell) \bigg.
	\nonumber \\
	& \bigg. \qquad \qquad \qquad \qquad \qquad - \sum_{\substack{\rho_{2} \in (\L^{\hom} \cap B_{R_{0}}) - \ell
			\\ |\rho_{2}| \geq R_{0}}} \hc{V}_{,\rho_{2}}(D \Ihom_{1} u(\ell)) \cdot D_{\rho_{2}}v^{\hom}(\ell) \bigg) \qquad \,\,
	\label{eq:Vdiffp3} \\
	& + \sum_{\substack{\ell_{1} \in \L \\ |\ell_{1}| < 2{R_{0}}}} \sum_{\rho_{1} \in \L - \ell_{1}} V_{\ell_{1},\rho_{1}}(Du(\ell_{1}))\cdot D_{\rho_{1}} \Iback v^{\hom}(\ell_{1}) \nonumber \\
	& \qquad \qquad \qquad \qquad \qquad - \sum_{\substack{\ell_{2} \in \L^{\hom} \\ |\ell_{2}| < 2R_{0}}} \sum_{\rho_{2} \in \L^{\hom}_{*}} \hc{V}_{,\rho_{2}}(D \Ihom_{1} u(\ell_{2}))\cdot D_{\rho_{2}}v^{\hom}(\ell_{2}). \quad
	\label{eq:Vdiffp4}
	\end{align}

	Let $\wf_{\rm h}(r):=(1+r)^{-s}$, where $s$ is given in {\asSEH} then $s>d/2$ implies that $\wf_{\rm h}\in L^2(\R^d)$.
	Note that for $|\ell| \geq 2{R_{0}}$ and $|\rho| < {R_{0}}$, we have
	$|\rho|\leq |\ell| - {R_{0}}$,
	then~\eqref{eq:Vdiffp1} can be estimated using \eqref{eq:Dv-DIv-comp}, \eqref{eq:Vhomogeneity} and
	Lemma \ref{lemma:normEquiv}
	\begin{align}\label{eq:Vdiffp1est}
	&\quad
	\bigg|\sum_{\substack{\ell \in \L^{\hom} \\ |\ell| \geq 2{R_{0}}}} \sum_{\substack{\rho \in \L^{\hom}_{*} \\ |\rho| < {R_{0}}}}
	\left( V_{\ell,\rho}(Du(\ell)) - \hc{V}_{,\rho}(D \Ihom_{1} u(\ell)) \right) \cdot D_{\rho}v^{\hom}(\ell) \bigg|
	\nonumber \\
	&\leq  C\sum_{\substack{\ell \in \L^{\hom} \\ |\ell| \geq 2{R_{0}}}} \sum_{\substack{\rho \in \Lhom_{*} \\ |\rho| < {R_{0}}}} \wf_{\rm h}(|\ell| - {R_{0}}) \wf_{1}(|\rho|) |D_{\rho}v^{\hom}(\ell)|
	\leq C\sum_{\substack{\ell \in \L^{\hom} \\ |\ell| \geq 2{R_{0}}}} \wf_{\rm h}(|\ell|/2) |Dv^{\hom}(\ell)|_{\wf_{1},1}
	\nonumber \\
	&\leq C \| \wf_{\rm h}(\ell/2) \|_{\hc{l}^2(\Lhom)} \| Dv^{\hom} \|_{\hc{l}^2_{\wf_{1},1}(\Lhom)} \leq C \| \wf_{\rm h}
	\|_{L^{2}(\mathbb{R}^{d})} \| Dv^{\hom} \|_{\hc{l}^2_{\mathcal{N}}(\Lhom)}.
	\end{align}
	An identical argument gives the following estimate for \eqref{eq:Vdiffp2}
	\begin{align}\label{eq:Vdiffp2est}
	&\quad
	\bigg|\sum_{\substack{\ell \in \L^{\hom} \\ |\ell| \geq 2{R_{0}}}} \sum_{\substack{\rho \in (\L^{\hom}_{*} \setminus B_{{R_{0}}}) - \ell \\ |\rho| \geq {R_{0}}}}
	\left( V_{\ell,\rho}(Du(\ell)) - \hc{V}_{,\rho}(D \Ihom_{1} u(\ell)) \right) \cdot D_{\rho}v^{\hom}(\ell) \bigg|
	\nonumber \\
	&\leq  C\sum_{\substack{\ell \in \L^{\hom} \\ |\ell| \geq 2{R_{0}}}} \sum_{\substack{\rho \in \Lhom_{*} \\ |\rho| \geq {R_{0}}}} \wf_{\rm h}(|\ell|/2) \wf_{1}(|\rho|)
	|D_{\rho}v^{\hom}(\ell)| \leq C\sum_{\substack{\ell \in \L^{\hom} \\ |\ell| \geq 2{R_{0}}}} \wf_{\rm h}(|\ell|/2) |Dv^{\hom}(\ell)|_{\wf_{1},1} \nonumber \\
	& \leq C \| \wf_{\rm h} \|_{L^{2}(\mathbb{R}^{d})} \| Dv^{\hom} \|_{\hc{l}^2_{\wf_{1},1}(\Lhom)} \leq C \| \wf_{\rm h}
	\|_{L^{2}(\mathbb{R}^{d})} \| Dv^{\hom} \|_{\hc{l}^2_{\mathcal{N}}(\Lhom)}.
	\end{align}
	For \eqref{eq:Vdiffp3}, we first consider the following part by using \eqref{eq:Vloc}, Lemmas \ref{lemma:normEquiv} and~\ref{lemma-Iback-norm-est}, together with the Cauchy--Schwarz inequality
	\begin{align}\label{eq:Vdiffp3estp1}
	&\quad
	\bigg| \sum_{\substack{\ell \in \L^{\hom} \\ |\ell| \geq 2{R_{0}}}} \sum_{\substack{\rho \in (\L \cap B_{{R_{0}}}) - \ell \\ |\rho| \geq {R_{0}}}} V_{\ell,\rho}(Du(\ell)) \cdot D_{\rho} \Iback v^{\hom}(\ell) \bigg|
	\nonumber \\
	& =\bigg| \sum_{\ell' \in \L \cap B_{{R_{0}}}} \sum_{\substack{\rho' \in (\Lhom \setminus B_{2{R_{0}}}) - \ell' \\ |\rho'| \geq {R_{0}}}} V_{\ell',\rho'}(Du(\ell')) \cdot D_{\rho'} \Iback v^{\hom}(\ell') \bigg|
	\nonumber \\
	&\leq C\sum_{\ell' \in \L \cap B_{{R_{0}}}} \sum_{\substack{\rho' \in (\Lhom \setminus B_{2{R_{0}}}) - \ell' \\ |\rho'| \geq {R_{0}}}} \wf_{1}(|\rho|) |D_{\rho'}\Iback v^{\hom}(\ell')|
	\leq C\sum_{\ell' \in \L \cap B_{{R_{0}}}} |D\Iback v^{\hom}(\ell')|_{\wf_{1},1} \nonumber
	\\ &\leq C |\L \cap B_{{R_{0}}}|^{1/2} \| D\Iback v^{\hom} \|_{\hc{l}^2_{\wf_{1},1}(\L)} \leq C \| D \Iback v^{\hom} \|_{\hc{l}^2_{\mathcal{N}}(\L)} \leq C \| D v^{\hom} \|_{\hc{l}^2_{\mathcal{N}}(\Lhom)} ,
	\end{align}
	where we have used the substitutions $\ell' = \ell + \rho, \rho' = - \rho$ to obtain the estimate above.
	The remaining term in \eqref{eq:Vdiffp3} can be estimated by an identical argument
	\begin{align}
	\bigg| \sum_{\substack{\ell \in \L^{\hom} \\ |\ell| \geq 2{R_{0}}}} \sum_{\substack{\rho \in (\Lhom \cap B_{{R_{0}}}) - \ell \\ |\rho| \geq {R_{0}}}} \hc{V}_{,\rho}(D\Ihom_{1} u(\ell)) \cdot D_{\rho}v^{\hom}(\ell) \bigg| \leq C \| Dv^{\hom} \|_{\hc{l}^2_{\mathcal{N}}(\Lhom)}. \label{eq:Vdiffp3estp2}
	\end{align}
	For \eqref{eq:Vdiffp4}, we can estimate the following part
	\begin{align} \label{eq:Vdiffp4estp1}
	&\quad
	\bigg|\sum_{\ell \in \L \cap B_{2{R_{0}}}} \sum_{\rho \in \L- \ell} V_{\ell,\rho}(Du(\ell))\cdot D_{\rho}\Iback v^{\hom}(\ell) \bigg|
	\nonumber \\
	&\leq C\sum_{\ell \in \L \cap B_{2{R_{0}}}} \sum_{\rho \in \L - \ell} \wf_{1}(|\rho|)  |D_{\rho} \Iback v^{\hom}(\ell)|
	\leq C \sum_{\ell \in \L \cap B_{2{R_{0}}}} |D \Iback v^{\hom}(\ell)|_{\wf_{1},1}
	\nonumber \\
	&\leq C |\L \cap B_{2{R_{0}}}|^{1/2} \| D\Iback v^{\hom} \|_{\hc{l}^2_{\wf_{1},1}(\L)}
	\leq C \| D v^{\hom} \|_{\hc{l}^2_{\mathcal{N}}(\Lhom)},
	\end{align}
	and similarly for the remaining term
	\begin{eqnarray}\label{eq:Vdiffp4estp2}
	\bigg|\sum_{\ell \in \L \cap B_{2{R_{0}}}} \sum_{\rho \in \L - \ell} \hc{V}_{,\rho}(D\Ihom_{1} u(\ell))\cdot D_{\rho} v^{\hom}(\ell) \bigg|
	\leq C \| D v^{\hom} \|_{\hc{l}^2_{\mathcal{N}}(\Lhom)}.
	\end{eqnarray}
	We obtain the desired estimate \eqref{eq:Thom-vhom}
	by combining the estimates \eqref{eq:Vdiffp1est}--\eqref{eq:Vdiffp4estp2} and remark that the final constant depends on $R_{0}$ 
	and hence on $u \in \UsH(\L)$, but is independent of $v^{\hom} \in \Usz(\Lhom)$.

	The proof of \eqref{eq:T-v} follows immediately by the same arguments.
\end{proof}
\fn{ \begin{remark} \label{Remark-homogeneity-est}
	In the proof of \eqref{eq:Thom-vhom} in Lemma \ref{lemmaforceinterpest}, we obtain the following estimate by collecting \eqref{eq:Vdiffp1est}--\eqref{eq:Vdiffp4estp2} and using that $\wf_{\rm h}(r) = (1 + r)^{-s}$:
	for all $u \in \UsH(\L)$, there exist constants $C, R_{1} > 0$ such that for all $v^{\hom} \in \UsH(\Lhom)$
	\begin{align}\label{eq:hom-vhom-est}
	& \fn{\left| \b\<T^{\hom}_{u},v^{\hom}\b\> \right|}
	\nonumber \\
	& \leq C \bigg(\sum_{\ell \in \L^{\hom}} (1+|\ell|/2)^{-s}|Dv^{\hom}(\ell)|_{\wf_{1},1}
	+ \sum_{\ell \in \L^{\hom} \cap B_{2R_{0}}} |D v^{\hom}(\ell)|_{\wf_{1},1}
	+ \sum_{\ell \in \L \cap B_{2R_{0}}} |D \Iback v^{\hom}(\ell)|_{\wf_{1},1} \bigg)
	\nonumber \\
	&\leq C \bigg( \sum_{\ell \in \L^{\hom}} (1+|\ell|/2)^{-s}
	|Dv^{\hom}(\ell)|_{\wt \wf_{1},1} + \sum_{\ell \in \L^{\hom} \cap B_{2R_{1}}} |D v^{\hom}(\ell)|_{\wt \wf_{1},1} \bigg) ,
	\end{align}
	where $\wt \wf(r) := \wf(c_{0}r) \in \mathscr{H}^{1}$, for a fixed constant $c_{0} \in (0,1]$. We remark that we have applied Lemma \ref{lemma-Ihom-norm-est} to obtain the final estimate. We also note that $R_{1} = R_{1}(u) > 0$ is independent of $v^{h} \in \UsH(\Lhom)$.
	This estimate will be used in the proof of Lemma \ref{lemma-g-ubar-decay}.
\end{remark} }

\begin{proof}
	[Proof of Theorem \ref{corr:deltaE0-pd-est}]
	Recall the definition of the functional ${F: \UsH(\L) \to \mathbb{R}}$ given in \eqref{eq:T-def-ass}, then we similarly define $F_{\hom}: \UsH(\Lhom) \to \mathbb{R}$ using \eqref{eq:Thom-ass}, so for $v \in \UsH(\L)$ and $v^{\hom} \in \UsH(\Lhom)$
	\begin{align*}
	\b\<F,v\b\> &= \sum_{\ell_{1} \in \L} \sum_{\rho_{1} \in \L - \ell_{1}} V_{\ell_{1}, \rho_{1}}(0) \cdot D_{\rho_{1}}v(\ell_{1}), \\
	\b\<F^{\hom},v^{\hom}\b\> &:= \sum_{\ell_{2} \in \Lhom} \sum_{\rho_{2} \in \Lhom_{*}} \hc{V}_{, \rho_{2}}(0) \cdot D_{\rho_{2}}v^{\hom}(\ell_{2}) = 0.
	\end{align*}
	Recall that we have shown in Lemma \ref{lemma_lattice_symm} that $F^{\hom} \equiv 0$. Lemma \ref{Lemma-Int-hom} implies that for $0 \in \UsH(\L)$, we have $\Ihom_{1} 0 = 0^{\hom} \in \UsH(\Lhom)$, which together with \eqref{eq:Tudef} implies that for all $v \in \UsH(\L)$
	\begin{align*}
	\langle F, v  \rangle = \langle F, v  \rangle - \langle F^{\hom}, \Ihom_{2} v  \rangle = \langle T_{0}, v  \rangle,
	\end{align*}
	where $T_{0}$ denotes the operator $T_{u}$ defined in \eqref{eq:Tudef} with $u = 0$. The desired estimate then follows from the estimate \eqref{eq:T-v} appearing in Lemma \ref{lemmaforceinterpest}
	\begin{displaymath}
	\big|\langle F, v  \rangle \big| = \big|\langle T_{0}, v  \rangle \big|
	\leq C \| Dv \|_{\hc{l}^2_{\mathcal{N}}(\L)}.
	\end{displaymath}
\end{proof}

Now that we have proved that Lemma \ref{lemma_lattice_symm}, Theorems \ref{theorem:E-Wc} and \ref{corr:deltaE0-pd-est} all hold, the following result holds immediately.

\begin{corollary} \label{Corollary-T-deltaE-equiv}
	For each $u \in \mathscr{H}(\L)$, the functionals $T_{u}, T^{\hom}_{u}$ defined in {\eqref{eq:Tudef}--\eqref{eq:Thomudef}} can be expressed as
	\begin{align*}
	\b\<T_{u},v\b\> &= \b\<\delta \E(u),v\b\> - \b\<\delta \E^{h}(\Ihom_{1}u),\Ihom_{2}v\b\>, \nonumber \\
	\b\<T^{\hom}_{u},v^{\hom}\b\> &= \b\<\delta \E(u),\Iback v^{\hom}\b\> - \b\<\delta \E^{h}(\Ihom_{1}u),v^{\hom}\b\>.
	\end{align*}
	Moreover, there exists a constant $C > 0$, depending on $u \in \mathscr{H}(\L)$, such that for all ${v \in \UsH(\L)}$ and ${v^{\hom}\in\UsH(\Lhom)}$
	\begin{align}
	\left| \b\<\delta \E(u),v\b\> - \b\<\delta \E^{h}(\Ihom_{1}u),\Ihom_{2}v\b\> \right| &\leq C \| Dv \|_{\hc{l}^2_{\mathcal{N}}(\L)}, \label{eq:dEcomp-v} \\
	\left| \b\<\delta \E(u),\Iback v^{\hom}\b\> - \b\<\delta \E^{h}(\Ihom_{1}u),v^{\hom}\b\> \right| &\leq C \| Dv^{\hom} \|_{\hc{l}^2_{\mathcal{N}}(\L^{\hom})}.
	\label{eq:dEcomp-vhom}
	\end{align}
\end{corollary}

\subsection{Proof of Theorem \ref{theorem:point-defect-decay}}
\label{sec:proof_decay_pd}

We shall use the following result in order to prove Theorem \ref{theorem:point-defect-decay}.

\begin{lemma} \label{lemma-g-ubar-decay}
	If the conditions of Theorem \ref{theorem:point-defect-decay} are satisfied,
	then for any $\bar{u}$ solving~\eqref{eq:1st_order_pd},
	there exists $g \in \big(\UsH\big(\Lhom))^{*}$ such that
	\begin{align*}
	\b\<H \Ihom_{1} \bar{u},v^{\hom}\b\>=\b\<g,Dv^{\hom}\b\> \qquad\forall~v^{\hom}\in\UsH(\Lhom) ,
	\end{align*}
	where $\Ihom_{1}:\UsH(\L)\rightarrow\UsH(\Lhom)$ is defined in Lemma \ref{Lemma-Int-hom}
	and $g$ satisfies
	\begin{align*}
	\left| \b\<g,Dv^{\hom}\b\> \right| \leq C \sum_{\ell \in \Lhom} \sum_{j=1}^{3} \tilde{g_{j}}(\ell) |Dv^{\hom}(\ell)|_{\wf_{j},j},
	\end{align*}
	with some constant $C>0$, $\wf_j\in\Hw_j$ for $1\leq j\leq 3$ and
	\begin{align*}
	\tilde{g_{1}}(\ell) &= \tilde{g_{2}}(\ell) = (1+|\ell|)^{-s} + \bigg( \sum_{k=1}^{2} |D\Ihom_{1} \bar{u}(\ell)|_{\wf_{k},k} \bigg)^{2}, \quad
	\tilde{g_{3}}(\ell) = |D\Ihom_{1} \bar{u}(\ell)|_{\wf_{3},3} ^{2}.
	\end{align*}
\end{lemma}

\begin{proof}
	For $\bar{u}$ solving \eqref{eq:1st_order_pd} and $v^{\hom} \in \UsH(\Lhom)$,
	we denote by $\tilde{\pmb v} = (\tilde{v}_{1},\tilde{v}_{2},\tilde{v}_{3}):=(\Ihom_{1} \bar{u},\Ihom_{1} \bar{u}, v^{\hom})\in\big(\UsH(\Lhom)\big)^3$.
	Using that $\delta \E^{\hom}(0) = \delta \E(\bar{u}) = 0$,
	we can rewrite the residual $\b\<H \Ihom_{1} \bar{u},v^{\hom}\b\>$ by
	\begin{align}
	\b\<H \Ihom_{1} \bar{u},v^{\hom}\b\> &= \b\<\delta^{2} \E^{\hom}(0) \Ihom_{1} \bar{u},v^{\hom}\b\> \nonumber \\
	&= \b\<\delta \E^{\hom}(0) + \delta^{2} \E^{\hom}(0) \Ihom_{1} \bar{u} - \delta \E^{\hom}(\Ihom_{1} \bar{u}),v^{\hom}\b\> \label{eq:res1} \\
	& \qquad + \b\< \delta \E^{\hom}(\Ihom_{1} \bar{u}),v^{\hom}\b\> - \b\< \delta \E(\bar{u}), \Iback v^{\hom}\b\>. \label{eq:res2}
	\end{align}
	For \eqref{eq:res1}, by following a similar argument to \eqref{eq:Ejregest1}--\eqref{eq:useful-eq}, we deduce by Taylor's theorem that
	\begin{align}
	&\left| \b\<\delta \E^{\hom}(0) + \delta^{2} \E^{\hom}(0) \Ihom \bar{u} - \delta \E^{\hom}(\Ihom \bar{u}),v^{\hom}\b\> \right| \leq \int_{0}^{1} \left| \b\<\delta^{3} \E^{\hom}( t \Ihom \bar{u}),\tilde{\pmb v}\b\> \right| \text{d}t \nonumber \\
	& \quad \leq C \sum_{\{A_{1},\ldots,A_{k}\} \in \mathcal{P}(3)} \sum_{\ell\in\L}\sum_{\rho_1,\fn{\ldots},\rho_j}
	\prod_{m \in A_{i}} |Dv_{m}(\ell)|_{\wf_{|A_{i}|},|A_{i}|}. \label{eq:useful2}
	\end{align}
	Recall the definition \eqref{eq:partition-def}, $\mathcal{P}(3)$ is the set of partitions of $\{1,2,3\}$ as
	\begin{align}\label{eq:useful3}
	\mathcal{P}(3) = \bigg\{ \Big\{ \{1\},\{2\},\{3\} \Big\}, \Big\{\{1,2\},\{3\} \Big\}, \Big\{\{1,3\},\{2\} \Big\}, \Big\{\{1\},\{2,3\} \Big\}, \Big\{1,2,3\Big\} \bigg \},
	\end{align}
	which together with \eqref{eq:useful2} implies
	\begin{align}\label{eq:g-res-quad-est}
	&\quad
	\left| \b\<\delta \E^{\hom}(0) + \delta^{2} \E^{\hom}(0) \Ihom_{1s} \bar{u} - \delta \E^{\hom}(\Ihom_{1} \bar{u}),v^{\hom}\b\> \right|
	\nonumber \\
	&\leq C \sum_{\ell\in\L} \bigg( |D\Ihom_{1} \bar{u}(\ell)|_{\wf_{1},1}^{2} |D v^{\hom}(\ell)|_{\wf_{1},1} + |D\Ihom_{1} \bar{u}(\ell)|_{\wf_{2},2}^{2} |D v^{\hom}(\ell)|_{\wf_{1},1} \bigg.
	\nonumber \\
	& \qquad \bigg. \qquad \quad + |D\Ihom_{1} \bar{u}(\ell)|_{\wf_{1},1} |D\Ihom_{1s} \bar{u}(\ell)|_{\wf_{2},2} |D v^{\hom}(\ell)|_{\wf_{2},2} + |D\Ihom_{1} \bar{u}(\ell)|_{\wf_{3},3}^{2} |D v^{\hom}(\ell)|_{\wf_{3},3} \bigg)
	\nonumber \\
	&\leq C \sum_{\ell\in\L} \left(
	\bigg( \sum_{k=1}^{2} |D\Ihom_{1} \bar{u}(\ell)|_{\wf_{k},k} \bigg)^{2}
	\bigg( \sum_{k'=1}^{2} |D v^{\hom}(\ell)|_{\wf_{k'},k'} \bigg)
	+ |D\Ihom_{1} \bar{u}(\ell)|_{\wf_{3},3}^{2} |Dv^{\hom}(\ell)|_{\wf_{3},3}
	\right).
	\end{align}

	To estimate \eqref{eq:res2},
	we obtain from {\asSEH} and \eqref{eq:hom-vhom-est} that
	\begin{align}\label{eq:g-res-hom-est}
	 & \fn{\left| \b\<T^{\hom}_{u},v^{\hom}\b\> \right|} =
	\left| \b\< \delta \E^{\hom}(\Ihom_{1} \bar{u}),v^{\hom}\b\> - \b\< \delta \E(\bar{u}), \Iback v^{\hom}\b\> \right|
	\nonumber \\
	& \leq C \bigg( \sum_{\ell \in \L^{\hom}} (1+|\ell|/2)^{-s}
	|Dv^{\hom}(\ell)|_{\wt \wf_{1},1} + \sum_{\ell \in \L^{\hom} \cap B_{2\Rcore}} |D v^{\hom}(\ell)|_{\wt \wf_{1},1} \bigg)
	\nonumber \\
	&\leq C \sum_{\ell \in \L^{\hom}} \left( 1 + |\ell| \right)^{-s} |Dv^{\hom}(\ell)|_{\wt \wf_{1},1}
	\leq C \sum_{\ell \in \L^{\hom}} \left( 1 + |\ell| \right)^{-s} \bigg( \sum_{k' = 1}^{2} |Dv^{\hom}(\ell)|_{\wt \wf_{k'},k'} \bigg),
	\end{align}
	where we define $\wt \wf_{k}(r) := \wf_{k}(c_{0}r) \in \Hw_{k},$ where $c_{0} \in (0,1]$ is a fixed constant given in the proof of Lemma \ref{lemma-Ihom-norm-est}.
	Combining the estimates \eqref{eq:g-res-quad-est} and \eqref{eq:g-res-hom-est} and using that $\wf_{k}(r) \leq \wt \wf_{k}(r)$ for all $r > 0$, we obtain the desired result
	\begin{align*}
	\left| \b\<g,Dv\b\> \right|
	&\leq C \sum_{\ell \in \Lhom} \Bigg( (1+|\ell|)^{-s} + \bigg( \sum_{k=1}^{2} |D\Ihom_{1} \bar{u}(\ell)|_{\wt \wf_{k},k} \bigg)^{2} \Bigg) \bigg( \sum_{k'=1}^{2}  |Dv^{\hom}(\ell)|_{\wt \wf_{k'},k'} \bigg) \\
	& \qquad \qquad \quad + C \sum_{\ell\in\L} |D\Ihom_{1} \bar{u}(\ell)|_{\wt \wf_{3},3}^{2} |Dv^{\hom}(\ell)|_{\wt \wf_{3},3},
	\qquad
	\end{align*}
	which completes the proof.
\end{proof}

\begin{proof}
	[Proof of Theorem \ref{theorem:point-defect-decay}]
	By Lemma \ref{lemma-g-ubar-decay}, we have \fn{
	\begin{align*}
	\tilde{g}_{k}(\ell) \leq C \bigg( (1+|\ell|)^{-s} + \sum_{j=1}^{3} |Du(\ell)|_{\wf_j,j}^{2} \bigg),
	\end{align*}
	for all $k = 1,2,3,$ hence}
	$\Ihom_{1} \bar{u}$ satisfies the conditions of Lemma \ref{lemma:decay-hom-eq}.
	\fn{Therefore} applying Lemma \ref{lemma:decay-hom-eq} gives the desired decay estimate \eqref{eq:pd-ubar-decay}.
\end{proof}

\section{Proofs: Dislocations}
\label{sec:proof-d}
\setcounter{equation}{0}

The purpose of this section is to prove Theorems \ref{corr:deltaE0-dislocation} and \ref{theorem:dislocation-decay}.
First, we introduce the elastic strain and describe its decay properties.

\fn{Recall that in Section \ref{sec:siteE}, we introduced the homogeneous site strain potential $V^{\hom}$. 
	For the sake of convenience, we simply denote $V^{\hom}$ by $V$ in all subsequent proofs.}

\subsection{Elastic strain}
\label{sec:appendix:dislocation:general}
\fn{The following setting for dislocations can be found in \cite{2013-defects-v3}. 
	Throughout this section, given $x = (x_{1},x_{2},x_{3}) \in \mathbb{R}^{3}$, we will denote $x_{12} := (x_{1},x_{2}) \in \mathbb{R}^{2}$.
We first introduce the slip operator $S_{0}$ acting on displacements $w: \L \to \R^3$ or $w: \R^{2} \to \R^3$ by
	\begin{displaymath}
	S_{0}w(x) := \left\{
	\begin{array}{ll}
	w(x), & \quad x_2 > \hat{x}_2, \\
	w(x - \burg_{12}) - \burg, & \quad x_2 < \hat{x}_2.
	\end{array}
	\right.
	\end{displaymath}
	For $y = x + w \in \Adm(\L)$, the corresponding map $\wt y \in \Adm^{0}(\L_{0})$ defined in \eqref{eq:wt y from y def} is given by $\wt y(\ell') = \ell' + w(\ell'_{12})$. We now define $\wt y^{S} \in \Adm^{0}(\L_{0})$ by $\wt y^{S}(\ell') = \ell' + S_{0}w(\ell'_{12})$. Given $\ell' = (\ell_{1}',\ell_{2}',\ell_{3}') \in \L_{0}$, if $\ell_{2}' < \hat{x}_{2}$, then $\wt{y}^{S}(\ell') = \wt y (\ell' - \burg)$ and otherwise $\wt y^{S}(\ell') = \wt y(\ell')$, therefore $\wt y \mapsto \wt y^{S}$ represents a relabelling of the atomic indices. Consequently, there exists a bijection $\pi: \L_{0} \to \L_{0}$ such that $\wt y = \wt y^{S} \circ \pi$ and {\asSEP} implies that for all $\ell' \in \L_{0}$
	\begin{equation}\label{eq:perm-invariance-site-E-dislocations}
	\Phi_{\ell'}(\wt{y}) = \Phi_{\pi(\ell')}(\wt{y}^S).
	\end{equation}
	It also follows that $y$ and $y^{S}:= x+S_{0}w \in \Adm(\L)$ are equivalent up to a relabelling of indices in $\L$.
}

\fn{We now} recall the far-field predictor for dislocations defined in Section \ref{sec:dislocation}:
\begin{eqnarray*}
	u_0(\ell):=\ulin(\xi^{-1}(\ell)) + u^{\rm c}(\ell), \quad \text{for all } \ell \in \L.
\end{eqnarray*}
The role of $\xi$ in the definition of $u_0$ \fn{ensures that $y^{S}_{0}$ defines an equivalent configuration to $y$ that additionally satisfies}
%
$y_0^S \in C^\infty(\Omega_\Gamma)$,
where $\Omega_\Gamma = \{ x_1 > \hat{x}_1 + r + \burg_1 \}$
with a sufficiently large $r$. 

We can translate \eqref{eq:perm-invariance-site-E-dislocations} \fn{into} a statement
about $u_0$ and the corresponding site strain potential $\fn{V}$. \fn{We define the $\hc{l}^2$-orthogonal operator $S$, together with its inverse $S^* = S^{-1}$ by}
%
\label{S:op}
\begin{align*}
Su(\ell)
:= \left\{ \begin{array}{ll} u(\ell), & \ell_2 > \hat{x}_2, \\
u(\ell-\burg_{12}), & \ell_2 < \hat{x}_2
\end{array} \right.
\quad {\rm and} \quad
S^* u(\ell)
:= \left\{ \begin{array}{ll} u(\ell), & \ell_2 > \hat{x}_2, \\
u(\ell+\burg_{12}), & \ell_2 < \hat{x}_2.
\end{array}\right.
\end{align*}
\fn{Using that $S_0 (u_0 + u) = S_0 u_0 + S u$,}
the permutation invariance \eqref{eq:perm-invariance-site-E-dislocations} can
now be rewritten as an invariance of \hc{$V$} under the slip $S_0$:
\begin{eqnarray}\label{slip-invariance}
\hc{V}\big(D(u_0+u)(\ell)\big) = \hc{V}\big(e(\ell)+\widetilde{D}u(\ell)\big),
\end{eqnarray}
\fn{for all $\ell \in \L$ and $u$ satisfying $x + u_{0} + u \in \Adm(\L)$,} where
\begin{equation}\label{eq:defn_elastic_strain}
e(\ell) := (e_\rho(\ell))_{\rho\in\L-\ell} \quad \text{with} \quad
e_\rho(\ell)
:= \left\{ \begin{array}{ll}
S^* D_\rho S_0u_0(\ell), & \ell\in\Omega_{\Gamma}, \\
D_\rho u_0(\ell), & \text{otherwise,}
\end{array} \right.
\end{equation}
and
\begin{equation}\label{eq:defn_Dprime}
\widetilde{D} u(\ell) :=
(\widetilde{D}_\rho u(\ell))_{\rho\in\L-\ell} \quad \text{with} \quad
\widetilde{D}_\rho u(\ell) :=
\left\{ \begin{array}{ll}
S^* D_\rho Su(\ell), & \ell \in\Omega_{\Gamma}, \\
D_\rho u(\ell), & \text{otherwise.}
\end{array} \right.
\end{equation}

\fn{We remark that applying \eqref{slip-invariance} with $u \equiv 0$ requires that the predictor satisfies ${y_{0} \in \Adm(\L)}$, which we now verify.}

\begin{lemma}\label{Lemma - y0 Adm}
\fn{There exists ${u^{\rm c} \in C^{\infty}_{\rm c}(\R^2; \R^{3})}$ such that the dislocation predictor configuration satisfies $y_{0} = x + u_{0} \in \Adm(\L)$, where $u_{0}(\ell') := \ulin(\xi^{-1}(\ell')) + u^{\rm c}(\ell')$ for $\ell' \in \L$.}
\end{lemma}

\fn{We remark that the compactly supported displacement $u^{\rm c}$ is introduced to ensure that the corresponding deformation $\wt y_{0}$, given by \eqref{eq:wt y from y def}, satisfies $\wt y_{0}(\ell) \neq \wt y_{0}(m)$ for all $\ell \in \L_{0}, m \in \L_{0} \setminus \ell$. This is a necessary condition for $y_{0} \in \Adm(\L)$.}

We now state an intermediate result which will be used to prove Lemma \ref{Lemma - y0 Adm}. \fn{First}, for $R > 0$, define the two-dimensional disc
\begin{align*}
D_{R} := \{ (x_{1},x_{2}) \in \mathbb{R}^{2}\, | \, (x_{1}^{2} + x_{2}^{2})^{1/2} < R \, \}.
\end{align*}
Additionally, let $u_{1}:= \ulin \circ \xi^{-1}$ and define ${\wt y_{1} : \left( \mathbb{R}^{2} \setminus \Gamma \right) \times \mathbb{R} \to \mathbb{R}^{3}}$ by ${\wt y_{1}(x) = x + u_{1}(x_{12})}$. 
Also, let $\Gamma_{S} := \{ (x_1, \hat{x}_2) ~|~x_1 \leq \hat{x}_1 \}$ denotes the reflected branch cut, then define $\wt y_{1}^{S} : (\R^{2} \setminus \Gamma_{S}) \times \R \to \R^{3}$ by $\fn{\wt y_1^S}(x) = x + S_0 u_{1}(x_{12})$, where $S_0 u_{1} \in C^\infty(\Omega_\Gamma)$.

As $\L \cap \Gamma = \emptyset$, so there exists $\varepsilon > 0$ such that $\L \cap \Gamma_{\varepsilon} = \emptyset$, where
\begin{align*}
\Gamma_{\varepsilon} := \left\{ \, (x_{1}, x_{2})  \, | \, x_{1} \geq \hat{x}_{1}, |x_{2} - \hat{x}_{2}| \leq \varepsilon  \,  \right\}.
\end{align*}
Now choose $\eta \in C^{\infty}(\R,\R)$ satisfying $0 \leq \eta \leq 1$, $\eta(t) = 0$ for $t \geq \varepsilon$, $\eta(t) = 1$ for $t \leq -\varepsilon$ and $|\eta'(t)| \leq C_{\eta} \varepsilon^{-1}$ for all $t \in [-\varepsilon,\varepsilon]$. Then define $I\wt y_{1}: \R^{3} \to \mathbb{R}^{3}$ by
\begin{displaymath}
I\wt y_{1}(x) := \left\{
\begin{array}{ll}
\wt y_{1}(x), & x_{12} \not \in \Gamma_{\varepsilon}, \vspace{5pt} \\
\wt y_{1}^{S}(x + \burg \eta(x_{2} - \hat{x}_{2}) ), & x_{12} \in \Gamma_{\varepsilon}.
\end{array}
\right.
\end{displaymath}
It follows from the definition that $I\wt y_{1}(\ell) = \wt y_{1}(\ell)$ for all $\ell \in \L_{0}$.
%

\begin{lemma} \label{Lemma - y0 no holes}
	There exist $\fn{\frak{m}' \in (0,1)}, M, R_{0} > 0$ such that for all $x_{1}, x_{2} \in \fn{D_{R_{0}}^{\rm c} } \times \R$
	\begin{align} \label{eq:Iy0-est}
	\frak{m}'|x_{1} - x_{2}| \leq |I\wt y_{1}(x_{1}) - I\wt y_{1}(x_{2})| \leq M|x_{1} - x_{2}|,
	\end{align}
	hence $I\wt y_{1}: D_{R_{0}}^{\rm c}  \times \R \to \R^{3}$ is smooth, injective and is a homeomorphism onto its image. Moreover there exists open, bounded $\Omega \subset \R^{2}$ such that $I \wt y_{1}( D_{R_{0}}^{\rm c} \times \R ) = \Omega^{\rm c} \times \R$.
\end{lemma}

\begin{proof}[Proof of Lemma \ref{Lemma - y0 no holes}]
	By \cite[Lemma 3.1]{2013-defects-v3}
	there exists $C > 0$ such that
	$|\nabla u_{1}(z)| \leq C|z|^{-1}$ for all ${z \in \mathbb{R}^{2} \setminus (D_{r} \cup \Gamma )}$, where $r > 0$ has been introduced in the definition of $\Omega_{\Gamma}$. Consequently, there exists $R_{0} \geq r > 0$ such that $D_{R_{0}} \cap \Gamma \supset \Omega_{\Gamma}^{\rm c} \cap \Gamma$ and additionally satisfying ${|\nabla u_{1}(z)| \leq 1/4}$ for all $z \in \mathbb{R}^{2} \setminus (D_{R_{0}} \cup \Gamma)$ and $|\nabla S_{0} u_{1}(z)| \leq 1/4$ for all $z \in \mathbb{R}^{2} \setminus (D_{R_{0} - |\burg|} \cup \Gamma_{S})$.
	
	We now justify that $I\wt y_{1}$ is smooth over $(\R^{2} \setminus D_{R_{0}} ) \times \R$. As $\wt y_{1} \in C^{\infty}((\R^{2} \setminus \Gamma ) \times \R)$, $\wt y_{1}^{S} \in C^{\infty}(\Omega_{\Gamma} \times \R)$, it only remains to show that $I\wt y_{1}$ is smooth over $(\R^{2} \setminus D_{R_{0}} \cap \partial \Gamma_{\varepsilon}) \times \R$. Consider $x \in (\R^{2} \setminus D_{R_{0}}) \times \R$ such that $x_{1} \geq \hat{x}_{1}$ and $x_{2} = \hat{x}_{2} - \varepsilon$, then taking the derivative from below in the $x_{2}$-direction gives $\nabla \wt y_{1}(x^{-}) = \nabla \wt y_{1}(x)$. Then, using that $\wt y_{1}^{S}(x + \burg) = \wt y_{1}(x)$, we deduce that
	\begin{align*}
	\nabla I \wt y_{1}(x^{+}) = \nabla \wt y_{1}^{S}(x + \burg \eta(- \varepsilon)) (I + \burg \eta'(-\varepsilon)) = \nabla \wt y_{1}^{S}(x + \burg) = \nabla \wt y_{1}(x),
	\end{align*}
	hence $I\wt y_{1}$ is differentiable at $x$. Repeating this argument also shows that $I\wt y_{1}$ is smooth at $x$. An analogous argument shows that $I\wt y_{1}$ is smooth at $x \in (\R^{2} \setminus D_{R_{0}}) \times \R$ satisfying $x_{1} \geq \hat{x}_{1}$ and $x_{2} = \hat{x}_{2} + \varepsilon$.
	
	We now show \eqref{eq:Iy0-est}. Let $x, x' \in (\R^{2} \setminus D_{R_{0}}) \times \R$, so that $x_{12}, x'_{12} \in \R^{2} \setminus D_{R_{0}}$, then define the path \\ $P_{x_{12},x'_{12}} := \{ \, x_{12} + t(x'_{12} - x_{12}) \, | \, t \in [0,1] \, \}.$ Now suppose $P_{x_{12},x'_{12}} \cap ( D_{R_{0}} \cup \Gamma_{\varepsilon} ) = \emptyset$, then
	\begin{align}
	|u_{1}(x_{12}) - u_{1}(x'_{12})| \leq \int_{0}^{1} |\nabla u_{1}(x_{12} +t(x'_{12} - x_{12}))| |x_{12} - x'_{12}| \dt \leq \frac{1}{4} |x_{12} - x'_{12}|. \label{eq:u0 x x' case 1}
	\end{align}
	Alternatively, suppose that $P_{x_{12},x'_{12}} \cap D_{R_{0}} \neq \emptyset$, then it is straightforward to construct a piecewise smooth path ${P_{x_{12},x'_{12}}' \subset \mathbb{R}^{2} \setminus D_{R_{0}}}$ from $x_{12}$ to $x'_{12}$, by following $\partial D_{R_{0}}$ instead of intersecting $D_{R_{0}}$, satisfying $|P_{x_{12},x'_{12}}'| \leq 2|x_{12} - x'_{12}|$. We additionally suppose that $P_{x_{12},x'_{12}}' \subset \mathbb{R}^{2} \setminus ( D_{R_{0}} \cup \Gamma_{\varepsilon} )$, hence it follows that
	\begin{align}
	| u_{1}(x_{12}) - u_{1}(x'_{12})| \leq |P_{x_{12},x'_{12}}'| \sup_{z \in P_{x_{12},x'_{12}}'} |\nabla u_{1}(z)| \leq \frac{1}{2}|x_{12} - x'_{12}|. \label{eq:u0 x x' case 2}
	\end{align}
	As $x_{12}, x'_{12} \not \in \Gamma_{\varepsilon}$, we have $I \wt y_{1}(x) = \wt y_{1}(x), I \wt y_{1}(x') = \wt y_{1}(x')$, hence \eqref{eq:u0 x x' case 1}--\eqref{eq:u0 x x' case 2} imply
	\begin{align} \label{eq:Iwty0-nocollest-1}
	|I\wt y_{1}(x) - I\wt y_{1}(x')| = |\wt y_{1}(x) - \wt y_{1}(x')| \geq |x - x'| - |u_{1}(x_{12}) - u_{1}(x'_{12})| \geq \frac{1}{2} |x - x'|.
	\end{align}
	Similarly, we also obtain
	\begin{align} \label{eq:Iwty0-bound-1}
	|I\wt y_{1}(x) - I\wt y_{1}(x')| \leq |x - x'| - |u_{1}(x_{12}) - u_{1}(x'_{12})| \leq \frac{3}{2} |x - x'|.
	\end{align}
	Now suppose that $P_{x_{12},x'_{12}} \cap \Gamma_{\varepsilon} \neq \emptyset$, then define $\wt x = x + \burg \eta(x_{2} - \hat{x}_{2})$, ${\wt x' = x' + \burg \eta(x'_{2} - \hat{x}_{2})}$, then there exists a path $P_{\wt x_{12}, \wt x'_{12}}$ joining $\wt x_{12}, \wt x'_{12}$ such that $P_{\wt x_{12}, \wt x'_{12}} \cap D_{R_{0}-|\burg|} = \emptyset$, \\ ${P_{\wt x_{12}, \wt x'_{12}} \cap \Gamma_{\varepsilon} \neq \emptyset}$ and $|P_{\wt x_{12}, \wt x'_{12}}| \leq 2|\wt x_{12}- \wt x'_{12}|$. Following the  argument \eqref{eq:u0 x x' case 2}--\eqref{eq:Iwty0-nocollest-1} gives
	\begin{align} \label{eq:u0 x x' case 3}
	|I \wt y_{1}(x) - I \wt y_{1}(x')| &= |\wt y_{1}^{S}(x + \burg \eta(x_{2} - \hat{x}_{2})) - \wt y_{1}^{S}(x' + \burg \eta(x'_{2} - \hat{x}_{2}))| = |\wt y_{1}^{S}(\wt x) - \wt y_{1}^{S}(\wt x')| \nonumber \\
	&\geq |\wt x - \wt x'| - |S_{0} u_{1}(\wt x_{12}) - S_{0} u_{1}(\wt x'_{12})| \geq \frac{1}{2}|\wt x - \wt x'|.
	\end{align}
	By the Mean Value Theorem, there exists $\theta = \theta(x_{2},x'_{2}) \in \R$ such that
	\begin{align*}
	\wt x - \wt x' = x - x' + \burg ( \eta(x_{2} - \hat{x}_{2}) - \eta(x'_{2} - \hat{x}_{2}) ) =  x - x' + \burg \eta'(\theta) (x_{2} - x'_{2}) =: A_{\theta}(x - x'),
	\end{align*}
	where the matrix $A_{\theta} \in \R^{3 \times 3}$ is invertible with inverse $A_{\theta}^{-1}$,
	\begin{align*}
	A_{\theta} =
	\begin{pmatrix}
	1 & \burg_{1} \eta'(\theta) & 0 \\
	0 & 1 & 0 \\
	0  & \burg_{3} \eta'(\theta) & 1 \\
	\end{pmatrix}, \quad
	A_{\theta}^{-1} =
	\begin{pmatrix}
	1 & -\burg_{1} \eta'(\theta) & 0 \\
	0 & 1 & 0 \\
	0  & -\burg_{3} \eta'(\theta) & 1 \\
	\end{pmatrix}.
	\end{align*}
	Consequently,
	\begin{align} \label{eq:A-est}
	|x - x'| = |A_{\theta}^{-1}(\wt x - \wt x')| \leq |A_{\theta}^{-1}||\wt x - \wt x'| \leq C_{0}|\wt x - \wt x'|,
	\end{align}
	where the constant $C_{0} \geq 1$ is independent of $\theta,x,x'$ and depends only on $C_{\eta}$. 
	We also obtain $|\wt x - \wt x'| \leq C_{0}|x - x'|$. Combining \eqref{eq:u0 x x' case 3}--\eqref{eq:A-est} yields
	\begin{align} \label{eq:Iwty0-nocollest-2}
	|I \wt y_{1}(x) - I \wt y_{1}(x')| &\geq \frac{1}{2}|\wt x - \wt x'| \geq \frac{1}{2C_{0}}| x - x'|,
	\end{align}
	and also
	\begin{align} \label{eq:Iwty0-bound-2}
	|I \wt y_{1}(x) - I \wt y_{1}(x')| &\leq \frac{3}{2}|\wt x - \wt x'| \leq \frac{3C_{0}}{2}| x - x'|,
	\end{align}
	Collecting the estimates \eqref{eq:Iwty0-nocollest-1}--\eqref{eq:Iwty0-bound-2}, we obtain the desired estimate \eqref{eq:Iy0-est}
	\begin{align*}
	\frak{m}'|x - x'| \leq |I \wt y_{1}(x) - I \wt y_{1}(x')| \leq M|x - x'|,
	\end{align*}
	for all ${x, x' \in (\R^{2} \setminus D_{R_{0}}) \times \R}$, where $\frak{m}' = (2C_{0})^{-1}\in(0,1),~M = \smfrac{3C_{0}}{2} > 0$. 
	It follows immediately from~\eqref{eq:Iy0-est} that $I \wt y_{1}: (\R^{2} \setminus D_{R_{0}}) \times \R \to \R^{3}$ is both injective and a closed map, thus is a homeomorphism onto its image.
	
	It follows from the definition of $I\wt y_{1}$ that $x = (x_{1},x_{2},x_{3}) \in I \wt y_{1}((\R^{2} \setminus D_{R_{0}}) \times \R)$ if and only if $(x_{1},x_{2},x_{3}') \in I \wt y_{1}((\R^{2} \setminus D_{R_{0}}) \times \R)$ for all $x_{3}' \in \R$, hence $I \wt y_{1}((\R^{2} \setminus D_{R_{0}}) \times \R) = A \times \R$ for some $A \subset \R^{2}$. Define $Iy_{1}: \R^{2} \setminus D_{R_{0}} \to \R^{2}$ by $Iy_{1}(z_{1},z_{2}) = (I\wt y_{1,1},I\wt y_{1,2})(z_{1},z_{2},0)$, so $A = Iy_{1}(\R^{2} \setminus D_{R_{0}})$. The arguments \eqref{eq:u0 x x' case 1}--\eqref{eq:Iwty0-nocollest-2} can be applied verbatim to also show: there exists $\frak{m}_{12} > 0$ such that for all $z,z' \in \R^{2} \setminus D_{R_{0}}$
	\begin{align} \label{eq:Iy0-alt-est}
	|I y_{1}(z) - I y_{1}(z')| \geq \frak{m}_{12}|z - z'|,
	\end{align}
	so we also deduce that $I y_{1}$ is a homeomorphism onto its image. In particular, this implies $Iy_{1}(\partial D_{R_{0}}) = \partial Iy_{1}( \R^{2} \setminus D_{R_{0}} )$, hence $Iy_{1}(\partial D_{R_{0}})$ is a closed loop that encloses an open, bounded region $\Omega \subset \R^{2}$. The estimate \eqref{eq:Iy0-alt-est} together with the fact that $Iy_{1}$ is a homeomorphism onto its image implies $Iy_{1}(\R^{2} \setminus D_{R_{0}}) = \R^{2} \setminus \Omega$, which concludes the proof.
\end{proof}

\fn{Using Lemma~\ref{Lemma - y0 no holes}, we now prove Lemma~\ref{Lemma - y0 Adm}. As part of our argument, we will show that the assumption of Lemma~\ref{lemma-HcHdense} is satisfied in the dislocation setting. 
	
	\begin{proof}
		[Proof of Lemma \ref{Lemma - y0 Adm}]
		Our proof closely follows the arguments \eqref{eq:Iy0-ass-rep}--\eqref{eq:u-v-open-calc-dis-5} used in the proof of Lemma~\ref{lemma-HcHdense}. We begin by applying Lemma~\ref{Lemma - y0 no holes}, so there exist $\frak{m'} \in (0,1), R_{0} > 0$ such that for all ${\ell, m \in D_{R_{0}}^{\rm c} \times \R}$
		\begin{align} \label{eq:nh-est-1}
		|\wt y_{1}(\ell) - \wt y_{1}(m)| = |I\wt y_{1}(\ell) - I\wt y_{1}(m)| \geq \frak{m}' |\ell - m|.
		\end{align}
		We recall the argument \eqref{eq:y0-3m4-est} in order to treat the case ${\ell \in D_{R_{0} + R_{1}}^{\rm c} \times \R}$, $m \in D_{R_{0}} \times \R$, where $R_{1} > 0$ will be specified shortly. Let $C_{1} := \min \{ \, |n_{12}| \, | \, n \in \L_{0} \cap ( D_{R_{0}}^{\rm c} \times \R ) \} > R_{0}$, then we can find a $\ell' \in \L_{0}$ a minimiser satisfying $|(\wt y_{1}(\ell') - \wt y_{1}(m))_{3}| \leq 1$ and $|(\ell'_{k} - m)_{3}| \leq C_{2}$, where the constant $C_{2}$ is independent of $m$. We define $C_{3} := \max \{ \, |\wt y_{1}(n)_{12}| + 1 \, | \, n \in \L_{0} \cap ( \bar{D}_{C_{1}} \times \R) \}$ and ${R_{1} := \max\{ C_{1} - R_{0} + \smfrac{8(2C_{3} +1)}{\frak{m}'}, 2(6R_{0} + 7C_{1}), 14C_{2} \}}$, then deduce that
		\begin{align}
		|\wt y_{1}(\ell) - \wt y_{1}(m)| &\geq |\wt y_{1}(\ell) - \wt y_{1}(\ell'_{k})| - | \wt y_{1}(\ell'_{k}) - \wt y_{1}(m) | \nonumber \\ &\geq \frak{m'}|\ell - \ell'_{k}| - (2C_{3} + 1) \geq \frac{7\frak{m'}}{8} |\ell - \ell'_{k}| \geq \frac{3\frak{m}'}{4} |\ell - m|. \label{eq:nh-est-2}
		\end{align}
		Next, we follow \eqref{eq:y0-3m4-est-2} by considering the case ${\ell \in D_{R_{0} + R_{1}} \times \R}$, $m \in D_{R_{0}} \times \R$ and defining $y_{1}: \L \to \R^{3}$ by \eqref{eq:y def}, ${C_{4} := \max \{ \, |y_{1}(n_{12})| + 1 \, | \, n_{12} \in \L \cap  D_{R_{0} + R_{1}} \}}$ and ${R_{2} := \max\{16C_{4},2R_{0} + R_{1}\}}$. Then,
		\begin{align} \label{eq:nh-est-3}
		|\wt y_{1}(\ell) - \wt y_{1}(m)| &= |y_{1}(\ell_{12}) - y_{1}(m_{12}) + (\ell - m)_{3}| \nonumber \\ &\geq |(\ell - m)_{3}| - 2C_{4}  \nonumber \geq \frac{7\frak{m}'}{8}|(\ell - m)_{3}| \\ &\geq \frac{3\frak{m}'}{4}(|(\ell - m)_{12}| +  |(\ell - m)_{3}|) \geq \frac{3\frak{m}'}{4}|\ell - m|.
		\end{align}
		The remaining case to consider is ${\ell \in D_{R_{0} + R_{1}} \times \R}$, $m \in D_{R_{0}} \times \R$ satisfying ${|(\ell - m)_{3}| \leq R_{2}}$. Due to the periodicity of $\wt y_{1}$ in the $e_{3}\text{-direction}$, it suffices to consider $m \in D_{R_{0}} \times [0,1)$ and ${\ell \in D_{R_{0} + R_{1}} \times [-R_{2},R_{2}+1)}$. Define 
		\begin{align*}
		c_{1} := \min \{ \, |\wt y_{1}(\ell') - \wt y_{1}(m')| \, | \, {m' \in D_{R_{0}} \times [0,1), \ell' \in D_{R_{0} + R_{1}} \times [-R_{2},R_{2}+1)} \},
		\end{align*}
		then if $c_{1} > 0$, then using that ${|\ell - m| \leq 2R_{0} + R_{1} + R_{2}}$, we deduce
		\begin{align} \label{eq:nh-est-4.1}
		|\wt y_{1}(\ell) - \wt y_{1}(m)| \geq c_{1} \geq \frac{c_{1}|\ell - m|}{2R_{0} + R_{1} + R_{2}} =: \frak{m}_{1} |\ell - m|.
		\end{align}
		In this case, we may simply define $u^{\rm c} \in C^{\infty}_{\rm c}(\R^{2};\R^{3})$ by $u^{\rm c} \equiv 0$, in which case we have $\wt y_{0} \equiv \wt y_{1}$. Collecting the estimates \eqref{eq:nh-est-1}--\eqref{eq:nh-est-4.1} gives
		\begin{align} \label{eq:y0-final-m-est}
		|\wt y_{0}(\ell) - \wt y_{0}(m)|  \geq \frak{m}|\ell - m| \quad \text{for all } \ell, m \in \L_{0},
		\end{align}
		with $\frak{m} := \min \{ \smfrac{3\frak{m}'}{4}, \frak{m}_{1} \} > 0$. 
		
		Otherwise, if $c_{1} = 0$, we add a small perturbation to $\wt y_{1}$ within $D_{R_{0}+1}$ to guarantee no collisions occur. For each $m'_{12} \in \L \cap D_{R_{0}}$, choose $a_{m'_{12}} \in D_{1}$ such that 
		\begin{align*}
		c_{2} := \min \big\{ \, |\wt y_{1}(\ell') - &\wt y_{1}(m') - a_{m'_{12}}| \, \\
		&\big| \, m' \in \L_{0} \cap (D_{R_{0}} \times \R), \ell' \in \L_{0} \cap (D_{R_{0} + R_{1}} \times \R) \big\} > 0.
		\end{align*}
		%
		%
		For $\delta > 0$, choose $\wt \eta_{\delta} \in C^{\infty}_{\rm c}(D_{\delta}; \R^{2})$ satisfying $0 \leq \wt \eta_{\delta} \leq 1$ and $ \wt \eta_{\delta}(0) = 1$, then define $\eta_{\delta} \in C^{\infty}_{\rm c}(D_{\delta};\R^{3})$ by $\eta_{\delta}(z) := (\wt \eta_{\delta}(z),0)$. We then construct $u^{\rm c} \in C^{\infty}_{\rm c}(\R^{2};\R^{3})$, given by $u^{\rm c}(z) := \sum_{m'_{12} \in \L \cap D_{R_{0}} } a_{m'_{12}} \eta_{\delta}(z - m'_{12}) $, where $\delta \in (0,1]$ is chosen to be sufficiently small to ensure that 
		the support of $u^{\rm c}$ is contained in $D_{R_{0}}$.
		
		Now define $u_{0}(z) := u_{1}(z) + u^{\rm c}(z)$, $y_{0}(z) := y_{1}(z) + u^{\rm c}(z)$ for $z \in \R^{2}$ and also ${\wt y_{0}(x) := \wt y_{1}(x) + u^{\rm c}(x_{12}) }$ for $x \in \R^{3}$. From the construction of $u^{\rm c}$, we obtain the following: for all $m \in D_{R_{0}} \times [0,1), \ell \in D_{R_{0}+R_{1}} \times [-R_{2},R_{2} + 1)$,
		\begin{align} \label{eq:nh-est-4.2}
		\!\!\!\! |\wt y_{0}(\ell) - \wt y_{0}(m)| \geq |y_{0}(\ell_{12}) - y_{0}(m_{12})| \geq \frac{c_{2}|\ell - m|}{2(R_{0}+1) + R_{1} + R_{2}} =: \frak{m}_{2} |\ell - m|.
		\end{align}
		We remark that as $u^{\rm c}$ has support in $D_{R_{0}}$ and that $|u^{\rm c}(z)| \leq 1$ for all $z \in D_{R_{0}}$, we observe that the estimates \eqref{eq:nh-est-1}--\eqref{eq:nh-est-3} continue to hold after substituting $\wt y_{1}$ with $\wt y_{0}$. Combining these estimates with \eqref{eq:nh-est-4.2} implies \eqref{eq:y0-final-m-est} with ${\frak{m} := \min \{ \smfrac{3\frak{m}'}{4}, \frak{m}_{2} \} > 0}$.
		
		We now define $I \wt y_{0}(x) := I \wt y_{1}(x)$ for $x \in D_{R_{0}}^{\rm c} \times \R$ and observe that as $u^{\rm c}$ has support in $D_{R_{0}}$, then $I \wt y_{0}$ is an interpolation of $\wt y_{0}$ satisfying $I \wt y_{0}(\ell) = \wt y_{0}(\ell)$ for all $\ell \in \L_{0} \cap (D_{R_{0}}^{\rm c} \times \R)$. So Lemma~\ref{Lemma - y0 no holes} ensures that $I\wt y_{0}$ satisfies the assumption \eqref{eq:H-Lambda-equiv-def}, hence Lemma~\ref{lemma-HcHdense} holds for dislocations. It follows from \eqref{eq:y0-final-m-est} and Lemma~\ref{lemma-HcHdense} that $y_{0} \in \Adm(\L)$, completing the proof.
\end{proof} }

We shall give the following far-field decay estimate of the dislocation predictor on the left half-plane.

\begin{lemma}\label{lemma-decay-u0}
	If the setting for dislocations hold and $\ell\in\L$ with $\ell_1<\hat{x}_1$,
	then there exist constants $C_j$ for $j\in\N$,
	such that for ${\bm \rho}=(\rho_1,\fn{\ldots},\rho_j)\in\b(\Lhom_{*}\b)^j$,
	\begin{eqnarray}\label{decay-u0}
	|D_{\bm \rho} u_0(\ell)| \leq C_j |\ell|^{-j} \max_{1\leq i\leq j} \left \{|\rho_i|^{j} \log|\rho_{i}| \right \}.
	\end{eqnarray}
\end{lemma}

\begin{remark}
	It may be possible to sharpen the decay estimates \eqref{decay-u0}. If so, then the decay assumptions {\asSEL} for  dislocations may be weakened. However, consider the case that $|\rho_2| = O(1)$, $m \in \Lhom$ fixed, $\ell = m - \rho_1$ with $|\ell| \to \infty$. In this case $|D_{\bm \rho} u_0(\ell)|
	=\b|u(m+\rho_2)-u(m)-u(\ell+\rho_2)+u(\ell)\b|
	\approx |m|^{-1} \leq |\ell|^{-2} |\rho_{1}|^2$,
	but the factor $|\rho_{1}|^2$ cannot be improved upon in an obvious way.
\end{remark}

\begin{proof}
	Let $j \in \mathbb{N}$, ${\bm \rho}=(\rho_1,\fn{\ldots},\rho_j)\in\b(\Lhom_{*}\b)^j$ and we consider $\ell \in \L$ sufficiently large.

	For the case $\max_{1 \leq i \leq j}|\rho_{i}| \leq \smfrac{|\ell|}{2j}$, using \eqref{ac_def_Lemma3.1}
	and Taylor's theorem, we deduce that
	\begin{align*}
	|D_{\bm \rho} u_0(\ell)| &\leq \prod_{i = 1}^{j}|\rho_{i}| \int_{0}^{1} \cdots \int_{0}^{1} \bigg|\nabla^{j} u_{0} \bigg(\ell + \sum_{i = 1}^{j} t_{i} \rho_{i} \bigg) \bigg| \, \dt_{1} \cdots \dt_{j} \\
	&\leq C \prod_{i = 1}^{j}|\rho_{i}| \int_{0}^{1} \cdots \int_{0}^{1} |\ell + \sum_{i = 1}^{j} t_{i} \rho_{i} |^{-j} \, \dt_{1} \cdots \dt_{j} \\
	&\leq C 2^{j} |\ell|^{-j} \prod_{i = 1}^{j}|\rho_{i}| \leq C_j |\ell|^{-j} \max_{1\leq i\leq j} \left \{|\rho_i|^{j} \log|\rho_{i}| \right \}.
	\end{align*}
	Now suppose that $\max_{1 \leq i \leq j}|\rho_{i}| > \smfrac{|\ell|}{2j}$, then we expand the finite-difference stencil to obtain
	\begin{align*}
	D_{\bm \rho} u_0(\ell) = \sum_{k = 0}^{j} (-1)^{j - k} \!\!\!\! \sum_{1 \leq i_{1} < \ldots < i_{k} \leq j}  u_{0}\bigg( \ell + \sum_{n = 1}^{k} \rho_{i_{n}} \bigg).
	\end{align*}
	Then, as \eqref{eq_ulin} implies $|u_{0}(x)| \leq C \log(2 + |x|)$, for all $x \in \R^{2} \setminus \Gamma$, it follows that
	\begin{align*}
	& |D_{\bm \rho} u_0(\ell)| \leq C \sum_{k = 0}^{j} \sum_{1 \leq i_{1} < \ldots < i_{k} \leq j}  \log \bigg( 2 + \Big| \ell + \sum_{n = 1}^{k} \rho_{i_{n}} \Big| \bigg) \\
	&\leq C 2^{j}  \log \bigg( 2 + |\ell| + \sum_{i = 1}^{j} |\rho_{i}|  \bigg) \leq \wt C_{j} \log \bigg( 1 + \max_{1 \leq i \leq j} |\rho_{i}| \bigg) \\
	& \leq \wt C_{j} (2j)^{j} \bigg( \frac{ \max_{1 \leq i \leq j} |\rho_{i}| }{ |\ell| } \bigg)^{j} \log \bigg( 1 + \max_{1 \leq i \leq j} |\rho_{i}| \bigg)
	\leq C_{j} |\ell|^{-j} \max_{1 \leq i \leq j} \left\{ |\rho_{i}|^{j} \log ( 1 + |\rho_{i}| ) \right\}.
	\end{align*}
	This completes the proof.
\end{proof}

\begin{remark}
	For finite-difference stencils with finite interaction range, Lemma~\ref{lemma-decay-u0}
	implies that  if $\ell\in\L$ with $\ell_2<\hat{x}_2$
	$|\rho_i|\leq \Rc~\forall~1\leq i\leq j$ with some
	cut-off radius $\Rc$, then $|D_{\bm \rho} u_0(\ell)| \leq C_j |\ell|^{-j}$.

	Lemma \ref{lemma-decay-u0} together with definition \eqref{eq:defn_elastic_strain}
	and the fact that $S_0u_0\in C^{\infty}(\Omega_{\Gamma})$
	give us the decay of the elastic strain:
	for $\ell\in\L$ and $|\ell|$ sufficiently large, there exist constants $C_j>0$ for $j\in\N_{0}$,
	such that
	\begin{align}
	|e_{\sigma}(\ell)| &\leq C_1 |\ell|^{-1} |\sigma| \log( 1 + |\sigma| ) \quad \text{and } \nonumber \\
	|D_{\bm \rho} e_{\sigma}(\ell)| &\leq C_{j+1} |\ell|^{-j-1} \max \Big\{ |\sigma|^{j} \log(1 + |\sigma|),
	\displaystyle{\max_{1\leq i\leq j}} |\rho_i|^{j} \log(1 + |\rho_{i}|) \Big\}, \quad \label{decay_elastic_strain}
	\end{align}
	with $\sigma\in\Lhom_{*}$ and ${\bm \rho}\in\b(\Lhom_{*}\b)^j$, when $j \geq 1$. The estimate \eqref{decay_elastic_strain} holds by the same argument used to prove Lemma \ref{lemma-decay-u0}.
\end{remark}

\subsection{Proof of Theorem \ref{corr:deltaE0-dislocation}}
\label{sec:proof_deltaE0_dislocation}

\hc{For the sake of convenience, we use $\wf_k$ to denote $\wf_k^{\log}\in\Hw_k^{\log}$ in the following subsections.} Recall the operator ${F: \Usz(\L) \to \mathbb{R}}$, defined in Theorem~\ref{corr:deltaE0-dislocation} by
\fn{ \begin{align} \label{eq:disl:delE0-Du}
\b\<F,u\b\> = \sum_{\ell \in \Lhom} \big\<\delta \fn{V}(Du_{0}(\ell)),Du(\ell)\big\>
= \sum_{\ell \in \Lhom} \b\< \del V\b(e(\ell)\b), \Del u(\ell) \b\>.
\end{align} }
To show that $F$ is defined, it is convenient to rewrite \fn{it in}
in force-displacement form \fn{as}
\begin{eqnarray}\label{eq:disl:delE0-f}
\<F, u \> =  \sum_{\ell \in \fn{\Lhom}} f(\ell) \cdot u(\ell),
\end{eqnarray}
where $f:{\Lhom}\to\R^{\ds}$ describes the residual forces.
The following lemma \fn{describes} the far-field decay of the residual force $f(\ell)$ \fn{and a} detailed expression of $f$ \fn{is given in its proof}.
\begin{lemma}\label{lemma:decay_fl_dislocation}
	Under the conditions of Theorem \ref{corr:deltaE0-dislocation},
	there exists $C>0$ such that
	\begin{eqnarray}\label{decay-fl}
	|f(\ell)|\leq C|\ell|^{-3} .
	\end{eqnarray}
\end{lemma}

\begin{proof}
	{\it Case 1: left half-plane: }
	We first consider the simplified situation when $\ell_1 < \hat{x}_1$,
	and will see below that a generalisation to $\ell_1 > \hat{x}_1$ is straightforward.

	Let us assume $\ell_2>\hat{x}_2$ first.
	Using \eqref{eq:disl:delE0-Du}, \eqref{eq:disl:delE0-f} and the definition
	of $\widetilde{D}$, we obtain from a direct calculation that
	\begin{align}\label{proof-decay-fl-6}
	\nonumber
	f(\ell) &= \sum_{\rho\in\Lhom_{*}}D_{-\rho}V_{,\rho}\b(e(\ell)\b)
	+ \bigg( \sum_{\substack{\rho\in\Lhom_{*} \nonumber \\ \ell-\rho - \burg_{12} \in\Omega_{\Gamma} \\ \ell_2-\rho_2<\hat{x}_2}}
	V_{,\rho}\b(e(\ell-\rho - \burg_{12})\b) - \sum_{\substack{\rho\in\Lhom_{*} \nonumber \\ \ell-\rho\in\Omega_{\Gamma} \\ \ell_2-\rho_2<\hat{x}_2}} V_{,\rho}\b(e(\ell-\rho)\b) \bigg)
	\\
	&=\tilde{f}(\ell)+ f^{\rm c}(\ell).
	\end{align}
	Note that $\ell_1<\hat{x}_1,~\ell_2>\hat{x}_2$ and
	$\ell-\rho\in\Omega_{\Gamma},~\ell_2-\rho_2<\hat{x}_2$ implies $|\rho|\geq |\ell|$.
	Then by \asSEL~there exists $\wf_{3} \in \Hw_{3}^{\log}$ such that
	\begin{align}\label{proof-decay-fl-6-}
	|f^{\rm c}(\ell)| \leq C\sum_{\substack{\rho\in\Lhom_{*} \\ |\rho| \geq |\ell|}}	\wf_{3}(|\rho|)
	\leq C\sum_{\substack{\rho\in\Lhom_{*} \\ |\rho| \geq |\ell|}} \wf_{3}(|\rho|) \left( \frac{|\rho|}{|\ell|} \right)^{3} \leq C |\ell|^{-3} \| \wf_{3} \|_{\Hw_{3}^{\log}}
	\leq C|\ell|^{-3},
	\end{align}
	\fn{where we have applied Lemma \ref{lemma:omega-sum-int-est}.}	We now estimate $|\tilde{f}(\ell)|$ by first expanding $V_{,\rho}$ to second order,
	\begin{equation}
	\label{eq:appprf:disl_res_expVrho}
	V_{,\rho}(e(\ell)) = V_{,\rho}(\pmb{0}) + \< \delta V_{,\rho}(\pmb{0}), e(\ell) \> + \int_0^1 (1-t)
	\< \delta^2 V_{,\rho}(te(\ell)) e(\ell), e(\ell) \> \dt.
	\end{equation}
	The symmetry assumption \asSEPS~implies that
	$\sum_\rho V_{,\rho}(\bfO) = 0$. Hence, we obtain
	\begin{eqnarray}\label{eq:appprf:disl:res1-def_f}
	\nonumber
	\tilde{f}(\ell) &=& \sum_{\rho,\sigma \in \Lhom_{*}} V_{,\rho\sigma}(\bfO) \cdot \widetilde{D}_{-\rho} e_{\sigma}(\ell)
	+ \sum_{\rho \in \Lhom_{*}} \int_0^1 (1-t) \widetilde{D}_{-\rho}
	\< \ddel V_{,\rho}\b(te(\ell)\b) e(\ell), e(\ell) \> \dt
	\\
	&=:& f^{(1)}(\ell) + f^{(2)}(\ell).
	\end{eqnarray}

	To estimate $f^{(1)}(\ell)$, we observe that $\widetilde{D}_{-\rho} e_{\sigma}(\ell) = D_{-\rho} D_{\sigma} u_0(\ell)$ as $\ell_1<\hat{x}_1$, then using as $u_0(x)$ is smooth for $x<\hat{x}_1$, for $|\rho|, |\sigma| \leq |\ell|/4$,
	applying \eqref{ac_def_Lemma3.1}
	and a Taylor expansion gives
	\begin{align}
	& \b|D_{-\rho} D_{\sigma} u_0(\ell) +  \D_\rho \D_\sigma u_0(\ell)\b| = 	\b| D_{\sigma} u_0(\ell- \rho) - D_{\sigma} u_0(\ell) +  \D_\rho \D_\sigma u_0(\ell)\b| \nonumber
	\\
	& \!\!\! = \bigg| \nabla_{\sigma} u_0(\ell-\rho)-\nabla_{\sigma} u_0(\ell)
	+ \frac{1}{2} \int_{0}^{1} \nabla_\sigma^{2} (  u_0(\ell-\rho-t_{1}\sigma)
	- u_0(\ell-t_{1}\sigma) ) \dt + \nabla_\rho \nabla_\sigma u_0(\ell)  \bigg| \nonumber
	\\
	& \!\!\! \leq \b| \nabla_{\sigma} u_0(\ell-\rho)-\nabla_{\sigma} u_0(\ell) + \nabla_\rho \nabla_\sigma u_0(\ell) \b| + \frac{1}{2} \int_{0}^{1} \b| \nabla_\sigma^{2}  u_0(\ell-\rho-t_{1}\sigma) - u_0(\ell-t_{1}\sigma) \b| \dt_{1} \nonumber \\
	& \!\!\! \leq \int_{0}^{1} |\nabla_{\rho}^{2} \nabla_{\sigma} u_{0}(\ell - t\rho)| \dt + \frac{1}{2} \int_{0}^{1} \int_{0}^{1} \b| \nabla_{\rho} \nabla_\sigma^{2}  u_0(\ell-t_{1}\sigma - t_{2}\rho) \b| \dt_{1} \dt_{2} \nonumber \\
	& \!\!\! \leq C |\ell|^{-3} \b( |\rho|^{2} |\sigma| + |\rho| |\sigma|^{2} \b) \leq C |\ell|^{-3} \max \b\{ |\rho|^{3} \log(1 + |\rho|), |\sigma|^{3} \log(1 + |\sigma|) \b\}. \hspace{-5pt} \label{eq:D-rho-sigma-u0-est-1}
	\end{align}

	For the remaining case $ \max\{ |\rho|, |\sigma| \} > |\ell|/4$, we apply the triangle inequality together with Lemma \ref{lemma-decay-u0} and \eqref{ac_def_Lemma3.1}
	to obtain
	\begin{align}
	&\b|D_{-\rho} D_{\sigma} u_0(\ell) +  \D_\rho \D_\sigma u_0(\ell)\b| \leq \b|D_{-\rho} D_{\sigma} u_0(\ell) \b| + \b| \D_\rho \D_\sigma u_0(\ell)\b| \nonumber \\
	& \quad \leq C |\ell|^{-2} \max\{ |\rho|^{2} \log(1 + |\rho|), |\sigma|^{2} \log(1 + |\sigma|) \} \nonumber \\
	& \quad \leq C |\ell|^{-3} \max\{ |\rho|^{3} \log(1 + |\rho|), |\sigma|^{3} \log(1 + |\sigma|) \}. \label{eq:D-rho-sigma-u0-est-2}
	\end{align}
	Combining \eqref{eq:D-rho-sigma-u0-est-1}--\eqref{eq:D-rho-sigma-u0-est-2}	and applying \asSEL~, there exist $\wf_{3} \in \Hw^{\log}_{3}, \wf_{4} \in \Hw^{\log}_{4}$ such that
	\begin{align}\label{proof-decay-fl-1}
	&\bigg| \sum_{\rho,\sigma \in \Lhom_{*}} V_{,\rho\sigma}(\bfO)
	\b( D_{-\rho} e_{\sigma}(\ell) + \D_\rho \D_\sigma u_0(\ell) \b) \bigg| \nonumber \\
	& \quad \leq C|\ell|^{-3} \sum_{\rho,\sigma \in \Lhom_{*}} \wf_{3}(|\rho|) \wf_{3}(|\sigma|) \b( |\rho|^{3} \log(1 + |\rho|) + |\sigma|^{3} \log(1 + |\sigma|) \b) \nonumber \\ & \qquad + C|\ell|^{-3} \sum_{\rho \in \Lhom_{*}} \wf_{4}(|\rho|) |\rho|^{3} \log(1 + |\rho|) \nonumber \\ & \quad \leq C |\ell|^{-3} \b( \| \wf_{3} \|_{\Hw_{3}^{\log}}^{2} + \| \wf_{4} \|_{\Hw_{4}^{\log}} \b) \leq C |\ell|^{-3},
	\end{align}
	\fn{where we have applied Lemma \ref{lemma:omega-sum-int-est}.}
	For the remaining term, we have from \eqref{ac_def_Lemma3.1} that
	${\D^2\up = \D^2\ulin + O(|x|^{-3})}$, which together with the fact
	${\sum_{\rho,\sigma \in \Lhom_{*}} V_{,\rho\sigma}(\bfO)\D_{\rho}\D_{\sigma} \ulin(x) = 0}$
	and \eqref{proof-decay-fl-1} leads to
	\begin{eqnarray}\label{proof-decay-fl-2}
	|f^{(1)}(\ell)| \leq C|\ell|^{-3}.
	\end{eqnarray}

	To estimate $f^{(2)}(\ell)$ in \eqref{eq:appprf:disl:res1-def_f},
	we decompose $f^{(2)}$ into two parts
	\begin{multline}\label{proof:decay_f_ell_ab}
	f^{(2)}(\ell) = \sum_{\rho \in \Lhom_{*},|\rho|\leq\frac{|\ell|}{2}} \int_0^1 (1-t) \widetilde{D}_{-\rho}
	\< \ddel V_{,\rho}\b(te(\ell)\b) e(\ell), e(\ell) \> \dt
	\\
	+ \sum_{\rho \in \Lhom_{*},|\rho|\geq\frac{|\ell|}{2}} \int_0^1 (1-t) \widetilde{D}_{-\rho}
	\< \ddel V_{,\rho}\b(te(\ell)\b) e(\ell), e(\ell) \> \dt
	~=:~f^{(2)}_A(\ell) + f^{(2)}_B(\ell) .
	\end{multline}

	We first express
	\begin{align*}
	&\widetilde{D}_{-\rho} \b \< \ddel V_{,\rho}\b(t_{1}e(\ell)\b) e(\ell), e(\ell) \b \> = D_{-\rho} \b \< \ddel V_{,\rho}\b(t_{1}e(\ell)\b) e(\ell), e(\ell) \b \> \\ & \quad = \b \< \ddel V_{,\rho}\b(t_{1}e(\ell - \rho)\b) e(\ell - \rho), e(\ell - \rho) \b \> - \b \< \ddel V_{,\rho}\b(t_{1}e(\ell)\b) e(\ell), e(\ell) \b \> \\
	& \quad = \b \< \ddel V_{,\rho}\b(t_{1}e(\ell)\b) ( e(\ell-\rho) + e(\ell) ), D_{-\rho}e(\ell) \b \> + \b \< D_{-\rho} \ddel  V_{,\rho}\b(t_{1}e(\ell)\b) e(\ell), e(\ell) \b \> \\
	& \quad = \b \< \ddel V_{,\rho}\b(t_{1}e(\ell)\b) ( e(\ell-\rho) + e(\ell) ), D_{-\rho}e(\ell) \b \> \\ & \qquad \quad - t_{1} \int_{0}^{1} \b \< \delta^{3} V_{,\rho}\b(t_{1}e(\ell - t_{2}\rho)\b) \nabla_{\rho} e(\ell - t_{2}\rho), e(\ell), e(\ell) \b \> \dt_{2},
	\end{align*}
	hence using \eqref{decay_elastic_strain} and \asSEL~together with \eqref{ac_def_Lemma3.1}
	to eventually deduce
	\begin{align}\label{proof:decay_f_ell_a}
	\nonumber
	& |f^{(2)}_A(\ell)|
	\leq C  \bigg(\int_{0}^{1}   \sum_{\substack{\rho,\sigma,\xi\in\Lhom_{*} \\|\rho|\leq\frac{|\ell|}{2}}} \b|  V_{,\rho\sigma\xi}(t_{1} e(\ell - \rho)) \b| \b| e_{\sigma}(\ell-\rho) + e_{\sigma}(\ell) \b| \b| D_{-\rho}e_{\xi}(\ell) \b| \dt_1
	\\ \nonumber
	& +  \int_{0}^{1}  \int_{0}^{1} \sum_{\substack{\rho,\sigma,\xi,\zeta\in\Lhom_{*} \\|\rho|\leq\frac{|\ell|}{2}}} \b| V_{,\rho\sigma\xi\zeta}(t_{1} e(\ell - t_2\rho)) \b| \b| \nabla_{\rho} e_{\sigma}(\ell - t_2\rho)) \b| \b| e_{\xi}(\ell) \b| \b| e_{\zeta}(\ell)  \b| \bigg) \dt_1\dt_2
	\\
	& \leq C|\ell|^{-3},
	\end{align}
	by considering the various partition cases in \eqref{eq:Vloc}.

	For the remaining case, recall that
	\begin{align*}
	&\widetilde{D}_{-\rho} \b \< \ddel V_{,\rho}\b(t_{1}e(\ell)\b) e(\ell), e(\ell) \b \> \\ & \quad = \b \< \ddel V_{,\rho}\b(t_{1}e(\ell - \rho)\b) e(\ell - \rho), e(\ell - \rho) \b \> - \b \< \ddel V_{,\rho}\b(t_{1}e(\ell)\b) e(\ell), e(\ell) \b \>,
	\end{align*}
	then using \asSEL, there exist $\wf_{3} \in \Hw^{\log}_{3}, \wf_{4} \in \Hw^{\log}_{4}, \wf_{5} \in \Hw^{\log}_{5}$ such that
	\begin{align}
	&\sum_{\substack{\rho\in\Lhom_{*} \\ |\rho| \geq |\ell|/2}} \b | \b \< \ddel V_{,\rho}\b(t_{1}e(\ell)\b) e(\ell), e(\ell) \b \> \b| \leq C |\ell|^{-3} \sum_{\rho\in\Lhom_{*}} \wf_{5}(|\rho|) |\rho|^{3} \log^{2}(1 + |\rho|) \nonumber \\ & \quad + C |\ell|^{-3} \Bigg( \bigg( \sum_{\rho\in\Lhom_{*}} \wf_{3}(|\rho|) |\rho| \log(1 + |\rho|)  \bigg)^{3} + \bigg( \sum_{\rho\in\Lhom_{*}} \wf_{4}(|\rho|) |\rho|^{2} \log^{2}(1 + |\rho|)  \bigg)^{3/2} \Bigg) \nonumber \\
	& \quad \leq C |\ell|^{-3} \Big(\| \wf_{3} \|_{\Hw_{3}^{\log}}^{3} + \| \wf_{4} \|_{\Hw_{4}^{\log}}^{3/2} + \| \wf_{5} \|_{\Hw_{5}^{\log}} \Big) \leq C |\ell|^{-3},
	\end{align}
	\fn{using Lemma \ref{lemma:omega-sum-int-est}.}
	For $|\rho| \geq |\ell|/2$ and $\sigma \in \Lhom_{*}$, we have from \eqref{lemma-decay-u0} that
	\begin{align*}
	|e_{\sigma}(\ell - \rho)| &= |D_{\sigma} u_{0}(\ell - \rho)| \leq |u_{0}(\ell - \rho + \sigma)| + |u_{0}(\ell - \rho)| \\ &\leq C \b( \log(2 + |\ell - \rho|) + |\sigma| \, \b) \leq C\b( \log(1 + |\rho|) + |\sigma| \, \b),
	\end{align*}
	hence we deduce \fn{from Lemma \ref{lemma:omega-sum-int-est} that}
	\begin{align}
	&\sum_{\substack{\rho\in\Lhom_{*} \\ |\rho| \geq |\ell|/2}} \b | \b \< \ddel V_{,\rho}\b(t_{1}e(\ell - \rho)\b) e(\ell - \rho), e(\ell - \rho) \b \> \b| \leq C |\ell|^{-3} \sum_{\rho\in\Lhom_{*}} \wf_{5}(|\rho|) |\rho|^{5} \nonumber \\ & \quad + C |\ell|^{-3} \Bigg( \bigg( \sum_{\rho\in\Lhom_{*}} \wf_{3}(|\rho|) |\rho|^{3} \log^{2}(1 + |\rho|)  \bigg)^{3} + \bigg( \sum_{\rho\in\Lhom_{*}} \wf_{4}(|\rho|) |\rho|^{4} \log^{2}(1 + |\rho|)  \bigg)^{3/2} \Bigg) \nonumber \\
	& \quad \leq C |\ell|^{-3} \Big(\| \wf_{3} \|_{\Hw_{3}^{\log}}^{3} + \| \wf_{4} \|_{\Hw_{4}^{\log}}^{3/2} + \| \wf_{5} \|_{\Hw_{5}^{\log}} \Big) \leq C |\ell|^{-3}. \label{proof:decay_f_ell_b}
	\end{align}
	Collecting the estimates \eqref{proof:decay_f_ell_ab}--\eqref{proof:decay_f_ell_b} yields
	\begin{align}\label{proof:decay_f_ell_c}
	|f^{(2)}(\ell)| & \leq C|\ell|^{-3},
	\end{align}
	hence $|f(\ell)| \leq C|\ell|^{-3}$	whenever $\ell_{2} > \hat{x}_{2}$.

	Combining \eqref{proof-decay-fl-6}, \eqref{proof-decay-fl-6-}, \eqref{eq:appprf:disl:res1-def_f},
	\eqref{proof-decay-fl-2} and \eqref{proof:decay_f_ell_c} gives \eqref{decay-fl} when $\ell$ lies in the left half-plane and $\ell_2>\hat{x}_2$.
	The case $\ell_2<\hat{x}_2$ follows verbatim by rewriting \eqref{proof-decay-fl-6} as
	\begin{align*}
	f(\ell) &= \sum_{\rho\in\Lhom_{*}}D_{-\rho}V_{,\rho}\b(e(\ell)\b)
	+ \bigg( \sum_{\substack{\rho\in\Lhom_{*} \nonumber \\ \ell-\rho + \burg_{12} \in\Omega_{\Gamma} \\ \ell_2-\rho_2<\hat{x}_2}}
	V_{,\rho}\b(e(\ell-\rho + \burg_{12})\b) - \sum_{\substack{\rho\in\Lhom_{*} \nonumber \\ \ell-\rho\in\Omega_{\Gamma} \\ \ell_2-\rho_2<\hat{x}_2}} V_{,\rho}\b(e(\ell-\rho)\b) \bigg).
	\end{align*}

	{\it Case 2: right half-space: }
	To treat the case $\ell_1 > \hat{x}_1$, we first observe that
	the definition of the reference solution with branch-cut
	$\Gamma = \{ (x_1, \hat{x}_2) ~|~x_1 \geq \hat{x}_1 \}$ was somewhat arbitrary,
	in that we could have equally chosen
	$\Gamma_S := \{ (x_1, \hat{x}_2) ~|~x_1 \leq \hat{x}_1 \}$.
	In this case the predictor solution $u_0$ would be replaced with $S_0  u_0$.
	Let the resulting energy functional be denoted by
	\begin{eqnarray*}
		\E_S(u) := \sum_{\ell\in\Lhom} V\b( DS_0 u_0(\ell)+Du(\ell)\b)
		- V\b( DS_0 u_0(\ell)\b).
	\end{eqnarray*}
	It is straightforward to see that if $\delta\E(\bar{u}) = 0$, then
	$\delta\E_S(S\bar{u}) = 0$ as well.

	With this observation, we can rewrite
	\begin{displaymath}
	f = \Del_{-\rho} V_{,\rho}(\Del_0 \up)
	= [R D_{-\rho} S] V_{,\rho}\b([R D S_0]\up\b)
	= R D_{-\rho} V_{,\rho}(D S_0\up).
	\end{displaymath}
	Since $S_0 \up$ is smooth in a neighbourhood of $|\ell|$ even if that
	neighbourhood crosses the branch-cut, we can now repeat the foregoing
	argument in {Case 1} to deduce again that
	$|Sf(\ell)| \leq C |\ell|^{-3}$ as well.  But since $S$ represents an $O(1)$
	shift, this immediately implies also that $|f(\ell)| \leq C |\ell|^{-3}$.
	This completes the proof of \eqref{decay-fl}.
\end{proof}

We also need the  following result to convert pointwise forces into divergence form.
Define $D_{\mathcal{N}}u(\ell):=\{D_{\rho}u(\ell)\}_{\rho\in\mathcal{N}(\ell)-\ell}$,
we have the following result inherited from \cite[Lemma 5.1 and Corollary 5.2]{2013-defects-v3}.

\begin{lemma}\label{lemma:divergence}
	Let $p > d$, $f : \mA\Z^d \to \R^{\ds}$ such that
	$|f(\ell)| \leq C_f |\ell|^{-p}$ for all $\ell \in \mA\Z^d$,
	and $\sum_{\ell \in \mA\Z^d} f(\ell) = 0$.
	Then there exists $g : \mA\Z^d \to (\R^{\ds})^{\mathcal{N(\ell)}-\ell}$
	and a constant $C>0$ depending on $p$ such that
	\begin{eqnarray*}
		\sum_{\ell\in\mA\Z^d} f(\ell) \cdot v(\ell) =	\sum_{\ell\in\mA\Z^d} \<g(\ell), D_{\mathcal{N}}v(\ell)\>
		\quad{\rm and}\quad
		|g(\ell)|_{\mathcal{N}} \leq C|\ell|^{-p+1} \quad \forall~\ell \in \mA\Z^d.
	\end{eqnarray*}
\end{lemma}

\begin{proof}
	This is an immediate consequence of  in \cite[Lemma 5.1]{2013-defects-v3}
	\fn{\eqref{ass:A:nn}, and} $\mathcal{N}(\ell) \supseteq \{\ell\pm Ae_i \fn{| 1 \leq i \leq 3}\}$\fn{, for all $\ell \in \L$}.
\end{proof}

\begin{proof}
	[Proof of Theorem \ref{corr:deltaE0-dislocation}]
	We have from Lemma \ref{lemma:decay_fl_dislocation} that $|f(\ell)|\leq C|\ell|^{-3}$,
	which together with Lemma \ref{lemma:divergence} yields a map $g$ such that
	$|g(\ell)|\leq C|\ell|^{-2}$ and
	\begin{eqnarray*}
		\b\<{F},v\b\> \leq C \|g\|_{\hc{l}^2}\|Dv\|_{\hc{l}^2}^2	\qquad\forall~v\in\UsH(\L).
	\end{eqnarray*}
	Therefore $F$ is  bounded and this completes the proof.
\end{proof}

\subsection{Proof of Theorem \ref{theorem:dislocation-decay}}
\label{sec:proof_decay_dislocation}

We will extend the arguments of Appendix \ref{sec:proof_decay_pd} to the dislocation case.
We shall modify the homogeneous difference operator \eqref{eq:def-Huv}
in the analysis for dislocations
by defining $\widetilde{H}:\UsH(\Lhom)\rightarrow\UsHd(\Lhom)$ as
\begin{eqnarray}\label{eq:def-Huv-tilde}
\b\< \widetilde{H}u,v \b\> :=
\sum_{\ell\in\Lhom}\b\< \delta^2 V({\pmb 0}) \widetilde{D}u(\ell), \widetilde{D}v(\ell)\b\>
\qquad \forall~v\in\UsH(\Lhom).
\end{eqnarray}
The following lemma is analogous to Lemma \ref{lemma-g-ubar-decay},
which gives the first-order residual estimate for dislocations.

\begin{lemma}
	\label{th:regdisl:residuals}
	If the conditions of Theorem~\ref{theorem:dislocation-decay} are satisfied,
	then for any $\bar{u}$ solving~\eqref{eq:1st_order_pd}, there exists
	$g:\Lhom\rightarrow (\R^{\ds})^{\Lhom_{*}}$ such that
	\begin{eqnarray*}
		\< \widetilde{H} \ua, v \> = \< g, \Del v \> \qquad \forall v \in \Usz,
	\end{eqnarray*}
	where $g$ satisfies
	\begin{align*}
	\left| \b\<g,\Del v\b\> \right| \leq C \sum_{\ell \in \Lhom} \sum_{j=1}^{3}
	\tilde{g}_{j}(\ell) |\Del v(\ell)|_{\wf_{j},j},
	\end{align*}
	with some constant $C>0$, $\wf_j\in\Hw_{j+2}^{\log}$ for $1\leq j\leq 3$ and
	\begin{align*}
	\tilde{g}_{1}(\ell) = \tilde{g}_{2}(\ell) = (1+|\ell|)^{-2} + \bigg( \sum_{k=1}^{2} |\Del \bar{u}(\ell)|_{\wf_{k},k} \bigg)^{2}, \quad
	\tilde{g}_{3}(\ell) = (1+|\ell|)^{-2} + |\Del \bar{u}(\ell)|_{\wf_{3},3} ^{2}.
	\end{align*}
\end{lemma}

\begin{proof}
	By the definition of $\widetilde{H}$, we can write
	\begin{eqnarray}\label{eq:prf:disl:tilH-residual-eqn}
	\nonumber
	\< \widetilde{H} \bar{u}, v \> &=&
	\sum_{\ell \in \Lhom} \B( \b\< (\ddel V(\bfO) - \ddel V(e(\ell)))\Del \bar{u}(\ell), \Del v(\ell) \b\>
	\\ \nonumber
	&&  + \b< \del V(e(\ell)) + \ddel V(e(\ell)) \Del \bar{u}(\ell) - \del V(e(\ell)+\Del \bar{u}(\ell)), \Del v(\ell) \b\>  \B) - \< \del\E(0), v \>
	\\
	&=:&  \<g^{(1)} + g^{(2)}, \Del v \> - \< f, v \>,
	\end{eqnarray}
	where we employed the force-displacement form \eqref{eq:disl:delE0-f} in the last step.

	To estimate $\tilde{g}^{(1)}$, we have from \eqref{decay_elastic_strain} that
	there exist $\wf_j\in\Hw_{j+2}^{\log}~(j=1,2,3)$ such that
	\begin{align}\label{proof-decay-dislocation-a}
	\nonumber
	&\quad
	\b|\<g^{(1)},\Del v\>\b|
	= \bigg|\sum_{ \ell\in\Lhom} \int_0^1 \Big\< \delta^3 V\big(te(\ell)\big) , \Big(e(\ell),  \Del_{\vsig}\bar{u}(\ell), \Del_{\xi} v(\ell)  \Big) \Big\> ~\dd t \bigg|
	\\
	&\leq \sum_{{\ell\in\Lhom}}\sum_{\rho,\vsig,\xi\in\Lhom_{*}} \int_0^1
	\Big| V_{,\rho\vsig\xi}\b(t e(\ell)\b) D_{\rho}e(\ell)\Del_{\vsig}\bar{u}(\ell)\Del_{\xi} v(\ell)  \Big| ~\dd t
	\\
	&\leq C \sum_{\ell\in\Lhom} (1+|\ell|)^{-1} \left(
	\bigg( \sum_{k=1}^{2} |\Del \bar{u}(\ell)|_{\wf_{k},k} \bigg)
	\bigg( \sum_{k'=1}^{2} |\Del v(\ell)|_{\wf_{k'},k'} \bigg)
	+ |\Del \bar{u}(\ell)|_{\wf_{3},3} |\Del v(\ell)|_{\wf_{3},3}
	\right)
	\end{align}
	where similar calculations as  \eqref{eq:useful2}--\eqref{eq:g-res-quad-est}
	and the decay assumptions \asSEL~for dislocations are used for the last estimate.

	Again using the same calculations as \eqref{eq:useful2}--\eqref{eq:g-res-quad-est},
	we deduce
	\begin{align}\label{proof-decay-dislocation-b}
	\b|\<g^{(2)},\Del v\>\b|
	\leq C \sum_{\ell\in\L} \left(
	\bigg( \sum_{k=1}^{2} |\Del \bar{u}(\ell)|_{\wf_{k},k} \bigg)^{2}
	\bigg( \sum_{k'=1}^{2} |\Del v(\ell)|_{\wf_{k'},k'} \bigg)
	+ |\Del \bar{u}(\ell)|_{\wf_{3},3}^{2} |\Del v(\ell)|_{\wf_{3},3}
	\right).
	\end{align}

	Combining \eqref{proof-decay-dislocation-a} and \eqref{proof-decay-dislocation-b},
	we obtain that
	\begin{align}
	|\<g^{(1)} + g^{(2)}, \Del v \>|
	&\leq C \sum_{\ell\in\Lhom} \bigg( (1+|\ell|)^{-2} + \Big(\sum_{k=1}^{2} |\Del \bar{u}(\ell)|_{\wf_{k},k} \Big)^2 \bigg)
	\bigg( \sum_{k'=1}^{2} |\Del v(\ell)|_{\wf_{k'},k'} \bigg) \nonumber
	\\
	& \qquad \quad + C \sum_{\ell\in\Lhom} \Big( (1+|\ell|)^{-2} + |\Del \bar{u}(\ell)|^2_{\wf_3,3} \Big)
	|\Del v(\ell)|_{\wf_3,3}.
	\label{proof-decay-dislocation-c}
	\end{align}

	For the $\<f, v\>$ group, we have from Lemma \ref{lemma:decay_fl_dislocation}
	that $|f(\ell)| \leq C |\ell|^{-3}$, which also implies $|Sf(\ell)| \leq C |\ell|^{-3}$.
	Then using a similar argument as Lemma \ref{lemma:divergence},
	we can derive existence of $g^{(3)}$, such that
	\begin{eqnarray}\label{proof-decay-dislocation-d}
	|\< f, v \>| = |\< g^{(3)}, \Del v \>| \leq \sum_{\ell\in\Lhom}
	\b(1+|\ell|\b)^{-2} |\Del v(\ell)|_{\wf_1,1}.
	\end{eqnarray}
	Setting $g := g^{(1)} + g^{(2)} - g^{(3)}$, 	\eqref{proof-decay-dislocation-c}--\eqref{proof-decay-dislocation-d} imply the desired result.
\end{proof}

Unlike the point defect case, we can not use Lemma \ref{lemma:decay-hom-eq}
directly, due to the ``incompatible finite-difference stencils'' $\Del u(\ell)$
in \eqref{eq:def-Huv-tilde}.  We will follow a boot-strapping argument
in \cite{2013-defects-v3} to bypass this obstacle,
starting from the following sub-optimal estimate.

\begin{lemma}
	\label{th:regdisl:subopt}
	Under the conditions of Theorem \ref{theorem:dislocation-decay},
	there exists $C > 0$ such that
	\begin{eqnarray}\label{decay-tilde-suboptimal}
	|\Del_{\rho}  \ua(\ell)| \leq C |\rho| \big(1+|\ell|\big)^{-1} .
	\end{eqnarray}
\end{lemma}

\begin{proof}
	Let us first consider the case when $\ell$ lies in the left half-plane $\ell_1<\hat{x}_1$.
	It is only necessary for us to consider when $|\ell|$ is sufficiently large,
	when $B_{3|\ell|/4}(\ell)$ does not intersect the branch cut $\Gamma$.

	Let $\eta$ be a cut-off function with $\eta(x) = 1$ in
	$B_{|\ell|/4}(\ell)$, $\eta(x) = 0$ in $\R^2 \setminus B_{|\ell|/2}(\ell)$
	and $|\nabla \eta| \leq C |\ell|^{-1}$.
	Further, let $v(k) := D_\rho \Gr(k - \ell)$, where $\Gr$ is the lattice Green's
	function associated with the homogeneous finite-difference operator $H$
	in \eqref{eq:Huv-hrho}, see Lemma \ref{lemma:decay-Green} for the definition of $\Gr$.
	Then,
	\begin{eqnarray}\label{eq:regdisl:prf:20}
	\nonumber
	D_\rho \bar{u}(\ell) = \< H \bar{u}, v \>
	&=& \Big( \< H\bar{u}, \eta v\> - \< \widetilde{H}\bar{u}, \eta v\> \Big)
	+  \b\< \widetilde{H}\bar{u}, \eta v\b\> + \b\< H \bar{u}, (1-\eta)v \b\>
	\\
	&:=& h^{(1)} + h^{(2)} + h^{(3)},
	\end{eqnarray}
	where $\eta v, (1-\eta) v$ are understood as pointwise function multiplications.

	We first compare $H$ and $\widetilde{H}$ to estimate $h^{(1)}$.
	Using Lemma \ref{lemma:decay-Green} (ii), we can derive from the
	definitions of $\widetilde{D}$ and $\eta$  that
	$| \Del_\sigma [\eta v](k) | \leq C|\sigma|\cdot|\rho|(1+|\ell-k|)^{-2}$.
	Moreover, we have from the definition of $\eta$ that if $|k-\ell| > 3|\ell|/4$, then $D_{\sigma}[\eta v](k)\neq 0$
	only when $|\sigma|>|\ell|/4$.
	Using the definition of $\widetilde{D}$ and the fact $B_{\frac{3|\ell|}{4}}(\ell)\cap\Gamma = \emptyset$,
	we have that if $|k-\ell|\leq 3|\ell|/4$, then $\Del_{\xi}u(k)\neq D_{\xi}u(k)$ only when $|\xi|\geq|\ell|/4$.
	Then we can obtain from the decay assumptions \asSEL~for dislocations that
	\begin{align}\label{diff-H-tildeH}
	\nonumber
	& \quad h^{(1)} =
	\b\< H\bar{u}, \eta v\b\> - \b\< \widetilde{H}\bar{u}, \eta v\b\>
	\\ \nonumber
	&=\sum_{k\in\Lhom\cap\Omega_{\Gamma}}\sum_{\sigma,\xi\in\Lhom_{*}} V_{,\sigma\xi}(0)
	\bigg( D_{\sigma}[\eta v](k) \Big( D_{\xi}\bar{u}(k) - \Del_{\xi}\bar{u}(k) \Big)
	+ \widetilde{D}_{\xi}u(k) \Big( D_{\sigma}[\eta v](k) - \Del_{\sigma}[\eta v](k) \Big) \bigg)
	\\ 
	&\leq C|\rho| |\ell|^{-2} \bigg( \sum_{\substack{k\in\Lhom\cap\Omega_{\Gamma} \\ |k-\ell|>\frac{3|\ell|}{4}}} \!\!\!\! (1+|k-\ell|)^{-2}
	+ \!\!\!\!\!\! \sum_{\substack{k\in\Lhom\cap\Omega_{\Gamma} \\ |k-\ell|\leq\frac{3|\ell|}{4}}} \!\!\!\! \b( 1+|k-\ell| \b)^{-2} \bigg)
	\leq C|\rho| |\ell|^{-1} .
	\end{align}

	Using Lemma \ref{lemma:decay-Green}, \ref{th:regdisl:residuals}
	and same calculations as \eqref{proof:res_b}--\eqref{proof:resb2},
	we have that there exist $\wf_{j}\in\Hw_{j+2}^{\log}~(j=1,2,3)$ such that
	\begin{align}\label{proof-decay-dislocation-h2}
	\nonumber
	h^{(2)} &=
	\< \widetilde{H} u, \eta v \> = \< g, \Del \eta v \>
	\\  \nonumber
	& \leq C|\rho| \sum_{k \in \Lhom} \bigg( \b(1+|k|\b)^{-2} +
	\sum_{j=1}^3 |\Del u(k)|_{\wf_{j},j}^2\bigg)
	\b( 1+|\ell-k| \b)^{-2}
	\\
	&\leq C|\rho| \bigg( |\ell|^{-2}\log|\ell|
	+ \sum_{k \in \Lhom} (1+|\ell-k|)^{-2} \bigg(\sum_{j=1}^3 |\Del u(k)|_{\wf_{j},j}^2 \bigg) \bigg) .
	\end{align}

	To estimate the last group in \eqref{eq:regdisl:prf:20}, we note that
	the definition of $\eta$ and Lemma \ref{lemma:decay-Green} implies
	$| D_\sigma [(1-\eta) v](k) | \leq C|\sigma|\cdot|\rho|(1+|\ell-k|)^{-2}$.
	Moreover, if $|k-\ell|\leq|\ell|/8$, then $D_\sigma [(1-\eta) v](k)\neq 0$ only when $|\sigma|>|\ell|/8$.
	Hence, we have from the decay assumptions \asSEL~for dislocations that
	there exist $\wf_{j}\in\Hw_{j+2}^{\log}~(j=1,2)$ such that
	\begin{align}\label{proof-decay-dislocation-h3}
	\nonumber
	&\quad h^{(3)} = \b\< H \bar{u},  (1-\eta) v \b\>
	\\ \nonumber
	&= \sum_{\substack{k\in\Lhom \\  |k-\ell|\leq \frac{|\ell|}{8}}} \sum_{\rho,\vsig\in\Lhom_{*}} V_{,\rho\vsig}(0) D_{\rho}\bar{u}(k) D_{\vsig}(1-\eta)v(k)
	+ \sum_{\substack{k\in\Lhom \\  |k-\ell|>\frac{|\ell|}{8}}} \sum_{\rho,\vsig\in\Lhom_{*}} V_{,\rho\vsig}(0) D_{\rho}\bar{u}(k) D_{\vsig}(1-\eta)v(k)
	\\ \nonumber
	&\leq C|\rho| \Bigg( |\ell|^{-2} \sum_{|k-\ell| \leq \frac{|\ell|}{8}} (1+ |\ell-k|)^{-2}
	+ \sum_{|k-\ell| >\frac{|\ell|}{8}} |\ell-k|^{-2} \bigg( \sum_{j=1}^2|D\bar{u}(k)|_{\wf_{j},j} \bigg) \Bigg)
	\\
	& \leq C|\rho| |\ell|^{-1}
	\end{align}

	Combining \eqref{eq:regdisl:prf:20}--\eqref{proof-decay-dislocation-h3}
	and the definition \eqref{eq:wnormdef},
	we have that if $\ell$ lies in the left half-plane and $|\ell|$ sufficiently large, then
	there exist $\wf_{j}\in\Hw_{j+2}^{\log}~(j=1,2,3)$ such that
	\begin{eqnarray}\label{eq:regdisl:prf:30}
	\sum_{k=1}^3\b| D u(\ell) \b|_{\wf_k,k} \leq C \Bigg( |\ell|^{-1}
	+ \sum_{k \in \Lhom} (1+|\ell-k|)^{-2}
	\bigg(\sum_{j=1}^3 |\Del u(k)|_{\wf_{j},j}^2 \bigg) \Bigg).
	\end{eqnarray}
	We can extend this estimate to the case when $\ell$ lies in the right
	half-plane by a {\it reflection argument} as used in the proof of Lemma \ref{lemma:decay_fl_dislocation}.
	We replace the branch-cut $\Gamma$ by
	$\Gamma_S := \{ (x_1, \hat{x}_2) ~|~x_1 \leq \hat{x}_1 \}$,
	the energy functional $\E(u)$ by $\E_S(u)$,
	and the predictor solution $u_0$ by $S_0  u_0$.
	The new problem is now $\delta\E(\bar{u}) = 0$, which is structurally
	identical to $\delta\E(\bar{u})=0$.
	Therefore, it follows that \eqref{eq:regdisl:prf:30} holds,
	but with $u$ replaced by $Su$ and for all $\ell_1 > \hat{x}_1$.

	We can now consider arbitrary $\ell$ and follow the proof of Lemma \ref{lemma:decay-hom-eq}.
	By taking
	\begin{eqnarray*}
		w(r) := \sup_{\ell\in\Lhom,~|\ell|\geq r}\bigg(\sum_{k=1}^{3}|Du(\ell)|_{\wf_k,k}\bigg)
	\end{eqnarray*}
	and repeating the argument used in \cite[Lemma 6.3]{2013-defects-v3} verbatim,
	we deduce that $\omega(r)\leq Cr^{-1}$.
	Inserting this into
	\eqref{diff-H-tildeH}--\eqref{proof-decay-dislocation-h3}
	gives us $|D_{\rho} u(\ell)| \leq C|\rho| |\ell|^{-1}$, which together with
	the definition of $\Del$ leads to $|\Del_{\rho} u(\ell)| \leq C|\rho||\ell|^{-1}$.
\end{proof}

Having established a preliminary pointwise decay estimate on $\Del \ua$, we now
apply a boot-strapping technique to obtain an optimal bound.

\begin{proof}
	[Proof of Theorem \ref{theorem:dislocation-decay}]
	Without loss of generality, we may assume
	$\ell$ belongs to the left half-plane, i.e. $\ell_1 < \hat{x}_1$.
	The right half-plane case can be obtained by the reflection arguments
	used in the proof of Lemma \ref{lemma:decay_fl_dislocation} and \ref{th:regdisl:subopt}.
	We again define $v$ as in the proof of Lemma \ref{th:regdisl:subopt}, and write
	\begin{align*}
	&\quad D_\rho \bar{u}(\ell) =  \< H \bar{u}, v \>
	\\
	&=\< \widetilde{H}\bar{u},v \>
	+ \sum_{k \in \Lhom\cap\Omega_{\Gamma}} \B( \< \delta^2 V(0) D\bar{u}(k), Dv(k) \>
	- \< \delta^2V(0) \Del \bar{u}(k), \Del v(k) \>  \B)
	\\
	&=: {\rm T}_1 + {\rm T}_2.
	\end{align*}

	To estimate the first group we note that ${\rm T}_1 = \< g, \Del v\>$, hence
	we can employ the residual estimates from Lemma \ref{th:regdisl:residuals}.
	Combining the estimates in Lemma \ref{lemma:decay-Green},
	\ref{th:regdisl:residuals} and \ref{th:regdisl:subopt} yields
	\begin{align}\label{proof-decay-dislacation-j}
	\b| {\rm T}_1 \b|  \leq \sum_{k \in \L} (1+|k|)^{-2} (1+|\ell-k|)^{-2}
	\leq C |\ell|^{-2} \log |\ell|.
	\end{align}

	To estimate ${\rm T}_2$, we will combine the result in Lemma \ref{th:regdisl:subopt}
	and a similar calculation as \eqref{diff-H-tildeH}:
	\begin{align}\label{diff-H-tildeH-T2}
	\nonumber
	& \quad T_2=
	\b\< H\bar{u}, v\b\> - \b\< \widetilde{H}\bar{u}, v\b\>
	\\ \nonumber
	&=\sum_{k\in\Lhom\cap\Omega_{\Gamma}}\sum_{\sigma,\xi\in\Lhom_{*}} V_{,\sigma\xi}(0)
	\bigg( D_{\sigma} v(k) \Big( D_{\xi}\bar{u}(k) - \Del_{\xi}\bar{u}(k) \Big)
	+ \widetilde{D}_{\xi}u(k) \Big( D_{\sigma} v(k) - \Del_{\sigma} v(k) \Big) \bigg)
	\\ \nonumber
	&\leq  C|\rho| \bigg( \sum_{\substack{k\in\Lhom\cap\Omega_{\Gamma} \\
			\Gamma \cap	B_{\frac{|k|}{2}}(k) =\emptyset} } (1+|k|)^{-2}(1+|k-\ell|)^{-2}
	+ \sum_{\substack{ k\in\Lhom\cap\Omega_{\Gamma} \\ \Gamma \cap B_{\frac{|k|}{2}}(k) \neq \emptyset} }
	(1+|k|)^{-1} \b( |k|+|\ell| \b)^{-2}   \bigg)
	\\
	&\leq C|\rho| |\ell|^{-2} \log |\ell| ,
	\end{align}
	where we have used the facts that if $\Gamma \cap B_{|k|/2}(k) =\emptyset$,
	then $D_{\sigma}u(k)\neq\Del_{\sigma}u(k)$ only when $|\sigma|>|k|/2$,
	and if $\Gamma \cap B_{|k|/2}(k) \neq \emptyset$, then
	$|\ell-k|\geq \frac{1}{2}(|\ell|+|k|)$.
	Collecting \eqref{proof-decay-dislacation-j}--\eqref{diff-H-tildeH-T2}
	completes the proof.
\end{proof}


\begin{thebibliography}{00}

\bibitem{Balluffi}
R. Balluffi,
{\it Introduction to Elasticity Theory for Crystal Defects}
(Cambridge University Press, 2012).

\bibitem{CaiBulChaLiYip2003}
W. Cai, V. Bulatov, J. Chang, J. Li and S. Yip,
Periodic image effects in dislocation modelling,
{\it Philosophical Magazine} {\bf 83} (2003) 539--567.

\bibitem{CancesLeBris:preprint}
E. Canc\`{e}s and C. Le Bris,
Mathematical modelling of point defects in materials science,
{\it Math. Models Methods Appl. Sci.}  {\bf 23} (2013), 1795--1895.

\bibitem{CancesDeleurenceLewin:2008}
E. Canc\`{e}s, A. Deleurence and M. Lewin,
A new approach to the modelling of local defects in crystals: The reduced Hartree-Fock case,
{\it Comm. Math. Phys.} {\bf 281} (2008) 129--177.

\bibitem{CattoLeBrisLions}
I. Catto, C. Le Bris and P. Lions,
{\it The Mathematical Theory of Thermodynamic Limits: Thomas-Fermi Type Models}
(Oxford University Press, 1998).

\bibitem{chen16b}
H. Chen, J. Kermode, F. Nazar and C. Ortner,
Locality of the reduced Hartree-Fock model with Yukawa poential,
in preparation.

\bibitem{chen16a}
H. Chen, J. Lu and C. Ortner,
Thermodynamic limits of crystal defects with finite temperature tight binding,
{\it Arch. Ration. Mech. Anal.} {\bf 230} (2018), 701--733.


\bibitem{chen15a}
H. Chen and C. Ortner,
QM/MM methods for crystalline defects. part 1: Locality of the tight binding model,
{\it Multiscale Model. Simul.} {\bf 14} (2016) 232--264.

\bibitem{chen15b}
H. Chen and C. Ortner,
QM/MM methods for crystalline defects. part 2: Consistent energy and force-mixing,
{\it Multiscale Model. Simul.} {\bf 15} (2017) 184--214.

\bibitem{csanyi05}
G. Cs\'{a}nyi, T. Albaret, G. Moras, M. Payne and A.~D. Vita,
Multiscale hybrid simulation methods for material systems,
{\it J. Phys.: Condens. Matter} {\bf 17} (2005) R691--R703.

\bibitem{Daw:1984a}
M. Daw and M. Baskes,
Embedded-atom method: derivation and application to impurities, surfaces, and other defects in metals,
{\it Phys. Rev. B} {\bf 29} (1984) 6443--6453.

\bibitem{2013-defects-v3}
V. Ehrlacher, C. Ortner and A. Shapeev,
Analysis of boundary conditions for crystal defect atomistic simulations,
{\it Arch. Ration. Mech. Anal.} {\bf 222} (2016) 1217--1268.

\bibitem{Bell34}
E.T. Bell
Exponential polynomials,
{\it Annals of Mathematics} {\bf 35} (1934) 258--277.

\bibitem{Finnis1984}
M. Finnis,
A simple empirical $n$-body potential for transition metals,
{\it Phil. Mag. A.} {\bf 50} (1984) 45--55.

\bibitem{finnis03}
M. Finnis,
{\it Interatomic Forces in Condensed Matter}
(Oxford University Press, 2003).

\bibitem{goringe97}
C. Goringe, D. Bowler and E. Hern\'{a}ndez,
Tight-binding modelling of materials,
{\it Rep. Prog. Phys.} {\bf 60} (1997) 1447--1512.

\bibitem{Gupta1981}
R. Gupta,
Lattice relaxation at a metal surface,
{\it Phys. Rev. B} {\bf 23} (1981) 6265--6270.

\bibitem{Hine2009}
N. Hine, K. Frensch, W. Foulkes and M. Finnis,
Supercell size scaling of density functional theory formation energies of charged defects,
{\it Phys. Rev. B} {\bf 79} (2009) 024112.

\bibitem{Hudson:stab}
T. Hudson and C. Ortner,
On the stability of Bravais lattices and their Cauchy--Born approximations,
{\it ESAIM:M2AN} {\bf 46} (2012) 81--110.

\bibitem{HudsonOrtner:disloc}
T. Hudson and C. Ortner,
Existence and stability of a screw dislocation under anti-plane deformation,
{\it Arch. Ration. Mech. Anal.} {\bf 213} (2014) 887--929.

\bibitem{LennardJones:1924a}
J. Jones,
On the determination of molecular fields. III. From crystal measurements and kinetic theory data,
{\it Proc. Roy. Soc. London A.} {\bf 106} (1924) 709--718.

\bibitem{2013-PRE-bqcfcomp}
X. Li, M. Luskin, C. Ortner and A. Shapeev,
Theory-based benchmarking of the blended force-based quasicontinuum method,
{\it Comput. Methods Appl. Mech. Engrg.} {\bf 268} (2014) 763--781.

\bibitem{LiOrtnerShapeevEtAl-blended}
X. Li, C. Ortner, A. Shapeev and B. Van~Koten,
Analysis of blended atomistic/continuum hybrid methods,
{\it Numer. Math.} {\bf 134} (2016) 275--326.

\bibitem{MakovPayne1995}
G. Makov and M. Payne,
Periodic boundary conditions in ab initio calculations,
{\it Phys. Rev. B} {\bf 51} (1995) 4014--4022.

\bibitem{Morrey}
C.B. Morrey,
{\it Multiple Integrals in the Calculus of Variations}
(Springer, 1966).

\bibitem{Morse:1929a}
P. Morse,
Diatomic molecules according to the wave mechanics. II. Vibrational levels,
{\it Phys. Rev.} {\bf 34} (1929) 57--64.

\bibitem{Nazar-thesis}
F. Nazar,
{\it Electronic structure of defects in the Thomas--Fermi--von Weizs\"{a}cker model of crystals},
PhD thesis (University of Warwick, 2016).

\bibitem{nazar14}
F. Nazar and C. Ortner,
Locality of the Thomas-Fermi-von Weizs\"{a}cker equations,
{\it Arch. Ration. Mech. Anal.} {\bf 224} (2017) 817--870.

\bibitem{2016-multipt}
D. Olson and C. Ortner,
Regularity and locality of point defects in multilattices,
{\it Appl. Math. Res. Express} {\bf 2017} (2016) 297--337.

\bibitem{ortner_shapeev12}
C. Ortner and A. Shapeev,
Interpolants of lattice functions for the analysis of atomistic/continuum multiscale methods,
arXiv:1204.3705.

\bibitem{OrtnerTheil2012}
C. Ortner and F. Theil,
Justification of the Cauchy--Born approximation of elastodynamics,
{\it Arch. Ration. Mech. Anal.} {\bf 207} (2013) 1025--1073.

\bibitem{OrtnerZhang2014}
C. Ortner and L. Zhang,
Atomistic/continuum blending with ghost force correction,
{\it SIAM J. Sci. Comput.} {\bf 38}  (2016) A346--A375

\bibitem{Sinclair:1971}
J. Sinclair,
Improved atomistic model of a bcc dislocation core,
{\it J. Appl. Phys.} {\bf 42} (1971) 5321--5329.

\bibitem{solovej91}
J. Solovej,
Proof of the ionization conjecture in a reduced Hartree-Fock model,
{\it Invent. Math.} {\bf 104} (1991) 291--311.

\bibitem{HandbookMaterialsModelling}
S. Yip, Ed.
{\it Handbook of Materials Modelling}
(Springer, 2005).

\bibitem{yukawa35}
H. Yukawa,
On the interaction of elementary particles,
{\it Proc. Phys. Math. Soc. Japan.} {\bf 17} (1935) 48--57.

\end{thebibliography}
\end{document}